\documentclass[11pt]{amsart}
\usepackage{euscript}           % Nice script font.
\usepackage{amsmath,amsthm}     % Handy math stuff, theorem environments.
\usepackage{amssymb}            % Fancy math symbols.
\usepackage{verbatim} % to enable comment environment
\newcommand{\bs}{\blacksquare}

\renewcommand{\S}{{\mathcal S}}
\newcommand{\T}{{\mathcal T}}

\newcommand{\uu}{{\mathcal U}}

\newcommand{\bQ}{{\mathbf Q}}

\newcommand{\bF}{{\mathbf F}}
\newcommand{\bM}{{\mathbf M}}

\newcommand{\bfd}{{\mathbf d}}
\newcommand{\Aa} {{\mathcal  A}}
\newcommand{\Ar}{{\Aa_{\mathrm{rel}}^\Z}}
\newcommand{\Ag}{{\Aa_\ipFun}}
\newcommand{\Cell}{{\mathcal C}ell}
\newcommand{\codim}{{\mathrm{codim}}}
\newcommand{\supp}{\mathrm{supp}}
\newcommand{\cl}{\mathrm{cl}}
\newcommand{\id}{{\mathrm{id}}}

\newcommand{\ad}{{\mathrm{ad}}}
\newcommand{\pre}{{\mathrm{pre}}}

\newcommand{\Sig}{{\mathrm{Sig}}}
\newcommand{\sig}{{\mathrm{sig}}}

\newcommand{\Sigr}{{\mathrm{sig}_{\mathrm{rel}}}}

\newcommand{\STop}{{\mathrm{STop}}}

\newcommand{\STopFun}{{\mathrm{STopFun}}}

\newcommand{\sch}{\mathrm{sch}}
\newcommand{\Rel}{{\mathrm{Rel},\mathrm{sch}}}
\newcommand{\Relz}{{\mathrm{Rel},\mathrm{sch},\geq 0}}
\newcommand{\Kan}{\mathrm{Kan}}

\DeclareMathOperator*{\colim}{colim}

\newcommand{\ip} {{\operatorname{IP}}}
\newcommand{\ipFun} {{\operatorname{IPFun}}}
\newcommand{\Fun} {{\operatorname{Fun}}}
\newcommand{\CC} {{\mathcal C}}

\newcommand{\pt} {{\operatorname{pt}}}
\newcommand{\DD} {{\mathcal  D}}
\newcommand{\M} {{\mathcal  M}}
\newcommand{\N} {{\mathcal  N}}
\renewcommand{\P} {{\mathcal  P}}

\newcommand {\syml} {\operatorname{L}^\bullet}
\newcommand {\intg}{{\mathbb Z}}
\newcommand{\SO} {{\operatorname{SO}}}

\newcommand {\R}{{\mathbb R}}
\newcommand {\RR}{{\mathcal R}}
\newcommand {\ZZ}{{\mathcal Z}}
\newcommand {\rat}{{\mathbb Q}}
\newcommand {\real}{{\mathbb R}}
\newcommand {\cplx}{{\mathbb C}}

\newcommand {\smlhf} {\ensuremath{\mbox{$\frac{1}{2}$}}}
\newcommand{\td}[1]{\tilde{#1}}
\newcommand{\St}{{\mathrm{St}}}
\newcommand {\BL}{{\mathbb L}} 
\newcommand{\si} {{\operatorname{Si}}}

\newcommand {\Z}{{\mathbb Z}}

\newcommand{\ra}{\rightarrow}                   % right arrow
\newcommand{\lra}{\longrightarrow}              % long right arrow
\DeclareMathOperator*{\hocolim}{hocolim}

\newcommand{\cpnt}{\mathbb{CP}^{n/2}}
\newfont{\german}       {eufm10 at 12pt}
\DeclareMathOperator{\Hom}{Hom} 

\numberwithin{equation}{section}
\newtheorem{thm}{Theorem}[section]
\newcounter{numerierer}
\newcounter{leer}

\newtheorem{defn}[thm]{Definition}
\newtheorem{prop}[thm]{Proposition}

\newtheorem{lemma}[thm]{Lemma}

\theoremstyle{definition}  % Bold headings and Roman body text.
\newenvironment{definition}{\begin{defn}\rm}{\end{defn}}

\newtheorem{example}[thm]{Example}
\newtheorem{convention}[thm]{Convention}

\newtheorem{notation}[thm]{Notation}

\newtheorem{remark}[thm]{Remark}
\usepackage[frame,matrix,curve, arrow]{xy}
\xymatrixcolsep{1.9pc}                          % Adjust size of diagrams.
\xymatrixrowsep{1.9pc}
\newdir{ >}{{}*!/-5pt/\dir{>}}                  % Make better tailed arrows.
\newdir{ |}{{}*!/-5pt/\dir{|}}                  % Make better mapsto arrows.


\begin{document}

\title[L-Fundamental Class for IP-Spaces and the Novikov Conjecture]{The L-Homology Fundamental Class for IP-Spaces and the Stratified Novikov Conjecture}
\author{Markus Banagl} 
\address{Mathematisches Institut, Universit\"at Heidelberg, INF 288, D-69120
Heidelberg, Germany}
\email{banagl@mathi.uni-heidelberg.de}
\author{Gerd Laures}
\address{Fakult\"at f\"ur Mathematik,  Ruhr-Universit\"at Bochum, NA1/66,
  D-44780 Bochum, Germany}
\email{gerd@laures.de}
\author{James E.\ McClure}
\address{Department of Mathematics, Purdue University, 150 N.\ 
University Street, West Lafayette, IN 47907-2067, USA}
\email{mcclure@math.purdue.edu} 
\thanks{
The first author was supported in part by a research grant of the
Deutsche Forschungsgemeinschaft.
The third author  was partially supported by a grant
from the Simons Foundation (\#279092 to James McClure).
He thanks the Lord for making his work possible.}

\begin{abstract}
An IP-space is a pseudomanifold whose defining local properties 
imply that its middle perversity global intersection homology groups satisfy 
Poincar\'e duality integrally.
We show that the symmetric signature induces a map of 
Quinn spectra from IP bordism to the symmetric $L$-spectrum
of $\Z$, which is, up to weak equivalence, an $E_\infty$ ring
map.
Using this map, we construct a fundamental $L$-homology class for IP-spaces, and as 
a consequence we prove the stratified Novikov conjecture for IP-spaces whose fundamental
group satisfies the Novikov conjecture.
\end{abstract}

\maketitle
\section{Introduction}

An intersection homology Poincar\'e space, or
IP-space, is a piecewise linear pseudomanifold such that the middle dimensional,
lower middle perversity integral intersection homology of even-dimensional links
vanishes and the lower middle dimensional, lower middle perversity intersection homology
of odd-dimensional links is torsion free.
This class of spaces was introduced by Goresky and Siegel in
\cite{gorsie} as a natural solution, assuming the IP-space to be compact and oriented, 
to the question: For which class of spaces does
intersection homology (with middle perversity) satisfy Poincar\'e duality over the integers?

If $X$ is a compact oriented IP-space whose dimension $n$ is a multiple of $4$,
then the signature $\sigma (X)$ of $X$ is the signature of the intersection form
\[ IH_{n/2} (X;\intg)/\operatorname{Tors} \times
  IH_{n/2} (X;\intg)/\operatorname{Tors} \longrightarrow \intg, \]
where $IH_*$ denotes intersection homology with the lower middle perversity,
\cite{gm1}, \cite{gm2}.
This signature is a bordism invariant for bordisms of IP-spaces.
The IP-space bordism groups have been investigated by Pardon in
\cite{pardon}, where it is shown that the signature (when $n=4k$) together with the
de Rham invariant (when $n=4k+1$) form a complete system of invariants.\\

Next we recall the theory of the $L$-homology fundamental class for manifolds.
Let $M^n$ be a closed oriented $n$-dimensional manifold.
The symmetric signature $\sigma^* (M)$ of \cite{mishchenko},\cite[p. 47]{ranicki} is an element of
the symmetric $L$-group $L^n (\intg [G]),$ where 
$G$ denotes the fundamental group $\pi_1M$.
It is a non-simply-connected generalization of the
signature $\sigma (M)$, since for $n=4k$ the canonical homomorphism
$L^n (\intg [G])\to L^n (\intg)=\intg$ maps $\sigma^* (M)$ to $\sigma (M)$.
Moreover, $\sigma^*$ is homotopy invariant and bordism invariant for bordisms
over the classifying space $BG$.
Let $\BL^\bullet = \BL^\bullet \langle 0 \rangle (\intg)$ denote the symmetric $L$-spectrum
with homotopy groups $\pi_n (\BL^\bullet)=L^n (\intg)$
and let $\syml_* (-)$ denote the homology theory determined by $\BL^\bullet$.
For an $n$-dimensional Poincar\'e space $M$ which is either a topological
manifold or a combinatorial homology manifold (i.e. a polyhedron
whose links of simplices are homology spheres), Ranicki defines
a canonical $\syml$-homology fundamental class $[M]_{\mathbb{L}} \in
\syml_n (M),$ see \cite[Prop. 16.16]{ranicki}. Its image under the assembly map
\[ \syml_n (M) \stackrel{\alpha}{\longrightarrow} L^n (\intg [G]) \]
is the symmetric signature $\sigma^\ast (M)$.
The class $[M]_{\mathbb{L}}$ is a topological
invariant, but, unlike the symmetric signature, not a homotopy invariant in general.
The geometric meaning of the $\syml$-homology fundamental class
is that its existence for a geometric Poincar\'e complex $X^n,$ 
$n\geq 5,$ assembling to the symmetric signature (which in fact any
Poincar\'e complex possesses), implies up to 2-torsion that $X$
is homotopy equivalent to a compact topological manifold.
(More precisely, $X$ is homotopy equivalent to a compact manifold
if it has an $\syml$-homology fundamental class, which assembles to the so-called
visible symmetric signature of $X$.)
Smooth manifolds $M$ possess a Hirzebruch $L$-class in $H^* (M;\rat)$,
whose Poincar\'e dual we denote by $L(M)\in H_* (M;\rat)$.
Rationally, $[M]_{\mathbb{L}}$ is then given by $L(M)$,
\[ [M]_{\mathbb{L}} \otimes 1 = L (M) \in \syml_n (M)\otimes \rat
  \cong \bigoplus_{j\geq 0} H_{n-4j} (M;\rat). \]
Thus, we may view $[M]_{\mathbb{L}}$ as an integral refinement of the
$L$-class of $M$. The identity $\alpha [M]_{\mathbb{L}} = \sigma^\ast (M)$
may then be interpreted as a non-simply connected generalization of the
Hirzebruch signature formula. These facts show that the 
$\syml$-homology fundamental class is much more powerful than $\sigma^* (M)$.
For example, 
there exist infinitely many manifolds $M_i, i=1,2,\ldots,$ in the
homotopy type of $S^2 \times S^4,$ distinguished by
the first Pontrjagin class of their tangent bundle
$p_1 (TM_i)\in H^4 (S^2 \times S^4)\cong \intg,$ namely
$p_1 (TM_i) = Ki,$ $K$ a fixed nonzero integer.
On the other hand, $\sigma^* (M_i) = \sigma^* (S^2 \times S^4)=0
\in L^6 (\intg [\pi_1 (S^2 \times S^4)])=L^6 (\intg)=0$.\\

We return to singular spaces. A Witt space 
is a piecewise linear pseudomanifold such that the middle dimensional,
lower middle perversity rational intersection homology of even-dimensional links
vanishes, \cite{siegel}. The symmetric signature $\sigma^* (X)\in L^n (\rat [G])$ 
and the $\syml$-homology fundamental class $[X]_{\mathbb{L}} \in (\syml (\rat))_n (X)$
of an oriented Witt space $X^n$ 
appeared first in the work of Cappell, Shaneson
and Weinberger, see \cite{csw} and \cite{Wein}, though a detailed construction
is not provided there. Regarding the relation between $\syml (\intg)$ and $\syml (\rat)$,
the functor $\operatorname{ad}_{\operatorname{sym}},$ described in \cite{LM},
from the category of rings with involution to the category of ad theories associates
to the localization map $\intg \to \rat$ a morphism $\ad^\intg \to \ad^\rat$ of
the associated symmetric Poincar\'e ad theories, which in turn induces a map of
Quinn spectra $\syml (\intg)\to \syml (\rat)$. On homotopy groups, this fits into
Ranicki's localization sequence
\[ \cdots \longrightarrow L^n (\intg) \longrightarrow L^n (\rat) \longrightarrow
  L^n (\intg, \intg - \{ 0 \}) \longrightarrow L^{n-1} (\intg) \longrightarrow \cdots, \]
where
\[ L^n (\rat) \cong \begin{cases}
 \intg \oplus (\intg/_2)^\infty \oplus (\intg/_4)^\infty,& n\equiv 0 (4) \\
   0,& n\not\equiv 0 (4).
\end{cases} \]

In \cite{banagl-msri}, the first author outlined a construction of $[X]_{\mathbb{L}}$
for IP-spaces $X$
based on ideas of Eppelmann \cite{eppelmann}, and pointed out that 
the existence of this class
implies in particular a definition of a symmetric signature $\sigma^* (X)$
as the image of $[X]_{\mathbb{L}}$ under assembly.
In \cite{almp}, it is shown that this symmetric signature, adapted to Witt spaces
and pushed into $K_* (C^*_r G)$ via
\[ L^* (\rat [G])\longrightarrow L^* (C^*_r G)\longrightarrow K_* (C^*_r G), \]
agrees rationally with the Albin-Leichtnam-Mazzeo-Piazza signature index class.
The first fully detailed construction of $\sigma^* (X)$ for Witt spaces $X$ has been
provided in \cite{friedmanmcclure}. That construction is closely parallel to the
original construction of Mi\v{s}\v{c}enko, but using singular intersection chains
on the universal cover instead of ordinary chains.
The methods of \cite{friedmanmcclure} carry over to
IP-spaces and yield a symmetric signature over $\intg$ for such spaces, as we show
in Section \ref{symsig}.\\

In the present paper, we give the first detailed construction of
an $\syml$-homology fundamental class $[X]_{\mathbb{L}} \in \syml_n (X)$ for 
IP-spaces $X$.
While Eppelmann used complexes of sheaves,
we are able to use the, for our purposes, more precise and geometric methods of 
\cite{friedmanmcclure}.
The main issue is to construct a map (at least in the derived category) on the spectrum level from IP bordism
to $\BL^\bullet$, for then $[X]_{\mathbb{L}}$ can readily be defined as the image
of the identity map $[\id_X]\in (\Omega_\ip)_n (X)$ under
$(\Omega_\ip)_n (X)\to \syml_n (X)$, see Definition \ref{def.xipxl}.
To obtain this map of spectra, we rely heavily on the technology
of ad theories and their associated Quinn spectra as developed by
the second and third author in \cite{LM}, \cite{LM2}.
Roughly, we construct first an ad theory of IP spaces, which automatically
gives an associated Quinn spectrum $\bQ_\ip$, whose homotopy groups
are Pardon's IP bordism groups.
Using the symmetric signature,
we define a morphism of ad theories from the IP ad theory to the
ad theory of symmetric algebraic Poincar\'e complexes over $\intg$.
The spectrum of the latter ad theory is the symmetric
$L$-spectrum $\BL^\bullet$. The morphism of ad theories then induces the
desired map of spectra.
We prove that our $\syml$-homology fundamental class has all the expected properties
(Theorem \ref{t3}): It is an oriented PL homeomorphism invariant,
its image under assembly is the symmetric signature and it agrees
with Ranicki's $\syml$-homology fundamental class when $X$ is a PL
manifold.\footnote{If $R$ is a PID then our methods would also give a
fundamental class in $ (\syml (R))_n(X)$
with analogous properties.}\\

As an application of our $\syml$-homology fundamental class, we discuss the
stratified homotopy invariance of the higher signatures of IP-spaces.
Let $X$ be an $n$-dimensional compact oriented IP-space, whose fundamental
group 
$G$ satisfies the strong Novikov conjecture, that is, the
assembly map
\[ \syml_n (BG) \longrightarrow L^n (\intg [G]) \]
is rationally injective. Then we prove that the stratified Novikov conjecture holds for $X$, 
i.e. the higher signatures
\[ \langle a, r_* L(X) \rangle,~ a\in H^* (BG;\rat), \]
where $r:X\to BG$ is a classifying map for the universal cover of $X$
and $L(X)\in H_* (X;\rat)$ is the Goresky-MacPherson $L$-class of $X$,
are stratified homotopy invariants, see Theorem \ref{thm.novikov}. 
The stratified Novikov conjecture has been treated from the analytic
viewpoint in \cite{ALMPnovikov}. \\

Here is an outline of the paper.  Sections \ref{sec.ipspacessymmcplx} and
\ref{s3} review the basic facts about IP-spaces and ad theories.  Section
\ref{s4} constructs an ad theory associated to IP-spaces.  Section \ref{s5}
reviews two (equivalent) ad theories associated to symmetric Poincar\'e
complexes.  Section \ref{s6} uses the symmetric signature (ignoring the
fundamental group)
to construct a map Sig of Quinn spectra from IP bordism to the symmetric 
$L$-spectrum of $\Z$.  Section \ref{ringoids} recalls the definition of
$L^n({\mathbb{Z}}[\pi_1 X])$, where $\pi_1 X$ denotes the fundamental groupoid
of $X$. Section \ref{symsig} constructs the symmetric signature 
of an IP-space $X$ as an element of 
 $L^n({\mathbb{Z}}[\pi_1 X])$.  
% Section \ref{gs6} gives a generalization of Section \ref{s6} which is used for
% one of the proofs in Section \ref{s8}.
Section
\ref{s8} constructs the $\syml$-theory fundamental class of an IP-space, and
Section \ref{gs6} shows 
that it
assembles to the symmetric signature constructed in Section \ref{symsig}.
Section \ref{s9} uses the results of Sections \ref{s8} and \ref{gs6} to prove the stratified 
Novikov conjecture
for IP-spaces.  Section \ref{mc} shows that the map Sig constructed in Section
\ref{s6} is, up to weak equivalence, an $E_\infty$ ring map; this is applied in
Section \ref{sfm} to prove that $[X\times Y]_{\mathbb{L}}
=
[X]_{\mathbb{L}}
\times
[Y]_{\mathbb{L}}$.
Section \ref{pt2}
proves a result needed for Section \ref{s8}, namely the fact that
the assembly map for IP bordism is a weak equivalence.  There are six
appendices.  Appendix \ref{a1} 
reviews the basic facts about the intrinsic filtration of a PL space.
Appendix \ref{add} gives background about
modules over additive categories which is needed for Sections
\ref{symsig}--\ref{gs6}.  Appendices \ref{subdiv},  
\ref{tech} and \ref{univ} give generalizations of some technical results from 
\cite{friedmanmcclure} which are
needed in Sections \ref{symsig} and \ref{gs6}.
Appendix \ref{am} proves a mutliplicative
property of the assembly map which may be of independent interest.  Appendix
\ref{aa1} corrects some signs in \cite{LM}.
\\

\begin{remark}
\label{not}
Here is a comparison between our notation and that in \cite{ranicki}.
The spectrum $\mathbb L^\bullet(\Z[\pi_1(X,x)])$ of  \cite[Definition 
14.1]{ranicki} is identical with the spectrum we denote by
$\mathbf{Q}^{\Z[\pi_1(X,x)]}$ (which was defined in \cite[Sections 9 and
15]{LM}).  The spectrum $\mathbb L^\bullet\langle 0\rangle(\Z[\pi_1(X,x)])$ 
of \cite[pages 151-152]{ranicki} is canonically weakly equivalent to the
spectrum we denote by $\mathbf{Q}^{\Z[\pi_1(X,x)]}_{\geq 0}$ (see Section
\ref{conn} below).  We remind the reader that in \cite{ranicki} the spectrum
$\mathbb L^\bullet\langle 0\rangle(\Z)$ is denoted by $\mathbb L$ (\cite[page
173]{ranicki}) and $L^n(\Z[\pi_1(X,x)])$ is defined to be $\pi_n$ of the
connective spectrum
$\mathbb L^\bullet\langle 0\rangle(\Z[\pi_1(X,x)])$ 
(\cite[page 60]{ranicki}).
\end{remark}

\paragraph{Acknowledgements.}
We would like to thank Matthias Kreck, Wolfgang L\"uck, Shmuel Weinberger and 
(especially) Greg Friedman for their help. We are also grateful to the referee for a very careful revision.

\section{Review of IP bordism}
\label{sec.ipspacessymmcplx}

We use the term {\it polyhedron} as defined in \cite[Definition 
1.1]{ro-sa.pltop}. 

\begin{definition}
\label{pseud}
An {\it $n$-dimensional PL pseudomanifold}
is a polyhedron $X$ for which some
(and hence every) triangulation has the following
properties.

\quad (a) Every simplex is contained in an $n$-simplex.

\quad (b) Every $(n-1)$-simplex is a face of exactly two $n$-simplices.
\end{definition}

\begin{definition}
\label{bpseud}
An {\it $n$-dimensional PL $\partial$-pseudomanifold} is a 
polyhedron $X$ with 
the property that some (and hence every) triangulation has the following
properties.

\quad (a) Every simplex is contained in an $n$-simplex.

\quad (b) Every $(n-1)$-simplex is a face of either one or two $n$-simplices;
the union of the $(n-1)$-simplices which are faces of one $n$-simplex is called
the {\it boundary} of $X$ and denoted $\partial X$.

\quad (c) The boundary $\partial X$ is an $(n-1)$-dimensional pseudomanifold.

\quad (d) The boundary is collared, that is, there is a PL embedding $\partial
X\times [0,1)\to X$ with open image which is the identity on $\partial X$.
\end{definition}

\begin{remark}
(i) The subspace $\partial X$ is independent of the triangulation.

(ii) The collaring condition is needed in order for Lefschetz duality to hold in
intersection homology (see \cite[Section 7.3]{FM}).
\end{remark}

\begin{definition}
An {\it orientation} of an $n$-dimensional PL pseudomanifold or PL 
$\partial$-pseudomanifold is a set of orientations of the
$n$-simplices of some triangulation such that the sum of the $n$-simplices with
these orientations
is a cycle (a relative cycle in the case of a $\partial$-pseudomanifold).
\end{definition}

For some purposes we need a stratification.
For a polyhedron $Y$, let
$c^\circ Y$ denote the open cone $([0,1)\times Y)/(0\times Y)$.  We recall the
inductive definition of stratified pseudomanifold:

\begin{definition}
\label{d1}
A {\it $0$-dimensional stratified PL pseudomanifold} $X$ is a  discrete set of 
points with the trivial filtration $X=X^0\supseteq X^{-1}=\emptyset$.
An {\it $n$-dimensional stratified  PL pseudomanifold}
$X$ is a polyhedron together
with a filtration by closed polyhedra
\begin{equation*}
X=X^n\supseteq X^{n-1} = X^{n-2}\supseteq \cdots \supseteq X^0\supseteq
X^{-1}=\emptyset
\end{equation*}
such that

\quad (a) $X-X^{n-1}$ is dense in $X$, and

\quad (b) for each point $x\in X^i-X^{i-1}$, there exists a neighborhood
$U$ of $x$ for which there is a  compact $n-i-1$ dimensional
stratified  PL pseudomanifold  $L$ and a PL homeomorphism
\begin{equation*}
\phi: \R^i\times c^\circ L\to U
\end{equation*}
that takes $\R^i\times c^\circ (L^{j-1})$ onto $X^{i+j}\cap U$. 
\end{definition}

The space $L$ in part (b) is determined up to PL homeomorphism by $x$ and the
stratification
(\cite[Lemma 2.5.18]{Fr}); it is called the {\it link} of $X$ at $x$ and 
denoted $L_x$.
Since the cone on the empty set $L^{-1}$ is a point, taking $j=0$ in (b)
shows that $X^i - X^{i-1}$ is a manifold for every $i$.
A PL pseudomanifold always possesses a stratification in the sense of
Definition \ref{d1}, by Proposition \ref{p1}(iv).
Conversely, if $X$ is an $n$-dimensional stratified PL pseudomanifold, then for any triangulation,
every simplex is contained in an $n$-simplex and
every $(n-1)$-simplex is a face of exactly two $n$-simplices, that is,
the underlying polyhedron of $X$ is indeed a PL pseudomanifold in the sense
of Definition \ref{pseud}.

\begin{definition} 
\label{d2}
An $n$-dimensional \emph{stratified PL $\partial$-pseudomanifold} is a 
PL $\partial$-pseudomanifold $X$
together with a
filtration by closed polyhedra such that

\quad (a) $X-\partial X$, with the induced filtration, is an
$n$-dimensional stratified
PL pseudomanifold,

\quad (b) $\partial X$, with the induced filtration, is an $n-1$ dimensional 
stratified PL pseudomanifold, and

\quad (c) there is a neighborhood $N$ of $\partial X$
with a homeomorphism of filtered spaces $N\to
\partial X\times [0,1)$ (where $[0,1)$ is given the trivial filtration) which
is the identity on $\partial X$.

\end{definition}

A PL $\partial$-pseudomanifold always possesses a stratification in the sense of
Definition \ref{d2}, by Proposition \ref{p2}.
Next recall the definition of intersection homology (\cite{gm1}, \cite{gm2}, 
\cite{borel}, \cite{kirwanwoolf}, \cite{banagl-tiss}).  We will denote the 
lower middle perversity, as usual, by $\bar{m}$.

\begin{definition} \label{def.ipspace}
(\cite{gorsie,pardon})
An $n$-dimensional \emph{IP-space} is an $n$-dimensional PL pseudomanifold
$X$ for which
some stratification has the following properties.

\quad (a) $IH^{\bar{m}}_l (L_x;\Z)=0$ for all $x\in X^{n-2l-1} - X^{n-2l-2}$, 
and

\quad (b) $IH^{\bar{m}}_{l-1} (L_x;\Z)$ is torsion free for all $x\in 
X^{n-2l} - X^{n-2l-1}$.
\end{definition}

\begin{remark}
\label{r6}
(i) IP stands for ``intersection homology Poincar\'e''.

(ii) Note that the stratification is not considered as part of the structure of
an IP-space.

(iii) If conditions (a) and (b) hold for some stratification then they hold for
every stratification (by the Proposition in \cite[Section 2.4]{gm2}).
\end{remark}

\begin{definition} 
(\cite{pardon})
\label{rev1}
An $n$-dimensional  \emph{$\partial$-IP-space} is an $n$-dimensional 
PL $\partial$-pseudomanifold $X$ for which
$X-\partial X$ is an IP-space.
\end{definition}

\begin{prop}
\label{p3}
If $X$ is a $\partial$-IP-space then $\partial X$ is an IP-space.
\end{prop}

\begin{proof}
Give $X$ the stratification of Proposition \ref{p2}.
By Remark \ref{r6}(iii),
the restriction of this stratification to $X-\partial X$ has properties (a) and
(b) of Definition \ref{def.ipspace}.  
Part (c) of Definition \ref{d2}
implies that the links of $\partial X$
are also links of $X-\partial X$, so $\partial X$ satisfies
Definition \ref{def.ipspace}.
\end{proof}

Next we consider IP bordism groups.  There are two ways to define them:
\begin{enumerate}
\item
The objects and
bordisms are the compact oriented IP-spaces and $\partial$-IP-spaces. 
\item
An object is a compact oriented IP-space with a given stratification, and
similarly for the bordisms.
\end{enumerate}
Pardon \cite{pardon} does not make it clear which definition he is using, but
fortunately the two definitions give the same bordism groups by \cite{f}.
We will use the first definition.

\section{Review of $\ad$ theories}
\label{s3}

We recall some definitions from \cite[Sections 2 and 
3]{LM}.  
A \emph{category with involution} is a category together with an endofunctor
$i$ which satisfies $i^2 = 1$.
The set of integers $\intg$ is a poset and therefore a category. We give
it the trivial involution.
A \emph{$\intg$-graded category} is a category $\Aa$ with involution together with
involution-preserving functors $d: \Aa \to \intg$ (called the \emph{dimension function}) and
$\varnothing: \intg \to \Aa$ such that $d\varnothing$ is equal to the identity functor. A \emph{$k$-morphism} between
$\intg$-graded categories is a functor which decreases the dimensions of objects by $k$ and
strictly commutes with $\varnothing$ and $i$.
We will write $\varnothing_n$ for $\varnothing (n)$.
Note that the existence of $d$ implies that when $d(A) > d(B)$, there are no morphisms
$A \to B$.

Let $K$ be a finite collection of PL balls in some $\real^n$, and write
$|K|$ for the union $\bigcup_{\sigma \in K} \sigma$. We say that $K$ is a \emph{ball complex}
if the interiors of the balls of $K$ are disjoint and the boundary of each ball of $K$ is a union of balls of
$K$ (thus the interiors of the balls of $K$ give $|K|$ the structure of a regular CW
complex). The balls of $K$ will also be called \emph{closed cells} of $K$.
A \emph{subcomplex} of a ball complex $K$ is a subset of $K$ which is a ball complex.
A \emph{morphism} of ball complexes is the composite of an isomorphism with an
inclusion of a subcomplex.
For a ball
complex $K$ and a
subcomplex $L$ we define
$\CC ell(K,L)$ to be the category in which the objects are the oriented 
closed cells 
of $K$ which are not in $L$, together with an empty cell $\emptyset_n$ for 
each dimension $n$, and the non-identity morphisms are 
given by inclusions of cells (with no requirement on the orientations).
The  category $\CC ell(K,L)$ is $\Z$-graded (\cite[Definition 3.3]{LM}), that 
is, it comes with an involution $i$ (which reverses the orientation), a 
dimension functor $d$ into the poset $\Z$ and a section 
functor given by $\emptyset_n$. 

Given a $\Z$-graded category $\Aa$, a {\it pre $(K,L)$-ad of degree $k$} is a
functor
$\CC ell(K,L)\to \Aa$ of $\intg$-graded categories which decreases dimensions by $k$.  The set of these is
denoted $\pre^k(K,L)$.
Note that a pre $(K,L)$-ad defines a pre $K=(K,\emptyset )$-ad by precomposition with the functor from $\CC ell(K)$ to $\CC ell(K,L)$ which sends the $n$-dimensional cells of $L$  to the empty cell $\emptyset_n$.

An {\em ad theory} with values in $\Aa$ (called the
{\it target category} of the ad theory) consists of a subset
$\ad^k(K,L)\subset \pre^k(K,L)$ for each $(K,L)$ and each $k$, satisfying 
certain axioms
(\cite[Definition 3.10]{LM}).  One of the axioms says that an element of
$\pre^k(K,L)$ is in $\ad^k(K,L)$ if and only if its image in $\pre^k(K)$ is in
$\ad^k(K)$, so to describe an ad theory it suffices to specify the sets
$\ad^k(K)$.

An ad theory gives rise to bordism groups $\Omega_*$ (\cite[Section 4]{LM}),
a spectrum $\mathbf Q$ (\cite{quinn}, \cite[Section
15]{LM}), and a weakly equivalent symmetric spectrum $\mathbf M$
(\cite[Section 17]{LM}). 

A {\it morphism} of ad theories
is a functor of target categories which takes ads to ads.
A morphism $\ad_1 \to \ad_2$ of ad theories induces maps
$\mathbf Q_1 \to \mathbf Q_2$ and $\mathbf M_1 \to \mathbf M_2$ 
of the associated spectra.

\begin{remark}
\label{aa41}
Later we will need to know that there is a canonical isomorphism
$
\pi_j \bQ\cong \Omega_j
$
for all $j$.
This is a consequence of
\cite[Proposition 16.4(i), Remark 14.2(i), and Definitions 4.1 and
4.2]{LM}; for the convenience of the reader we describe the isomorphism more
explicitly.  By definition, $\pi_j \bQ$ is $\colim_n \pi_{n+j} Q_n$, where the
$Q_n$ are the spaces of the spectrum, and the maps $\pi_{n+j} Q_n\to 
\pi_{n+1+j} Q_{n+1}$ are induced by the suspension maps (\cite[page 44]{LM}
and Appendix \ref{aa1})
and are isomorphisms because $\bQ$ is an $\Omega$ spectrum (\cite[Proposition
15.9]{LM}).  
It therefore suffices to give 
compatible maps
\begin{equation}
\label{bo}
\pi_{n+j} Q_{n}\to \Omega_j
\end{equation}
for all $n$.  Next recall that $Q_n$ is the geometric realization 
of a based semisimplicial set $P_n$ (\cite[Definition 15.4]{LM}).  If $s$ is an
$(n+j)$ simplex of $P_n$ which has all its faces at the basepoint then the
induced map $|s/\partial s|\to Q_n$ represents an element $\bar{s}$ of 
$\pi_{n+j} Q_n$, and since $P_n$ is a Kan complex (\cite[Lemma 15.12]{LM}) 
all elements of $\pi_{n+j} Q_n$ are obtained in this way (\cite[Remark 
6.5]{RS}).  Also, by \cite[Definitions 15.4 and 3.10(e)]{LM}, there is an 
element $s_0\in \ad^{-j}(*)$ which corresponds to $s$ under the map induced by the $(n+j)$-isomorphism from $\CC ell(\Delta^{n+j}, \partial \Delta^{n+j})$ to $\CC ell(*)$.  Now we define the map
\eqref{bo} by letting $\bar{s}$ go to the bordism class of $s_0$ 
(\cite[Definitions 4.1 and 4.2]{LM}.
\cite[Proposition 16.4(i), Remark
14.2(i) and page 44]{LM} show that this is a well-defined isomorphism which is
compatible with the suspension maps as $n$ varies.
\end{remark}

\subsection{The ad theory of oriented topological manifolds}
As motivation for the ad theory of IP-spaces, we briefly recall the ad theory
$\ad_\STop$ (\cite[Example 3.5 and Section 6]{LM}; also see \cite[Section
2]{LM2}).
The target category $\Aa_{\STop}$ has as objects the compact oriented
topological manifolds with boundary which are subsets of some $\R^n$. The morphisms between objects of the same
dimension are the orientation-preserving homeomorphisms, and the other 
morphisms are the inclusions with image in the boundary.

To describe the set $\ad^k_\STop(K)$ we need to recall two definitions from 
\cite[Section 5]{LM}.  A $\Z$-graded category $\Aa$  is called  
{\it balanced} if it comes with a natural  involutive bijection for all objects $A,B$ of different dimensions
$$\eta :\Aa(A,B)\ra \Aa(A,i(B))$$
which commutes with the involution  $i$; examples are $\CC ell(K)$ and
$\Aa_{\STop}$.  Functors between balanced 
categories are called {\it balanced} if they 
commute with $\eta$.

Let $\CC ell^\flat (K)$ be the category whose objects are the (unoriented) 
cells of $K$ and whose morphisms are the inclusions.  Let $\Aa^\flat_{\STop}$
be the category whose objects are compact orientable topological manifolds,
whose morphisms between objects of the same dimension are homeomorphisms, and
whose other morphisms are  
the inclusions with image in the boundary.
A balanced functor
$$F: \CC ell(K) \ra \Aa_{\STop}$$ 
induces a functor
$$F^\flat:  \CC ell^\flat (K)\lra \Aa^\flat_{\STop}.$$ 
We define $\ad_\STop^k(K)\subset \pre_\STop^k(K)$ to be the set of 
functors $F$ with the following properties.

\quad (a)
$F$ is balanced.

\quad (b)
If $(\sigma',o')$ and $(\sigma,o)$ are oriented cells with
$\dim\sigma'=\dim\sigma-1$, and if
the incidence number $[o,o']$ is equal
to $(-1)^k$, 
then the map
\[
F(\sigma',o')\to\partial F(\sigma,o)
\]
is orientation preserving.

\quad (c)
For each $\sigma$, $\partial F^\flat(\sigma)$ is the colimit in Top of
$ F^\flat|_{\Cell^\flat(\partial \sigma)}.$

\medskip
It is shown in \cite[Appendix B]{LM} and \cite[Section 11]{LM2} that
the spectrum ${\mathbf Q}_\STop$ (resp., the symmetric spectrum ${\mathbf 
M}_\STop)$ obtained from this ad theory is weakly equivalent to the usual 
Thom spectrum MSTop (resp., as a symmetric spectrum). 

\begin{remark} 
There are many variations of classical bordism spectra which can be represented by Quinn spectra as well. For example, one can consult  \cite{baas-laures} for the case of bordism with singularities of Baas-Sullivan type.
\end{remark}

\section{The $\ad$ theory of IP-Spaces} 
\label{s4}

Recall Proposition \ref{p2}.

\begin{definition}
\label{d3}
Let $X$ and $X'$ be PL $\partial$-pseudomanifolds of dimensions $n,n'$. 
A {\it strong embedding} $f:X\to X'$ is a PL embedding for which
$X[n-i]=f^{-1}((X')[n'-i])$ for $0\leq
i\leq n$.
\end{definition}

Let $\Aa_\ip$ be the $\Z$-graded category whose objects are compact 
oriented $\partial$-IP-spaces which are subsets of some $\R^n$, whose morphisms between objects of the same dimension are
the orientation-preserving PL homeomorphisms and whose other morphisms
are the strong embeddings with image in the boundary. 
The involution $i$ reverses the orientation.  Then $\Aa_\ip$ is a balanced
$\Z$-graded category.
(The requirement that the morphisms between objects of different dimensions 
are strong embeddings will not actually be used until the proof of Lemma
\ref{l1}(ii)).
Before defining $\ad^k_\ip(K)$, we need a fact about PL topology which will be
proved at the end of this section.

\begin{lemma} \label{lem.pl2}
Let $\mathcal{P}$ be the category whose objects are compact polyhedra and 
whose morphisms are PL embeddings. Let $K$ be a ball complex and
$G: \CC ell^\flat (K)\to \mathcal{P}$ 
a covariant functor such that, for every $\sigma$, the map
\[
\colim_{\tau\in\partial \sigma} G(\tau)\to G(\sigma)
\]  
is a monomorphism.
Then for every subcomplex $L$ of $K$

{\rm (i)} the space
$\colim_{\sigma \in L} G (\sigma)$ has a PL structure for which 
the maps $G (\sigma)\to \colim_{\sigma \in L} G (\sigma)$ for $\sigma \in L$ 
are PL embeddings, and

{\rm (ii)} the map
$$\colim_{\sigma \in L} G (\sigma)\to \colim_{\sigma \in K} G (\sigma)$$ 
is a PL embedding.

{\rm (iii)}
Suppose that $\upsilon,\upsilon'$ are cells of $K$, that $x$ is a point of 
$G(\upsilon)$, that $y$ is a point of $G(\upsilon')$ which is not in the 
image of $\colim_{\tau\in\partial \upsilon'} G(\tau)\to G(\upsilon')$, and that
$x$ and $y$ map to the same point of $\colim_{\sigma \in K} G (\sigma)$. Then 
$\upsilon'\subset \upsilon$.  
\end{lemma}

Now let $\Aa^\flat_\ip$ be the category whose objects are compact orientable
IP-spaces with boundary, whose morphisms between objects of the same dimension
are PL homeomorphisms, and whose other morphisms are the strong embeddings with 
image
in the boundary.
A balanced functor
\[
F: \CC ell (K) \longrightarrow \Aa_\ip
\]
induces a functor
\[ F^\flat: \CC ell^\flat (K) \longrightarrow \Aa^\flat_\ip. \] 

\begin{definition} \label{def.ipad}
Let $K$ be a ball complex.
Define $\ad_\ip^k(K)\subset \pre_\ip^k(K)$ to be the set of functors $F$ with 
the following properties:

\quad (a) $F$ is balanced.

\quad (b) If $(\sigma', o')$ and $(\sigma, o)$ are oriented cells with $\dim 
\sigma' = \dim \sigma -1$, and if the incidence number $[o, o']$ is equal to 
$(-1)^k$, then the map
\[  F(\sigma', o')\longrightarrow \partial F(\sigma, o) \]
is orientation preserving. 

\quad (c) For each $\sigma \in K$, the map
\[
\colim_{\tau \in \partial \sigma} F^\flat (\tau)
\to
\partial F^\flat(\sigma)
\]
is a bijection.
\end{definition}

\begin{thm}
\label{t1}
$\ad_\ip$ is an ad theory.
\end{thm}

\begin{remark}
\label{r2}
By 
\cite[Lemma 2.11.7]{Fr},
the Cartesian product of $\partial$-IP-spaces is a $\partial$-IP-space, and the
product of an element of $\ad_\ip^k(K)$ with an element of $\ad_\ip^l(L)$ is an
element of $\ad_\ip^{k+l}(K\times L)$.  Thus $\ad_\ip$ is a multiplicative ad
theory (\cite[Definition 18.4]{LM}) and the associated symmetric spectrum
$\bM_\ip$ is a symmetric ring spectrum (\cite[Theorem 18.5]{LM}).

Moreover, $\ad_\ip$ is a commutative ad theory (\cite[Definition 3.3]{LM2}), so
by Theorem 1.1 of
\cite{LM2} there is a commutative symmetric ring spectrum
$\bM^{\mathrm{comm}}_\ip$
which is weakly equivalent as a symmetric ring spectrum to $\bM_\ip$.
Specifically, there is a symmetric ring spectrum $\mathbf A$ and ring maps
\[
\bM_\ip \leftarrow {\mathbf A} \rightarrow \bM^{\mathrm{comm}}_\ip
\]
which are weak equivalences.
\end{remark}

\begin{proof}[Proof of Theorem \ref{t1}]
The only parts of \cite[Definition 3.10]{LM} which are not obvious are (f) (the
gluing axiom) and (g) (the cylinder axiom).  

For part (g),
let $F$ be a $K$-ad; we need to define $J(F): \Cell(K\times I)\to
\Aa_\ip$.  First note that the statement of part (g) specifies what $J(F)$ has
to be on the subcategories $\Cell(K\times 0)$ and $\Cell(K\times 1)$.  The
remaining objects have the form
$(\sigma \times I,o\times o')$ and we define $J(F)$ for such an object to be
$F(\sigma,o)\times (I,o')$, where $(I,o')$ denotes the PL $\partial$-manifold
$I$ with orientation $o'$. $F(\sigma,o)\times I$ is a
$\partial$-IP-space because the link
at a point $(x,t)$ is the link in $F(\sigma,o)$ at $x$. The
inclusions of $F(\sigma,o)\times \{0\}$ and $F(\sigma,o)\times \{1\}$ in 
$F(\sigma,o)\times I$ are
strong embeddings (see Definition \ref{d3}) by the definition of the
stratification in Proposition \ref{p2}.

For part (f), let $K$ be a ball complex and $K'$ a subdivision of $K$.  Let $F$
be a $K'$-ad. We need
to show that there is a $K$-ad $E$ which agrees with $F$ on each residual 
subcomplex of $K$. We may assume (by induction over the lowest dimensional
cell of $K$ that is not a cell of $K'$) that $|K|$ is a PL
$n$-ball, that $K$ has exactly one $n$-cell, and
that $K'$ is a subdivision of $K$ which agrees with $K$ on the boundary of
$|K|$.   
Let $L$ be the subcomplex of $K'$ consisting of cells in the boundary
of $|K|$.  

Now let $\tau$ denote the $n$-cell of $K$, and choose an orientation $o$ of
$\tau$.
To specify the $K$-ad $E$, we only need to define $E(\tau,o)$. 

Let $X$ denote $\colim_{\sigma \in K'}
F^\flat (\sigma)$ and give $X$ the PL structure provided by Lemma 
\ref{lem.pl2}(i). 
By Lemma \ref{nn1} below, 
$X$ is a PL $\partial$-pseudomanifold. 

We will write $[F^\flat(\sigma)]$ for the image of $F^\flat(\sigma)$ in $X$; by
Lemma \ref{lem.pl2}(i) this is a PL $\partial$-pseudomanifold.

We claim that $X$ is a $\partial$-IP-space. By Definition \ref{rev1},
we need to show that $X-\partial X$ is an IP-space.  
We give $X-\partial 
X$ the intrinsic stratification (see Proposition \ref{p1}(iv)).
Let $x\in
X-\partial X$.  By Lemma \ref{lem.pl2}(iii) there is a unique $\sigma\in 
K'-L$ for which $x$ is in the
interior of $[F^\flat(\sigma)]$; give $[F^\flat(\sigma)]$ the intrinsic
stratification and let
$U$ be a distinguished neighborhood of $x$
in $[F^\flat(\sigma)]$.  The proof of 
\cite[Proposition 6.6]{LM} shows
that $x$ has a neighborhood $V$ in $X$ such that
there is a PL homeomorphism
\[
f:V\to U\times A,
\]
where $A$ is a Euclidean space.  The
filtration of $V$ inherited from $X$ is the same as the intrinsic
stratification of $V$ by Proposition \ref{p1}(i), and (by Proposition
\ref{p1}(ii) and (iii)) $f$ 
takes this 
filtration to the Cartesian product of the intrinsic stratification of $U$ with
the trivial stratification of $A$.  This implies that the link of $x$ in $X$ is
the same as the link of $x$ in $[F^\flat(\sigma)]$, and so conditions (a) and 
(b)
of Definition \ref{def.ipspace} are satisfied.

Now give the $n$-cells $\sigma$ of
$K'$ the orientations $o_\sigma$ which agree with $o$.  Then $X$ has an
orientation which agrees with the orientations of the $F(\sigma,o_\sigma)$, and
we define $E(\tau,o)$ to be $X$ with this orientation.  We claim that $E$ is a
pre $K$-ad with values in $\Aa_\ip$.  For this it only remains to show
that $[F^\flat(\sigma)]\to X$ satisfies Definition \ref{d3}
when $\sigma$ is a cell of $L$.
We denote the stratification on a PL pseudomanifold
(resp., PL $\partial$-pseudomanifold) $Y$ provided by Proposition \ref{p1}(iv)
(resp., Proposition \ref{p2}) by $Y^*$ (resp., $Y[*]$).
By its definition, the
filtration $X[*]$ agrees (up to a dimension shift) 
with $(\partial X)^*$, so it suffices to show that $(\partial X)^*$ agrees with
$[F^\flat(\sigma)][*]$. Next we observe that (by the proof of \cite[Proposition
6.6]{LM}) each point of $\partial [F^\flat(\sigma)]$ has a neighborhood $U$ in 
$\partial [F^\flat(\sigma)]$ and a neighborhood $V$ in $\partial X$ 
with 
$V\approx U\times A$, where $A$ is a Euclidean space.  Then $(\partial X)^*$ agrees with $V^*$ by
Proposition \ref{p1}(i), and (by Proposition \ref{p1}(ii) and (iii)) the latter agrees 
with $U^*\times A$ (where $A$ is given the trivial filtration). 
This implies that 
$(\partial X)^*$ agrees with the restriction of 
$[F^\flat(\sigma)][*]$ to $\partial [F^\flat(\sigma)]$. Moreover, $(\partial 
X)^*$ also agrees with the restriction of $[F^\flat(\sigma)][*]$ to 
$[F^\flat(\sigma)]-\partial [F^\flat(\sigma)]$ by 
Proposition \ref{p1}(i), so the two filtrations agree on all of
$[F^\flat(\sigma)]$.

Finally, we need to check that $E$ satisfies Definition \ref{def.ipad}. Parts 
(a)
and (b) are obvious from the way $E$ was constructed, and part (c) is 
given by Lemma \ref{nn1}.
\end{proof}

\begin{lemma}
\label{nn1}
With the PL structure given by Lemma \ref{lem.pl2}(i),

{\rm{(i)}}
$X$ is a PL $\partial$-pseudomanifold, and 

{\rm{(ii)}}
the map 
\[
\colim_{\sigma \in L}
F^\flat (\sigma)
\to
\colim_{\sigma \in K'}
F^\flat (\sigma)
= X
\]
is a PL embedding with image $\partial X$.
\end{lemma}

\begin{proof}[Proof of Lemma \ref{nn1}]
% For each cell $\sigma$ of $K'$, let us write 
% \[
% \alpha_\sigma:F^\flat(\sigma)\to X
% \]
% for the map to the colimit.  For $\sigma\subset\sigma'$ let us write
% \[
% i_{\sigma,\sigma'}:F^\flat(\sigma)\to F^\flat(\sigma')
% \]
% for the map iduced by the inclusion.

It follows from Lemma \ref{lem.pl2}(i) that
there is a triangulation of $X$ for which each 
$[F^\flat (\sigma)]$ is a subcomplex.  The top-dimensional simplices have 
dimension $n-k$, where $k$ is the degree of $F$.

The proof of \cite[II.6.2]{wh} shows that $K'$ has the 
following properties (where ``cell'' means closed cell):

\begin{enumerate}
\item Every cell is contained in an $n$-cell.

\item Every $(n-1)$-cell in $L$ is contained in exactly one $n$-cell of $K'$.

\item Every $(n-1)$-cell not in $L$ is contained in exactly two $n$-cells.

% \item For any two $n$-cells $\sigma,\sigma'$, there is a finite
% sequence $\sigma=\sigma_1,\ldots,\sigma_m=\sigma'$ of $n$-cells such that 
% each consecutive pair has an $(n-1)$-face in common.

\end{enumerate}

To see that $X$ satisfies part (a) of Definition \ref{bpseud}, let $s$ be any 
simplex of $X$.  Then $s$ is contained in $[F^\flat (\sigma)]$ 
for some $\sigma$, and (1) implies that 
$\sigma\subset \sigma'$ for some $n$-dimensional 
$\sigma'$.  Since $[F^\flat(\sigma')]$ 
is a PL $\partial$-pseudomanifold of dimension $n-k$, it has an 
$(n-k)$-dimensional simplex which contains $s$ as required.

% Next observe that, by the definition of colimit, an equation 
% $\alpha_\sigma(x)=\alpha_{\sigma'}(y)$ holds if and only if the two
% sides of the equation are connected by 
% a chain of elementary relations of the form
% $\alpha_\upsilon(a)=\alpha_{\upsilon'}(i_{\upsilon,\upsilon'}a)$.  

For part (b), let $s$ be an $n-k-1$ simplex of $X$.  

First suppose that $s$ is
not contained in any $[F^\flat(\sigma)]$ with dim$(\sigma)=n-1$.  
Then by (1) and Lemma \ref{lem.pl2}(iii), there is a unique $\sigma$ of
dimension $n$ for which $s$ is contained in $[F^\flat(\sigma)]$, and
since $[F^\flat(\sigma)]$ is a PL $\partial$-pseudomanifold of dimension $n-k$
it has exactly two $(n-k)$-dimensional simplices which contain $s$. 

Next suppose that 
$s\subset [F^\flat(\sigma)]$ with dim$(\sigma)=n-1$.  By Lemma
\ref{lem.pl2}(iii) there can only be one such $\sigma$.  If $\sigma\in L$, then
by (2) there is exactly one $n$ cell $\sigma'$ containing $\sigma$.
$[F^\flat(\sigma')]$ has exactly one $(n-k)$ simplex containing $s$, and 
Lemma \ref{lem.pl2}(iii) implies that this is the only $(n-k)$ simplex of $X$
which contains $s$.  If $\sigma\notin L$, then by (3) there are exactly two $n$
cells $\sigma'$ and $\sigma''$ which contain $s$.  Each of $[F^\flat(\sigma']$
and $[F^\flat(\sigma'')]$ has exactly one $(n-k)$ simplex which contains $s$,
and Lemma \ref{lem.pl2}(iii) implies that these are the only $(n-k)$ 
simplices of $X$ which contain $s$.

For part (c), we first observe that the proof of part (b) shows that 
\begin{equation}
\label{nn2}
\partial X=\bigcup_{\sigma\in L} [F^\flat(\sigma)].
\end{equation}
The rest of the proof for part (c) is similar to that for parts (a) and (b), 
but using \cite[II.6.2]{wh} instead of (1), (2) and (3).

For part (d), we first observe that the proof of
\cite[Proposition 6.6]{LM} shows that $\partial X$ is locally collared (in the
sense of \cite[page 24]{ro-sa.pltop}); now \cite[Theorem 2.25]{ro-sa.pltop}
shows that $\partial X$ is collared.

Finally part (ii) of the lemma follows from Lemma \ref{lem.pl2}(ii) and 
Equation \eqref{nn2}. 
\end{proof}

It remains to prove Lemma \ref{lem.pl2}.  The main ingredient is the following,
which is Exercise 2.27(2) in \cite{ro-sa.pltop}.

\begin{lemma} \label{lem.pl1}
Let $P,Q$ and $R$ be polyhedra and let
$f:R\hookrightarrow P,$ $g:R\hookrightarrow Q$ be PL embeddings. Then the 
pushout of
\[ P \hookleftarrow R \hookrightarrow Q, \]
formed in the category of topological spaces, has a PL structure for which
the inclusions of $P$ and $Q$ are PL embeddings.
\qed
\end{lemma}

\begin{proof}[Proof of Lemma \ref{lem.pl2}]
First we prove (i) and (ii).
Assume inductively that (i) and (ii) hold for any ball complex with at most $k$ 
cells.  Let $K$ be a ball complex with $k+1$ cells and let $\sigma \in K$ 
be a top-dimensional cell. Let $K_0=K - \{ \sigma \}$. 
Let $P = \colim_{\tau \in K_0} G (\tau)$ and 
$R=\colim_{\tau\subset\partial\sigma}G(\tau)$; the inductive hypothesis implies
that $P$ and $R$ have PL structures 
for which all maps $G(\tau)\to P$ and $G(\tau)\to R$ are PL embeddings, and it
also implies that $R\to P$ is a PL embedding.
We are given that the map $R\to G(\sigma)$ is a monomorphism, and it is PL
since its restriction to each $G(\tau)$ is PL.
Now let $S$ denote $\colim_{\tau\in K} G(\tau)$.  Then $S$ is 
the pushout of 
\[
P \hookleftarrow R \hookrightarrow G(\sigma),
\]
so by 
Lemma \ref{lem.pl1} it has a PL structure for which $P\to S$ and $G(\sigma)\to
S$ are PL maps; it follows that $G(\tau)\to S$ is a PL map for every $\tau$.
It remains to check that part (ii) of Lemma \ref{lem.pl2} holds, so
let $L$ be a subcomplex of $K$.  The map 
\[
i:\colim_{\tau\in L} G(\tau)\to S
\]
is PL, since its restriction to each $G(\tau)$ is PL, so we only need to check 
that $i$ is a monomorphism.  If $\sigma\notin L$ this follows from the
inductive hypothesis and the fact that $P\to S$ is a monomorphism.
If $\sigma\in L$ then $\colim_{\tau\in L} G(\tau)$ is the pushout
of
\[
\colim_{\tau\in L-\{\sigma\}}G(\tau)
\hookleftarrow
R
\hookrightarrow
G(\sigma),
\]
and this pushout maps by a monomorphism to the pushout of 
\[
P \hookleftarrow R
\hookrightarrow G(\sigma)
\]
which is $S$.

It remains to prove (iii).  Suppose that $\upsilon'\not\subset\upsilon$ and let
$L$ be the minimal subcomplex of $K$ that contains $\upsilon$ and $\upsilon'$.
By (ii), $x$ and $y$ map to the same point of $\colim_{\sigma\in L} G(\sigma)$.
But this implies that $x$ and $y$ are related by a sequence of elementary
relations of the form $a\sim G(i_{\sigma,\sigma'})(b)$, where $\sigma\subset 
\sigma'\in L$ and $i_{\sigma,\sigma'}$ is the inclusion map.  This is a 
contradiction, since $y$ cannot be part of such an elementary relation.
\end{proof}

\section{$\ad$ theories of symmetric Poincar\'e complexes}
\label{s5}

The ad theory constructed in the previous section gives a spectrum $\bQ_\ip$
and a symmetric spectrum $\bM_\ip$.  Our next goal is to use the symmetric
signature to construct maps (in the derived category of spectra and the derived
category of symmetric spectra)
from $\bQ_\ip$ and $\bM_\ip$ to suitable versions
of the symmetric L-theory spectrum of $\Z$.  In order to do this we need an ad
theory for symmetric Poincar\'e complexes over $\Z$.  In \cite[Section 9]{LM}
we gave an ad theory (denoted $\ad^\Z$) which was suggested by definitions from
\cite{ww} and \cite{ranicki}; in particular this leads to a spectrum $\bQ^\Z$ 
which is identical to Ranicki's spectrum ${\mathbb L}^\bullet (\Z)$.  But 
this turns out not to be well-adapted to questions of commutativity (see the 
beginning of \cite[Section 12]{LM2}) or to intersection homology (see the 
introduction to \cite{friedmanmcclure}), so in \cite[Section 12]{LM2} the 
second and third authors introduced a modification $\ad_\mathrm{rel}^\Z$ 
(rel stands for ``relaxed'') which 
gives a spectrum $\bQ_\mathrm{rel}^\Z$ weakly equivalent to Ranicki's 
${\mathbb L}^\bullet (\Z)$.  In this section we review this material; we 
should mention that everything extends from $\Z$ to an arbitrary 
ring-with-involution $R$ and that \cite{LM} and \cite{LM2} develop the theory 
in this generality.

\subsection{The $\ad$ theory $\ad^\Z$}
\label{ss3}

As motivation we begin with the theory $\ad^\Z$.
A chain complex over $\Z$ is called {\it finite} if it is free abelian and 
finitely generated in each degree and nonzero in only finitely many degrees; 
it is called {\it homotopy finite} if it is free abelian in each degree and
chain homotopic to a finite complex.
Let $\DD$ be the category of homotopy finite chain complexes and chain maps.
Let $W$ be the 
standard free resolution of $\Z$ by $\Z[\Z/2]$ modules. 
The $n$-dimensional objects of the target category $\Aa^\Z$ 
are pairs
$(C,\varphi)$, where $C$ is an object of $\DD$ and 
$$ \varphi : W \ra C \otimes C$$
is a $\Z/2$-equivariant chain map which raises degrees by $n$.  Here, the $\Z/2$-action on $C\otimes C$ switches the factors. The 
morphisms $(C,\varphi )\ra (C',\varphi')$ are the chain maps $f:C\to C'$, 
with the
additional requirement that $(f\otimes f)\circ \varphi=\varphi'$ when the
dimensions are equal.
The involution reverses the sign of $\varphi$. 
Next, $(\ad^\Z)^k(K)\subset (\pre^\Z)^k(K)$ is defined to be the set of 
functors $F$ with the following properties:

\quad (a) $F$ is balanced. This allows us to write
$F(\sigma,o)$ as $(C_\sigma,\varphi_{\sigma,o})$.

\quad (b)
$F$ is well-behaved, that is,  each map $C_\tau\to C_\sigma$ is a split
monomorphism
in each dimension, and (writing $C_{\partial \sigma}$ for
$\colim_{\tau\subset\partial\sigma}C_\tau$) each map
\[
C_{\partial \sigma}\to C_\sigma
\]
is a split monomorphism in each dimension.

\quad (c) $F$ is closed, that is, for each cell $\sigma$ of $K$ the graded homomorphism from 
the cellular chain complex 
$\mbox{cl}(\sigma )$ to $\Hom(W,C_\sigma \otimes C_\sigma ) $
which takes  $(\tau ,o)$ to the composite
$$
W\stackrel{\varphi_{(\tau,o)}}{\lra}C_\tau\otimes C_\tau\ra C_\sigma \otimes C_\sigma
$$
is a chain map.  This implies that $\varphi_{\sigma,o}$ represents a class
$[\varphi_{\sigma,o}]$ in $H_n(\Hom(W, (C_\sigma/C_{\partial\sigma})\otimes 
C_\sigma))$; the augmentation $\varepsilon:W\to \Z$ is a quasi-isomorphism and
hence induces an isomorphism
\[
\varepsilon^*:H_n((C_\sigma/C_{\partial\sigma})\otimes 
C_\sigma)\to H_n(\Hom(W, (C_\sigma/C_{\partial\sigma})\otimes 
C_\sigma));
\]
we denote 
$(\varepsilon^*)^{-1}[\varphi_{\sigma,o}]$ by $\mathbf 
c$.\footnote{\label{ff}Since
different notation was used in \cite{LM}, we should explain how ${\mathbf
c}$ is
related to the element denoted by $\bar{\phi}_*({\mathbf i})$ in 
\cite[Section 9]{LM}. 
First note that, if $\iota$ is the map $\Z\to W$ which 
takes 1 to 1, then ${\mathbf i}$ can be taken to be $\iota_*(1)$.  Now
the isomorphism
$
H_n((C_\sigma/C_{\partial\sigma})\otimes 
C_\sigma)\cong H_n(\Hom(\Z, (C_\sigma/C_{\partial\sigma})\otimes 
C_\sigma))
\xrightarrow{\varepsilon^*}
H_n(\Hom(W, (C_\sigma/C_{\partial\sigma})\otimes 
C_\sigma))
$
takes $\bar{\phi}_*({\mathbf i})$ to the class represented by 
$\bar{\phi}\circ\iota\circ\varepsilon$, which is the same as the class
represented by $\bar{\phi}$ (since $\iota\circ\varepsilon$ is chain homotopic to
the identity $W\to W$), and this is the same as the class we have denoted by
$[\varphi_{\sigma,o}]$.  Thus $\bar{\phi}_*({\mathbf
i})=(\varepsilon^*)^{-1}[\varphi_{\sigma,o}]=\mathbf c$.}

\quad (d)
$F$ is nondegenerate, that is, for each $\sigma$ the slant product with
$\mathbf c$ gives an isomorphism
$$
H^*(\Hom (C_\sigma,\Z))
\to
H_{\dim\sigma-k-*}(C_\sigma/ C_{\partial \sigma}).
$$

\medskip

\begin{remark}
\label{r1}
In \cite{LM}, the second and third authors 
give a construction of
a symmetric signature map
$\bQ_\STop\to \bQ^\Z$ in the derived category of spectra, using ideas from
\cite{ranicki}
(see \cite[Section 10 and the end of Section 8]{LM}).
The starting point
for this construction is the 
observation that, if $M$ is a compact oriented $\partial$-manifold and $\xi 
\in S_n(M),$ where $S_* (-)$ denotes singular chains, 
represents the fundamental class of $M$ then the composite
$$ W\cong  W \otimes \Z \stackrel{1\otimes \xi }{\lra} W \otimes 
S_*M\stackrel{EAW}{\lra} S_*M\otimes S_*M$$
(where $EAW$ is the 
extended Alexander-Whitney map, which can be constructed using acyclic 
models) is an object of $\Aa^\Z$. 
The Alexander-Whitney map (and {\it a fortiori} the extended Alexander-Whitney
map) does not exist for intersection chains, which is one
reason we need the modification of $\Aa^\Z$ given in the next subsection.
\end{remark}

\begin{remark}
The ad theory $\ad^\Z$ is multiplicative (\cite[Definitions 18.1 and 
9.12]{LM}) but not commutative (see the beginning of Section 12 of 
\cite{LM2}).
\end{remark}

\subsection{The $\ad$ theory $\ad^\Z_{\mathrm{rel}}$}
\label{adrel}

An object of the category $\Aa^\Z_\mathrm{rel}$ is a quadruple
$(C,D,\beta,\varphi)$, where $C$ is an object of $\DD$, $D$ is a chain complex
with a $\Z/2$-action, $\beta$ is a quasi-isomorphism
$C\otimes C\to D$ which is also a $\Z/2$-equivariant map, and $\varphi$
is an element of the fixed point set $D_n^{\Z/2}$.  A morphism $(C,D,\beta,\varphi)\to
(C',D',\beta',\varphi')$ is a
pair $(f:C\to C',g:D\to D')$, where $f$ and $g$ are chain
maps, $g$ is $\Z/2$-equivariant,
$g\beta=\beta'(f\otimes f)$,
and (if the dimensions are equal)
$g(\varphi)=\varphi'$.

\begin{example}
\label{e1}
If $(C,\varphi)$ is an object of $\Aa^\Z$
then the quadruple $(C,(C\otimes C)^W,\beta,\varphi)$ is a
relaxed quasi-symmetric complex, where
$$(C\otimes C) ^W=\Hom(W,C\otimes C)$$
is the $\Z/2$-chain complex with fixed points the equivariant chain maps and  $\beta:C\otimes C\to
(C\otimes C)^W$
is induced by the augmentation $W\to \Z$.  This construction gives a functor
$\Aa^\Z\to \Aa^\Z_\mathrm{rel}$.
\end{example}

\begin{example}
\label{e2}
In the situation of Remark \ref{r1}, we obtain an object of
$\Aa^\Z_\mathrm{rel}$ by letting $C$ be $S_*M$, $D$ be $S_*(M\times M)$, $\beta$
be the cross product, and $\varphi$ be the image of $\xi \in S_n (M)$ under the diagonal
map.
\end{example}

Now $(\ad^\Z_\mathrm{rel})^k(K)\subset (\pre^\Z_\mathrm{rel})^k(K)$ is defined to 
be the set of functors $F$ with the following properties:

\quad (a) $F$ is balanced. This allows us to write
$F(\sigma,o)$ as $(C_\sigma,D_\sigma, \beta_\sigma, \varphi_{\sigma,o})$.

\quad (b) $F$ is well-behaved, that is, all maps $C_\tau\to C_\sigma$,
$D_\tau\to D_\sigma$, $C_{\partial \sigma}\to C_\sigma$ and $D_{\partial\sigma}
\to D_\sigma$ are split monomorphisms in each dimension.  This implies that the
map $\beta_*:
H_*(C_\sigma\otimes C_\sigma,(C\otimes C)_{\partial\sigma})
\to
H_*(D_\sigma,D_{\partial\sigma})$ is an isomorphism (\cite[Lemma
12.7(ii)]{LM2}).

\quad (c) $F$ is closed, that is, for each $\sigma$ the map
\[
\cl(\sigma)\to D_\sigma
\]
which takes $\langle \tau,o\rangle$ to $\varphi_{\tau,o}$
is a chain map.  This implies that $\varphi_{\sigma,o}$ represents a class
$[\varphi_{\sigma,o}]$ in $H_n(D_\sigma,D_{\partial\sigma})$.

\quad (d) $F$ is nondegenerate, that is, for each $\sigma$ the
slant product with
$(\beta_*)^{-1}([\varphi_{\sigma,o}])$
is an isomorphism
\[
H^*(\Hom(C_\sigma,\Z))
\to
H_{\dim \sigma-k-*}(C_\sigma/C_{\partial\sigma}).
\]

\begin{remark}
\label{r3}
(i)
The functor $\Aa^\Z\to \Aa^\Z_\mathrm{rel}$ in Example \ref{e1} gives a map of
spectra $\bQ^\Z\to \bQ^\Z_\mathrm{rel}$ which is a weak equivalence
(\cite[Section 13]{LM2}).

(ii) The ad theory $\ad^\Z_\mathrm{rel}$ is commutative 
(\cite[Definition 3.3 and Remark 12.13]{LM2}) so Theorem 1.1 of \cite{LM2}
shows that there is a commutative symmetric ring spectrum
$(\bM^\Z_\mathrm{rel})^{\mathrm{comm}}$ which is weakly equivalent as a symmetric 
ring 
spectrum to $\bM^\Z_\mathrm{rel}$.
\end{remark}

\subsection{Connective versions.}
\label{conn}
In the sequel, we will mostly deal with connective versions of $L$-theory rather than with  periodic ones. There is a general procedure which takes an ad theory to another ad theory and which makes the associated Quinn spectrum connective: define the sub functor $\ad_{\geq 0}$ of an ad theory by
$$ \ad_{\geq 0}^k(K,L)=\ad^k(K,L\cup K^{(k-1)}),$$
where $K^{(n)}$ denotes the $n$-skeleton of $K$. 
We leave it to the reader to check the properties of an ad theory for  
$\ad_{\geq 0}$. 

In order to see that this has the desired effect on homotopy groups,
we recall from \cite[Theorem 16.1]{LM} that the homotopy group $\pi_k \bQ$ of a 
Quinn spectrum coincides with the $k$-th bordism group. The bordism group is  
obtained from the set of  $(-k)$-dimensional $*$-ads by identifying two 
elements if there is an $I$-ad which restricts to the given ones at the ends. 
For $K=*$ or $I$ we have  $\ad_{\geq 0}^{-k}(K)=\ad_{}^{-k}(K)$ for 
$k\geq 
0$ and $\ad_{\geq 0}^{-k}(*)=\{ \emptyset_{-k} \}$ for $k<0$ and the claim 
follows.

If the nonempty objects of the category $\Aa$ are concentrated in nonnegative 
dimensions then all ad theories with values in $\Aa$ are connective; this is 
the case for $\Aa_{STop}$ and $\Aa_{IP}$. 
For the theory $\ad^\Z$, there is a map from the Quinn 
spectrum $\bQ^\Z_{\geq 
0}$ to the connective spectrum $\mathbb L=\mathbf L^\bullet\langle 
0\rangle(\Z)$ which is a weak equivalence
(
the two spectra are not
identical because 
in \cite[Definition 15.2]{ranicki} the restriction to the $(k-1)$
skeleton is required to be acyclic but not zero).

A morphism of ad theories induces a morphism of their connective versions. 
Any morphism from a 
connective ad theory to another ad theory takes values in the 
connective version.

If 
the theory ad is multiplicative
(resp. commutative) then so is $\ad_{\geq 0}$.  There is a map of Quinn spectra
$\bQ_{\geq 0} \ra \bQ$ and a map of symmetric spectra $\mathbf{M}_{\geq 0} \to 
\mathbf{M}$ which is a ring map if ad is multiplicative and weakly equivalent
to an $E_\infty$ map if ad is commutative.

\section{The symmetric signature as a map of spectra}
\label{s6}

In this section we construct symmetric signature maps 
\[
\Sig:\bQ_\ip \to
\bQ^\Z_{\mathrm{rel},\geq 0}
\]
(in the derived 
category of spectra) and
\[
\Sig:\bM_\ip \to \bM^\Z_{\mathrm{rel},\geq 0}
\]
(in the derived category of symmetric spectra). 
The first step is to give a variant of the ad theory $\ad_\ip$.
As we have seen in Remark \ref{r1} and Example \ref{e2}, in order for a 
compact oriented $\partial$-manifold to give rise to an object of 
$\Aa^\Z_{ \mathrm{rel}}$ we must choose a chain representative for the 
fundamental class.  The same is true for $\partial$-IP-spaces, so in Subsection
\ref{ss1} we
construct a suitable ad theory $\ad_\ipFun$ and we show that the forgetful maps
$\bQ_\ipFun\to \bQ_\ip$ and $\bM_\ipFun\to \bM_\ip$ are weak equivalences.
Next, in Subsection \ref{ss2} we
construct a functor
\[
\sig:\Aa_\ipFun\to\Aa^\Z_{\mathrm{rel}}
\]
which induces a natural transformation
$$\sig:\ad_\ipFun(K)\to\ad^\Z_{\mathrm{rel},\geq 0}(K)$$  
for strict ball complexes $K$ (see \cite[third paragraph of Section 14 and 
Remark 14.1]{LM2} for the meaning and
significance of strictness).
These results allow us to make the following definition.
\begin{definition} 
\label{d4}
The symmetric signature map
\[
\Sig:\bQ_\ip \to
\bQ^\Z_{\mathrm{rel},\geq 0}
\]
is the composite
\[
\bQ_\ip \xleftarrow{\simeq}
\bQ_\ipFun
\xrightarrow{\sig}
\bQ^\Z_{\mathrm{rel},\geq 0}.
\]
The symmetric signature map 
\[
\Sig:\bM_\ip \to \bM^\Z_{\mathrm{rel},\geq 0}
\]
is the composite
\[
\bM_\ip \xleftarrow{\simeq}
\bM_\ipFun
\xrightarrow{\sig}
\bM^\Z_{\mathrm{rel},\geq 0}.
\]
\end{definition} 

\subsection{The $\ad$ theory $\ad_\ipFun$}
\label{ss1}

We denote singular intersection chains with perversity $\bar{p}$ by 
$IS^{\bar{p}}_*$.  By \cite[Proposition 7.7]{FM}, 
an orientation of a compact $n$-dimensional $\partial$-IP-space $X$ 
determines a fundamental class $\Gamma_X\in IH^{\bar{0}}_n(X,\partial 
X;\Z)$, where $\bar{0}$ denotes the 0 perversity.

We define a category 
$\Aa_\ipFun$ as follows. The objects are pairs $(X,\xi)$, where $X$
is a compact oriented $\partial$-IP-space and 
$\xi \in IS^{\bar{0}}_n(X;\Z)$ is a chain whose image in
$IS^{\bar{0}}_n(X,\partial X;\Z)$ it represents
the fundamental class $\Gamma_X$; there is also an empty object of
dimension $n$ for each $n$.  The morphisms $(X,\xi)\to(X',\xi')$ 
between objects of the same dimension are PL homeomorphisms which take $\xi$
to $\xi'$, and the other morphisms are strong embeddings with image in the
boundary.
There is a forgetful functor $\Aa_\ipFun\to \Aa_\ip$, and we define 
$\ad_\ipFun^k(K)\subset \pre_\ipFun^k(K)$ to be the set of functors $F$
such that 

\quad (a) the composite of $F$ with the forgetful functor is an element of
$\ad_\ip^k(K)$, and

\quad (b) for each oriented cell $(\sigma,o)$ of $K$, the equation 
\[
\partial \xi_{\sigma,o}
=
\sum \xi_{\sigma',o'}
\]
holds, where $\sigma'$ runs through the cells of $\partial \sigma$ and 
$o'$ is the orientation for which the incidence number $[o,o']$ is $(-1)^k$.

\begin{prop}
\label{fp1}
$\ad_\ipFun$ is a connective ad theory.
\end{prop}
\begin{proof}
We only need to check parts (f) and (g) of \cite[Definition
3.10]{LM}.  
For the proof of (f) we use the gluing construction in the proof 
of Theorem \ref{t1} (and the notation there) and we define $\xi_{\tau,o}$ to 
be $\sum \xi_{\sigma,o_\sigma}$, where
$\sigma$ runs throught the $n$-cells of $K'$; then $\xi$ maps to a
representative of $\Gamma_X$ by \cite[Corollary 5.16]{FM}.

For (g), let $F$ be a $K$-ad, and write 
$F(\sigma,o)=(X_\sigma,\xi_{\sigma,o})$.
As in the proof of Theorem \ref{t1}, we only need to specify $J(F)$ on
oriented cells of the form $(\sigma\times I,o\times o')$, and moreover we can
assume that $o'$ is the standard orientation.  Let 
$
s:\Delta^1\to I
$
be the standard oriented homeomorphism, and define
$
J(F)(\sigma\times I,o\times o')=
(X_\sigma\times I,\xi\times s).
$
\end{proof}
The forgetful functor $\Aa_\ipFun\to\Aa_\ip$ gives rise to a morphism
$\ad_\ipFun \to \ad_\ip$ of ad theories.
\begin{prop}
\label{fp2}
The maps 
\[
\bQ_\ipFun
\to
\bQ_\ip
\]
and
\[
\bM_\ipFun
\to
\bM_\ip
\]
induced by the forgetful functor $\Aa_\ipFun\to\Aa_\ip$
are weak equivalences.
\end{prop}

\begin{proof}
Recall the definition of the bordism groups of an ad theory (\cite[Definitions
4.1 and 4.2]{LM}).
By Remark \ref{aa41} and \cite[Proposition 17.7]{LM}, it
suffices to show that the map of bordism groups
\[
(\Omega_\ipFun)_*\to (\Omega_\ip)_*
\]
is an isomorphism.  This map is obviously onto, and it is a monomorphism by the
proof of \cite[Lemma 8.2]{LM}.
\end{proof}

\begin{remark}
\label{r7}
$\ad_\ipFun$ is a commutative ad theory, so
by Theorem 1.1 of
\cite{LM2} there is a commutative symmetric ring spectrum
$\bM^{\mathrm{comm}}_\ipFun$
which is weakly equivalent as a symmetric ring spectrum to $\bM_\ipFun$.
Moreover, the forgetful map $\ad_\ipFun\to \ad_\ip$ is strictly multiplicative,
so the proof of \cite[Theorem 1.1]{LM2} gives a commutative diagram
\[
\xymatrix{
\bM_\ip
&
{\mathbf A}
\ar[l]_-\simeq
\ar[r]^-\simeq
&
\bM^{\mathrm{comm}}_\ip
\\
\bM_\ipFun
\ar[u]_\simeq
&
{\mathbf B}
\ar[l]_-\simeq
\ar[r]^-\simeq
\ar[u]_\simeq
&
\bM^{\mathrm{comm}}_\ipFun
\ar[u]_\simeq
}
\]
in which $\mathbf A$ and $\mathbf B$ are symmetric ring spectra and all
arrows are ring maps.
\end{remark}

\subsection{Background}

Before proceeding we need to recall some information about generalized 
perversities.  
For a stratified $n$-dimensional $\partial$-pseudomanifold $Y$ the 
components of $Y^i-Y^{i-1}$ are called {\it $i$-dimensional strata}; the
$n$-dimensional strata are called {\it regular} and the others {\it singular}.
Recall (\cite[Definition 3.1.1]{Fr}) that a {\it generalized
perversity}\footnote{These are simply called perversities in \cite{Fr}.}
on $Y$  is a function $\bar{p}$ from the set of strata of $Y$ to $\Z$ which is 
0 on the regular strata;
an ordinary perversity $\bar{q}$ can be thought of as a generalized perversity
taking a stratum $S$ to $\bar{q}(\mathrm{codim}(S))$.
We use the definition of intersection homology for general perversities given
in \cite[Definition 6.2.2]{Fr}.
Let $\bar{n}$ be the upper middle perversity.

Let $X$ be a $\partial$-IP-space, and 
give $X$ the stratification of Proposition \ref{p2}.  Give $X\times X$ the
product stratification. Define
a generalized perversity
$Q_{\bar{n},\bar{n}}$ on $X\times X$ as follows.  
\[
Q_{\bar n,\bar n}(S_1\times S_2)=\begin{cases}
 \bar n(S_1)+\bar n(S_2)+2, &\text{$S_1,S_2$ both singular strata,}\\
\bar n(S_1), &\text{$S_2$ a regular stratum and $S_1$ singular,}\\
\bar n(S_2), &\text{$S_1$ a regular stratum and $S_2$ singular,}\\
0,&\text{$S_1,S_2$ both regular strata.}
 \end{cases}
\]
By \cite[Subsection 4.1]{FM}, the diagonal map induces a chain map
\begin{equation}
\label{eq3}
d:
IS^{\bar{0}}_*(X;\Z)
\to 
IS^{Q_{\bar{n},\bar{n}}}_*(X\times X;\Z)
\end{equation}
(this is the reason we need generalized perversities).
By \cite[Theorem 6.4.6 and Remark 6.4.7]{Fr}, the cross product induces an 
equivalence
\[
IS^{\bar{n}}_*(X;\Z)
\otimes
IS^{\bar{n}}_*(X;\Z)
\to
IS^{Q_{\bar{n},\bar{n}}}_*(X\times X;\Z).
\]

\subsection{The functor $\sig:\Aa_\ipFun\to\Aa^\intg_{\mathrm{rel}}$}
\label{ss2}

\begin{lemma}
\label{l1}
{\rm (i)}
Let $(X,\xi)$ be an object of $\Aa_\ipFun$.  Give $X$ the stratification of
Proposition \ref{p2}
and give $X\times X$ the product stratification.  Then 
$(C,D,\beta,\varphi)$ is an object of $\Aa^\Z_\mathrm{rel}$, 
where
\begin{align*}
&C=IS_*^{\bar{n}}(X;\Z), \\
&D=IS^{Q_{\bar{n},\bar{n}}}_*(X\times X;\Z), \\
& \beta\ \text{is the cross product, and}\\
& \varphi\ \text{is the image of
$\xi$ under the diagonal map \eqref{eq3}.}
\end{align*}

{\rm (ii)} Let $f:(X,\xi)\to (X',\xi')$ be a morphism in $\Aa_\ipFun$ and let
$(C,D,\beta,\varphi)$ and $(C',D',\beta',\varphi')$ be the objects of
$\Aa^\Z_\mathrm{rel}$ corresponding to $(X,\xi)$ and $(X',\xi')$.  Then $f$
induces a morphism $(C,D,\beta,\varphi)\to (C',D',\beta',\varphi')$.
\end{lemma}

\begin{proof}
(i) 
We only need to show that $IS_*^{\bar{n}}(X;\Z)$ is
homotopy finite. The complex
$IS_*^{\bar{n}}(X;\Z)$ is free because
it is a subcomplex of the singular chain complex
$S_*(X;\Z)$.  Next let $T$ be a triangulation of $X$ 
which is compatible with the stratification of $X$.  
Let
$IC_*^{T,\bar{n}}(X;\Z)$ denote the complex of PL intersection
chains which are simplicial with respect to $T$.  
This is free, finitely generated in each degree, and nonzero in only finitely
many degrees.
By \cite[Corollary 5.4.6]{Fr}, 
the inclusion 
\[
IC_*^{T,\bar{n}}(X;\Z)
\to
IS_*^{\bar{n}}(X;\Z)
\]
is a quasi-isomorphism.
Since the domain and range are free, it is a chain
homotopy equivalence, and thus $IS_*^{\bar{n}}(X;\Z)$ is
homotopy finite.

(ii) If the dimensions are not equal then the definition of $\Aa_\ipFun$ shows
that $f$, and hence also $f\times f$, is a strong embedding,
so they induce maps of intersection chains and 
the result follows.  If the dimensions are equal then $f$ is
a PL homeomorphism so, by Propositions \ref{p1}(iii) and \ref{p2}, $f$ and 
$f\times f$ preserve the filtrations and therefore induce maps of 
intersection chains.
\end{proof}
Lemma \ref{l1} gives a functor
\[
\sig:\Aa_\ipFun\to\Aa^\Z_{{ \mathrm{rel}}}.
\]

\begin{prop}
\label{fp3}
If $K$ is strict and $F\in\ad_\ipFun^k(K)$ then $\sig\circ F\in 
(\ad^\Z_{\mathrm{rel},\geq 0})^k(K)$.
\end{prop}

The proposition gives the maps
\[
\sig:\bQ_\ipFun
\to
\bQ^\Z_{\mathrm{rel},\geq 0}
\]
and
\[
\sig:\bM_\ipFun
\to
\bM^\Z_{\mathrm{rel},\geq 0}
\]
which are needed for Definition \ref{d4}.

\begin{proof}[Proof of Proposition \ref{fp3}]
Since $\sig \circ F$ is closed by property (b) of the definition in Subsection 
\ref{ss1},
we only need to check that $\sig\circ F$ is well-behaved and nondegenerate.
Write $F(\sigma,o)=(X_{\sigma,o},\xi_{\sigma,o})$.  Let $Y_\sigma$ be the
underlying $\partial$-IP-space of $X_{\sigma,o}$ (forgetting the orientation).

To show that $\sig\circ F$ is well-behaved, we need to show that the functors
$IS_i^{\bar{n}}(X_\sigma;\Z)$ and $IS^{Q_{\bar{n},\bar{n}}}_*(X_\sigma\times
X_\sigma;\Z)$ are well-behaved.  We give the proof for the first of these 
functors; the proof for the second is similar.

First we need to know that for cells $\sigma$, $\tau$ of $K$ with
$\sigma\subset\tau$ the monomorphism
\[
IS_i^{\bar{n}}(Y_\sigma;\Z)
\to
IS_i^{\bar{n}}(Y_\tau;\Z)
\]
is split for each $i$.  For this it suffices to show that the quotient
$IS_i^{\bar{n}}(Y_\tau;\Z)/IS_i^{\bar{n}}(Y_\sigma;\Z)$ is free, and this in
turn follows from the fact that the map from this quotient to the free 
abelian group $S_i(Y_\tau;\Z)/S_i(Y_\sigma;\Z)$ is a monomorphism.

Next we need to know that for each cell $\sigma$ the map
\[
\colim_{\sigma\subset\partial\tau} IS_i^{\bar{n}}(Y_\sigma;\Z)
\to
IS_i^{\bar n}(Y_\tau ; \Z)
\]
is a split monomorphism.  It is a monomorphism by Lemma \ref{aa11}(iii).  To
see that it is split it suffices to show that the quotient
\[
IS_i^{\bar{n}}(Y_\tau;\Z)/\colim_{\sigma\subset\partial\tau}
IS_i^{\bar{n}}(Y_\sigma;\Z)
\]
is free.
By Lemma \ref{aa11}(iii) this is the same as the quotient
\[
IS_i^{\bar{n}}(Y_\tau;\Z)/\sum_{\sigma\subset\partial\tau}
IS_i^{\bar{n}}(Y_\sigma;\Z),
\]
and by Lemma \ref{aa11}(ii) the map from this 
quotient to the free abelian group
$S_i(Y_\tau;\Z)/\sum_{\sigma\subset\partial\tau}
S_i(Y_\sigma;\Z)$
is a monomorphism.
This concludes the proof that $IS_i^{\bar{n}}(Y_\sigma;\Z)$ is well-behaved.

For nondegeneracy, we first observe that, by Lemma \ref{aa11} below, the map
\[
\colim_{\tau\in \partial\sigma} IS_*^{\bar{n}}(Y_\tau;\Z)
\to
IS_*^{\bar{n}}(\partial Y_\sigma;\Z)
\]
is a quasi-isomorphism for every simplex $\sigma$ of $K$.  

Now it suffices to show that the horizontal map in the following
diagram is an isomorphism for each oriented simplex $(\sigma,o)$. 
\[
\xymatrix{
H_*(\Hom(IS_*^{\bar{n}}(X_{\sigma,o};\Z),\Z))
\ar[rr]^-{\backslash (\beta_*)^{-1}([\varphi_{\sigma,o}])}
\ar[rrd]_-{\smallfrown \Gamma_{X_{\sigma,o}}}
&&
IH_{\dim\sigma-k-*}^{\bar{n}}(X_{\sigma,o},\partial X_{\sigma,o};\Z)
\\
&&
IH_{\dim\sigma-k-*}^{\bar{m}}(X_{\sigma,o},\partial X_{\sigma,o};\Z)
\ar[u]
}
\]
The construction of the cap product is given in Appendix \ref{tech}, and the
fundamental class $\Gamma_{X_{\sigma,o}}$ is given by \cite[Proposition 
7.7]{FM}.
Inspection of the definitions shows that the diagram commutes, and the slanted
arrow is an isomorphism by Theorem \ref{T: univ lef}, so we only need to show 
that the vertical arrow is an isomorphism.  
For this it suffices to show that the maps
\[
IH_*^{\bar m}(\partial X_{\sigma,o};\Z)\to
IH_*^{\bar n}(\partial X_{\sigma,o};\Z)
\]
and 
\[
IH_*^{\bar m}(X_{\sigma,o};\Z)\to
IH_*^{\bar n}(X_{\sigma,o};\Z)
\]
are isomorphisms.  The first of these is an isomorphism by Proposition 
\ref{p3} and the argument in \cite[Subsection 5.6.1]{gm2}.  To see that the
second map is an isomorphism we oberve that if $W$ denotes $X_{\sigma,o}$ 
with a 
collar of the boundary removed then the maps $W\to X_{\sigma,o}$ and $W\to
X_{\sigma,o}-\partial 
X_{\sigma,o}$ are stratified homotopy equivalences and therefore induce 
isomorphisms of 
intersection homology (see \cite[Appendix A]{FM}, so it suffices to observe
that the map
\[
IH_*^{\bar m}(X_{\sigma,o}-\partial X_{\sigma,o};\Z)\to
IH_*^{\bar n}(X_{\sigma,o}-\partial X_{\sigma,o};\Z)
\]
is an isomorphism by the argument in \cite[Subsection 5.6.1]{gm2}.
\end{proof}

Before stating the lemma let us recall that intersection
homology and the notion of stratified homotopy can be defined for any 
filtered space, and that a stratified 
homotopy equivalence induces an isomorphism of intersection
homology (\cite[Section 2]{fr}).  Also recall that if $X$ is a functor from 
$\Cell(K)$ to PL spaces and $L$ is a subcomplex of $K$ then we write $X_L$
for $\colim_{\sigma\in L} X_\sigma$.

Note: part (i) of the following lemma is stated in more generality than is
needed in this section (the case $U=X_{L'}$ would suffice for that); the extra
generality is needed for Subsection \ref{aa2}.

\begin{lemma}
\label{aa11}
Let $F\in\ad_\ipFun^k(K)$, and write $F(\sigma,o)=(X_\sigma,\xi_{\sigma,o})$.
For every subcomplex $L$ of $K$, give $X_L$ the
filtration which restricts to the filtration  of
Proposition \ref{p2}
on each $X_\sigma$.  Then

{\rm{(i)}} for every pair of subcomplexes $L'\subset L$, and for every open set
$U$ of $L'$, there is a 
neighborhood
$V$ of $U$ in $X_L$
such that $V\cap X_{L'}=U$ and the inclusion $U\to V$ is a
stratified deformation retract where the retraction $r:V\to U$ has the
property that $r(x)\in X_\sigma$ whenever $x\in X_\sigma$, and

{\rm{(ii)}} for every subcomplex $L$ and every $i\in \Z$, the intersection 
of $IS^{\bar n}_i (X_K;\Z)$ and
$\sum_{\sigma\in L} S_i(X_\sigma;\Z)$ (considered as subgroups of
$S_i(X_K;\Z)$)
is
$\sum_{\sigma\in L} IS^{\bar n}_i(X_\sigma)$,
and

{\rm{(iii)}} if $K$ is strict then for every subcomplex $L$ the map
\[
\colim_{\sigma\in L} IS^{\bar n}_*(X_\sigma;\Z)
\to
\sum_{\sigma\in L} IS^{\bar n}_*(X_\sigma;\Z)
\]
is an isomorphism, and

{\rm{(iv)}} for every subcomplex $L$ the map
\[
\colim_{\sigma\in L} IS_*^{\bar{n}}(X_\sigma;\Z)
\to
IS_*^{\bar{n}}(X_L)
\]
is a quasi-isomorphism.
\end{lemma}

\begin{remark}
\label{aa10}
For the proof of part (iv)
we will use the following fact: given a commutative diagram
\[
\xymatrix{
A
\ar[d]
&
B
\ar[d]
\ar[l]
\ar[r]
&
C
\ar[d]
\\
D
&
E
\ar[l]
\ar[r]
&
F
}
\]
in the category of chain complexes, where all horizontal maps are monomorphisms
and all vertical maps are quasi-isomorphisms, then the induced map from the
pushout of the top row to the pushout of the bottom row is a quasi-isomorphism.
\end{remark}

\begin{proof}[Proof of Lemma \ref{aa11}]
Give the set of subcomplexes $L$ a total ordering such that if $L'\subset L$
then $L'<L$.  We will prove each part by induction over this
total order.  So let $L$ be a subcomplex and suppose that all parts have
been proved for all subcomplexes $<L$.

For (i), let $L'$ be a subcomplex of $L$ and let $U$ be an open set of $L'$.  
Let $\tau$ be a cell of $L$ of
maximal dimension which
is not in $L'$, and let $M$ be the subcomplex of $L$ consisting of all cells
except $\tau$.  By inductive hypothesis, 
$U$ has a neighborhood $W$ in $X_M$ with the properties given in the statement
of part (i).
By Definition \ref{d2}(c), $W \cap X_{\partial\tau}$ has a collar neighborhood
$W_1$ for which the inclusion $W \cap X_{\partial\tau}\to W_1$ is a stratified
deformation retract.  Now let $V=W\cup W_1$; then $V$ is the desired
neighborhood of $U$ in $X_L$.
% because V\cup W retracts to V and then to X_{L'}

For parts (ii), (iii) and (iv), let $\tau$ be a cell of $L$ of maximal
dimension and let $L'$ be
the subcomplex of $L$
consisting of all cells except $\tau$.

For (ii), let $\xi\in IS^{\bar n}_i (X_K;\Z)\cap
\sum_{\sigma\in L} S_i(X_\sigma;\Z)$.  Write 
\[
\xi=\sum a_j s_j,
\]
where $a_j\in \Z$ and the $s_j$ are singular simplices.  Let $U$ be the neighborhood of
$X_{\tau}$ in $X_L$ given by part (i) and let $r:U\to 
X_{\tau}$ be the stratified retraction.  
Applying Proposition \ref{aa17} with the open set $U$ gives an intersection
chain
\[
\bar{\xi}=\sum a_j \bar{s_j}.
\]
If we write 
\[
\eta=\sum_{\supp(s_j)\subset X_\tau}a_j  s_j
\]
then $\xi-\eta$, $\bar{\xi}-\eta$ and $r_*(\bar{\xi}-\eta)$ are all in 
in $\sum_{\sigma\in L'} S_*(X_\sigma;\Z)$. 
Now 
\[
\xi=(\eta+r_*(\bar{\xi}-\eta))+ (\xi-\eta-r_*(\bar{\xi}-\eta)).
\]
The first summand is in $IS_*^{\bar{n}}(X_\tau;\Z)$ (because it is equal to
$r_*(\bar{\xi})$), and the second is in
$IS_*^{\bar{n}}(X_K;\Z)$ (because $\xi$ and the first summand are) and in
$\sum_{\sigma\in L'} S_*(X_\sigma;\Z)$, so by the inductive hypothesis $\xi$
is in $\sum_{\sigma\in L} IS^{\bar n}_i(X_\sigma, \Z)$ as required.

For (iii) and (iv), we use the fact that 
$\colim_{\sigma\in L}IS_*^{\bar{n}}(X_\sigma;\Z)$ is
the pushout of the diagram
\begin{equation}
\label{aa20}
\colim_{\sigma\in L'}IS_*^{\bar{n}}(X_\sigma;\Z)
\leftarrow
\colim_{\sigma\in \partial \tau}IS_*^{\bar{n}}(X_\sigma;\Z)
\rightarrow
IS_*^{\bar{n}}(X_\tau;\Z).
\end{equation}

For (iii), the map in question is obviously an epimorphism, so we only need 
to show that the map from the pushout of \eqref{aa20} to 
$IS^{\bar n}_*(X_L;\Z)$ is a monomorphism.  Using the inductive hypothesis,
the pushout of \eqref{aa20} is the same as the pushout of 
\[
\sum_{\sigma\in L'}IS_*^{\bar{n}}(X_\sigma;\Z)
\leftarrow
\sum_{\sigma\in \partial \tau}IS_*^{\bar{n}}(X_\sigma;\Z)
\rightarrow
IS_*^{\bar{n}}(X_\tau;\Z).
\]
Suppose that $\xi\in
\sum_{\sigma\in L'}IS_*^{\bar{n}}(X_\sigma;\Z)$ and $\eta\in
IS_*^{\bar{n}}(X_\tau;\Z)$ with 
\begin{equation}
\label{aa21}
\xi+\eta=0; 
\end{equation}
we need to show that $\xi\in
\sum_{\sigma\in \partial \tau}IS_*^{\bar{n}}(X_\sigma;\Z)$; by part (ii) it
suffices to show that 
$\xi\in \sum_{\sigma\in \partial\tau}S_*(X_\sigma;\Z)$.
Equation \eqref{aa21} implies that each simplex $s$ of $\xi$ has its support 
in both $\tau$ and a cell of $L'$; since $K$ is strict this implies that it has
its support in a cell of $\partial \tau$ as required.

Proof of (iv).
By the inductive hypothesis and Remark \ref{aa10}, the map from the  pushout 
of \eqref{aa20}
to the pushout of
\begin{equation}
\label{aa12}
IS_*^{\bar{n}}(X_{L'};\Z)
\leftarrow
IS_*^{\bar{n}}(X_{\partial\tau};\Z)
\to
IS_*^{\bar{n}}(X_\tau;\Z)
\end{equation}
is a quasi-isomorphism.
By part (i) there is a neighborhood
$U$ of $X_{L'}$ in $X_L$ and a neighborhood $V$ of $X_\tau$ in $X_L$ such 
that
the inclusions $X_{L'}\to U$ and $X_\tau\to V$ are stratified deformation
retractions. Then $X_{\partial\tau}\to U\cap V$ is a stratified deformation 
retraction, and hence the map from the pushout of \eqref{aa12} to the pushout
of
\begin{equation}
\label{aa13}
IS_*^{\bar{n}}(U;\Z)
\leftarrow
IS_*^{\bar{n}}(U\cap V;\Z)
\to
IS_*^{\bar{n}}(V;\Z)
\end{equation}
is a quasi-isomorphism.  Finally, the map from the pushout of \eqref{aa13} to
$IS_*^{\bar{n}}(X_L;\Z)$ is a quasi-isomorphism by \cite[Proposition
6.3]{friedmanmcclure}
and \cite[Proposition 
2.9]{FTrans}.
\end{proof}

%\begin{remark}
%\label{aa19}
%For later use we note that the proof of part (ii) shows something more
%general:
%if $V$ is an open subset of $X_K$ then
%\[
%IS^{\bar n}_i (X_K;\Z)\cap
%\sum_{\sigma\in L} S_i(V\cap X_\sigma;\Z)=
%\sum_{\sigma\in L} IS^{\bar n}_i(V\cap X_\sigma).
%\]
%\end{remark}

\section{$L$-theory of the fundamental groupoid.}
\label{ringoids}

Let $Z$ be a path-connected space.

For our further work, we need to use the fundamental groupoid $\pi_1 Z$ rather
than the fundamental group $\pi_1(Z,z_0)$.  The main reason is that we will be
using the results of \cite{wwa}, and these require a functor defined on
unbased spaces; cf.\  lines 9--12 of \cite[Subsection 2.1]{wwa}. An additional
benefit is that we won't need to choose basepoints. 

In this section we define $L$-spectra and $L$-groups over $\pi_1 Z$. We will
make use of the definitions in Appendix \ref{add}.

First we define an additive category
${\mathbb{Z}}[\pi_1 Z]$ by letting the objects be the points of $Z$  and letting the
abelian group of morphisms from $z$ to $z'$ be the free abelian group
generated by the morphisms from $z$ to $z'$ in the groupoid. The composition of morphisms is similar to the multiplication in a group ring.

For a module $\M$ we write 
$\M_z$ for  the restriction of $\M$ to
the full subcategory ${\mathbb{Z}}[\pi_1 (Z,z)]$ with only one object $z$; this is a left module (in the
usual sense) over the ring ${\mathbb{Z}}[\pi_1 (Z,z)]$, and $\M$ is 
determined up to unique isomorphism by $\M_z$.

\begin{remark}
\label{aa5}
The evident map $\M_z\otimes_{\Z[\pi_1(Z,z)]} \N_z\to 
\M\otimes_{\Z[\pi_1 Z]} \N$ is an isomorphism for any $z$ (because there is a
map in the other direction which is inverse to it).
\end{remark}

For left modules $\M$ and $\M'$, we define $\Hom_{{\mathbb{Z}}[\pi_1
Z]}(\M,\M')$ to be the abelian group of natural transformations.  

\begin{remark}
\label{aa6}
The
restriction map
\[
\Hom_{{\mathbb{Z}}[\pi_1 Z]}(\M,\M')
\to
\Hom_{{\mathbb{Z}}[\pi_1(Z,z)]}(\M_z,\M'_z)
\]
is an isomorphism for all $z$ (because there is a map in the other direction
which is inverse to it).
\end{remark}

$\M$ is {\it free} if $\M_z$ is a free 
${\mathbb{Z}}[\pi_1 (Z,z)]$-module for
some (and hence for all) $z$.  

The category of left ${\mathbb{Z}}[\pi_1 Z]$ modules is an abelian category, so
the concepts of chain complex, chain homotopy, and quasi-isomorphism can be
defined in the usual way.  

\begin{remark}
\label{nn19} 
Note that two chain complexes are chain homotopy equivalent
(resp., quasi-isomorphic) over ${\mathbb{Z}}[\pi_1 Z]$ if and only if their
restrictions to ${\mathbb{Z}}[\pi_1(Z,z)]$ are chain homotopy equivalent (resp.,
quasi-isomorphic) over ${\mathbb{Z}}[\pi_1(Z,z)]$ for some $z$.
\end{remark}

The {\it involution} of ${\mathbb{Z}}[\pi_1 Z]$ is the additive functor
\[
{\mathbb{Z}}[\pi_1 Z]\to {\mathbb{Z}}[\pi_1 Z]^{\mathrm{op}}
\]
which takes a morphism $g:z\to z'$ in the groupoid to $g^{-1}:z'\to z$.  This 
restricts to the usual involution of ${\mathbb{Z}}[\pi_1(Z,z)]$.

With these definitions, it is straightforward to generalize
the definitions and results of \cite[Section 9]{LM} and \cite[Sections
12--14]{LM2} with the ring $R$ replaced by ${\mathbb{Z}}[\pi_1 Z]$, with these
modifications:

\begin{itemize}
\item
For the
analogue of \cite[Definition 9.9(b)]{LM}, the left $R$-module $R$ should be 
replaced by any left ${\mathbb{Z}}[\pi_1 Z]$ module of the form
$\Hom(z,-)$.
\item
For the proof of the analogue of \cite[Theorem 9.11]{LM}, note that 
a pre $K$-ad $G$ with values in $\Aa^{\Z[\pi_1 Z]}$ determines, for each $z\in
Z$, a pre $K$-ad $G_z$ with values in $\Aa^{\Z[\pi_1(Z,z)]}$, and that (using
Remarks \ref{aa5} and \ref{aa6})
$G$ is
an ad if and only if each $G_z$ is.
\end{itemize}

In 
particular, we obtain an ad theory $\ad^{{\mathbb{Z}}[\pi_1 Z]}$, bordism
groups $\Omega_*^{{\mathbb{Z}}[\pi_1 Z]}$, and spectra $\bQ^{{\mathbb{Z}}[\pi_1 
Z]}$ and
$\bQ^{{\mathbb{Z}}[\pi_1 Z]}_{\geq 0}$.  In analogy with \cite[page 
60]{ranicki} we define 
$L^n({\mathbb{Z}}[\pi_1 Z])$ to be 
$\Omega_n^{{\mathbb{Z}}[\pi_1
Z]}$ for $n\geq 0$ and 0 for $n<0$; by Remark \ref{aa41},
$L^n({\mathbb{Z}}[\pi_1 Z])\cong
\pi_n \bQ^{{\mathbb{Z}}[\pi_1 Z]}_{\geq 0}$
(cf.\ Remark \ref{not}).

\begin{remark}
\label{rev4}
Let $z\in Z$.
From what has been said it is clear that the forgetful functor from 
${\mathbb{Z}}[\pi_1 Z]$ modules to ${\mathbb{Z}}[\pi_1(Z,z)]$ modules
induces an isomorphism
\[
L^n({\mathbb{Z}}[\pi_1 Z])
\to
L^n({\mathbb{Z}}[\pi_1(Z,z)]).
\]
\end{remark}

\begin{remark}
\label{tens}
For use in the next section we observe that if $\M$ and $\M'$ are left 
${\mathbb{Z}}[\pi_1 Z]$
modules
we can define a tensor product $\M\otimes \M'$ to be the left 
${\mathbb{Z}}[\pi_1 Z]$
module which 
takes $z$ to $\M_z\otimes \M'_z$, with the evident action of the morphisms.
If $\ZZ$ is the constant right module with value $\Z$ then the isomorphism
of left $Z[\pi_1 Z]$ modules
\[
\ZZ^t\otimes\M\otimes \M' \cong \M\otimes \M'
\]
induces an isomorphism 
\[
\ZZ\otimes_{{\mathbb{Z}}[\pi_1 Z]}(\M\otimes \M') \cong
\M^t\otimes_{{\mathbb{Z}}[\pi_1 Z]} \M'.
\]
\end{remark}

\begin{remark}
\label{nn12}
For use in Section \ref{gs6}, we give a 
variant of $\ad^{\Z[\pi_1 Z]}$.
Let $\Aa^{\Z[\pi_1 Z]}_\mathrm{fin}$ be the full subcategory of
$\Aa^{\Z[\pi_1 Z]}$ consisting of objects $(C,\varphi)$ for which
$C$ is finite (not just homotopy finite).
For a ball complex $K$, let $\ad^{\Z[\pi_1 Z]}_\mathrm{fin}(K)$ be the 
set of pre $K$-ads with values in $\Aa^{\Z[\pi_1 Z]}_\mathrm{fin}$
for which the composite with the forgetful functor to $\Aa^{\Z[\pi_1 Z]}$
is an ad.  The proof of \cite[Theorem 9.11]{LM} 
generalizes to show that  $\ad^{\Z[\pi_1 Z]}_\mathrm{fin}$ is an ad theory, and the 
proof of \cite[Proposition B.17]{LM2} (specifically,
the proof that $\Omega^c_*$ is an isomorphism) shows that the map 
$\ad^{\Z[\pi_1 Z]}_\mathrm{fin}\to \ad^{\Z[\pi_1 Z]}$ induces an isomorphism of 
bordism groups.
\end{remark}

\section{The symmetric signature of an IP-space}
\label{symsig}

For a compact oriented $n$-manifold $M$ (and more generally for a 
Poincar\'e duality space) the symmetric signature $\sigma^*(M)$ is an element 
of the symmetric $L$-group $L^n(\Z[\pi_1 M] )$.
The symmetric signature was introduced by
Mi\v{s}\v{c}enko\footnote{Mi\v{s}\v{c}enko's construction gives an element of
$L^n({\mathbb{Z}}[\pi_1(X,x)])$ for each $x$, and these are consistent as $x$
varies, so by Remark \ref{rev4} one obtains a well-defined element of
$L^n({\mathbb{Z}}[\pi_1 X])$.}
as a tool for studying the Novikov conjecture, and since then it has become an
important part of surgery theory (see \cite{ranicki}, for example).  The
symmetric signature has many useful properties, such as homotopy invariance,
bordism invariance, and a product formula.

The paper \cite{csw} has a brief description of a construction (using
controlled topology) which assigns to a compact oriented Witt space $X$ and a
point $x\in X$ a symmetric signature 
in $L^n({\mathbb{Q}}[\pi_1(X,x)])$, with properties analogous to those of 
Mi\v{s}\v{c}enko's symmetric signature (further information about this 
construction is given in \cite[pages 209-210]{Wein}).  A simpler construction 
with the same properties 
was given in \cite[Subsection 5.4]{friedmanmcclure}. The two constructions are 
known to agree rationally by an argument due to Weinberger (cf.\ 
\cite[Proposition 11.1]{almp}) and, independently, Banagl-Cappell-Shaneson
\cite[Proposition 2]{bc1}.

In this section we show that when $X$ is a connected compact oriented IP-space 
of dimension $n$
the method used in 
\cite{friedmanmcclure} gives a symmetric signature 
\[
\sigma^*_\ip(X)\in L^n({\mathbb{Z}}[\pi_1 X])
\]
with the usual properties.

\begin{remark}
\label{nn32}
Recall (\cite[Section 12]{LM2}) that the relaxed symmetric Poincar\'e 
ad theory $\ad^\Z_{\mathrm{rel}}$ described in Subsection \ref{adrel} has an 
analog  $\ad^R_{\mathrm{rel}}$ when $R$ is any ring with involution.  The same
construction gives an ad theory $\ad^{{\mathbb{Z}}[\pi_1 X]}_{\mathrm{rel}}$.%
\footnote{In \cite[Section 9]{LM} it was assumed that for any
object $(C,\varphi)$ the chain
complex $C$ is free over $R$ (this assumption was built into the definition 
of the category $\mathcal D$ \cite[Definition 9.2(v)]{LM}).  This was for two
reasons: 
to ensure that the $L$ groups would 
be the same as in \cite{ranicki}, but it is not needed for this, see 
\cite[Remark B.18]{LM2};
and for functoriality
(\cite[Section 13]{LM}), but it is not needed for this, see
\cite[Subsection B.7]{LM2}.  
The same category $\mathcal D$ was used in \cite[Definition 12.1]{LM2}, but
freeness is not needed there either.
We shall therefore assume that for objects $(C,D,\beta,\varphi)$ in
$\Aa^{{\mathbb{Z}}[\pi_1 X]}_{\mathrm{rel}}$ the chain complex $C$ is 
homotopy finite but not necessarily free.}
The construction in \cite[Example 12.2(i)]{LM} generalizes to give a map
\[
\ad^{{\mathbb{Z}}[\pi_1 X]}
\to
\ad^{{\mathbb{Z}}[\pi_1 X]}_{\mathrm{rel}}
\]
and 
the proof of \cite[Proposition 13.3]{LM2} generalizes to show that this induces
an isomorphism of bordism groups.  Thus we have an isomorphism
\begin{equation}
\label{e10}
L^n({\mathbb{Z}}[\pi_1 X])
=
(\Omega^{{\mathbb{Z}}[\pi_1 X]})_n
\xrightarrow{\cong}
(\Omega^{{\mathbb{Z}}[\pi_1 X]}_{\mathrm{rel}})_n
\end{equation}
for $n\geq 0$,
\end{remark}

Because of the isomorphism \eqref{e10}, 
we can 
construct $\sigma^*_\ip(X)$ by giving a suitable  element of 
$(\Omega^{{\mathbb{Z}}[\pi_1 X]}_{\mathrm{rel}})_n$.
For this in turn it suffices (by \cite[Definition 4.2]{LM}) to construct an 
element of 
$\ad^{{\mathbb{Z}}[\pi_1
X]}_\mathrm{rel}(*)$, which we do as follows.
For each $x\in X$
let $\tilde{X}_x$ be the universal cover constructed in the usual way as
equivalence classes of paths starting at $x$. Let $A_X$ be the
chain complex of 
left $\Z[\pi_1 X]$ modules with 
\[
(A_X)_x=IS^{\bar
0}_*(\tilde{X}_x;\Z), 
\]
and let $\ZZ$ be the constant left $\Z[\pi_1 X]$ module with value $\Z$.
Choose a cycle
$\xi\in \ZZ\otimes_{\Z[\pi_1 X]} A_X$
which maps to a representative for the fundamental class 
$\Gamma_X\in IH^{\bar 0}_n(X;\Z)$; this is always possible by
\cite[Proposition 6.1.3]{friedmanmcclure}.

\begin{definition}
\label{fd1}
Let $(C_X,D_X,\beta_X,\varphi)$ be defined as follows.

$\bullet$ $C_X$ is the chain complex of $\Z[\pi_1 X]$ modules with 
$(C_X)_x=IS^{\bar{n}}_*(\tilde{X}_x;\Z)$.

$\bullet$
Let $E_X$ be the chain complex of left $\Z[\pi_1 X]$ modules with 
\[
(E_X)_x=IS^{Q_{\bar n,\bar 
n}}_*(\tilde X_x\times \tilde X_x;\Z),
\]
where $\Z[\pi_1(X,x)]$ acts diagonally;
then 
$D_X$ is the chain complex of abelian groups 
\[
\ZZ\otimes_{\Z[\pi_1 X]} E_X,
\]
with the evident $\Z/2$ 
action.

$\bullet$
$\beta_X$ is the map
\[
(C_X)^t\otimes_{{\mathbb{Z}}[\pi_1 X]}  C_X\to D_X
\]
determined by the composites
\begin{multline*}
IS^{\bar{n}}_*(\tilde{X}_x;\Z)^t
\otimes_{{\mathbb{Z}}[\pi_1(X,x)]}
IS^{\bar{n}}_*(\tilde{X}_x;\Z)
\cong
\Z\otimes_{{\mathbb{Z}}[\pi_1(X,x)]}
(IS^{\bar{n}}_*(\tilde{X}_x;\Z)
\otimes
IS^{\bar{n}}_*(\tilde{X}_x;\Z))
\\
\xrightarrow{1\otimes\times}
\Z\otimes_{\mathbb{Z}[\pi_1(X,x)]}IS^{Q_{\bar n,\bar n}}_*(\tilde X_x\times 
\tilde 
X_x);\Z),
\end{multline*}
where the first isomorphism is given by Remark \ref{tens}.

$\bullet$
$\varphi\in (D_X)^{\Z/2}$ is the image of $\xi$ under the map 
$\ZZ\otimes_{\Z[\pi_1 X]}
A_X\to \ZZ\otimes_{\Z[\pi_1 X]} E_X$ induced by the diagonal maps
\[
IS^{\bar
0}_n(\tilde{X}_x;\Z)
\to
IS^{Q_{\bar n,\bar n}}_n(\tilde 
X_x\times \tilde X_x;\Z).
\]
\end{definition}

\begin{lemma}
\label{fl1}
$(C_X,D_X,\beta_X,\varphi)$ is an element of
$\ad^{{\mathbb{Z}}[\pi_1 X]}_\mathrm{rel}(*)$.
\end{lemma}

\begin{proof}
First we need to show that $C_X$ is homotopy finite.
Fix $x\in X$. 
Proposition \ref{finite} gives a quasi-isomorphism over
${\mathbb{Z}}[\pi_1(X,x)]$ from $(C_X)_x$ to a finite chain complex $A$ over
${\mathbb{Z}}[\pi_1(X,x)]$.
$(C_X)_x$ is chain homotopy equivalent to a free complex over
${\mathbb{Z}}[\pi_1(X,x)]$ by Proposition \ref{f1}(i), and 
\cite[Exercise IV.4.2]{HS} says that a quasi-isomorphism between free complexes
is a chain homotopy equivalence, so the quasi-isomorphism $(C_X)_x\to A$
is a chain homotopy equivalence over ${\mathbb{Z}}[\pi_1(X,x)]$.  Now $A$ 
extends to a finite chain complex $B$ over $\Z[\pi_1 X]$, 
and the chain homotopy equivalence $(C_X)_x\to A$ extends to a
chain homotopy equivalence $C_X\to B$ over $\Z[\pi_1 X]$.

$\beta_X$ is a quasi-isomorphism by Proposition \ref{P: tensor flat} and Remark
\ref{x2}.

% The same is true for $E_X$, and hence $D_X$ is chain homotopy equivalent over
% $\Z$ to a finite chain complex over $\Z$.

It remains to show that the slant product with
$(\beta_*)^{-1}([\varphi])$ is an isomorphism.
For this we observe that the proof of \cite[Proposition 5.17]{friedmanmcclure}
goes through with $\Z$ coefficients instead of $F$ coefficients if we use
Theorem \ref{T: universal duality}
in place of \cite[Theorem 4.1]{friedmanmcclure}.
Inspection of the definitions in \cite[Subsection 5.4]{friedmanmcclure} (using the
fact that the map denoted $\Upsilon$ there is the same as $\beta_*$)
shows that (after replacing $F$ coefficients by $\Z$ coefficients)
the element denoted in the proof of \cite[Proposition 5.17]{friedmanmcclure} 
by $\psi_*(\iota)$ is the same as $(\beta_*)^{-1}([\varphi])$, so Lemma 
\ref{fl1} follows from the $\Z$-version of \cite[Proposition 
5.17]{friedmanmcclure}.
\end{proof}

Next observe that
the bordism class of $(C_X,D_X,\beta_X,\varphi)$ is
independent of the choice of $\xi$ by \cite[Remark 
5.6]{friedmanmcclure}.

\begin{definition}
\label{d5}
Let $X$ be a compact oriented IP-space of dimension $n$.  Then
$\sigma^*_\ip(X)\in L^n({\mathbb{Z}}[\pi_1 X])$ is the element which maps to 
the bordism class of 
$(C_X,D_X,\beta_X,\varphi)$ under the isomorphism \eqref{e10}.
%$(\Omega^{{\mathbb{Z}}[\pi_1 X]}_{\mathrm{rel}})_n\cong
%L^n({\mathbb{Z}}[\pi_1 X])$.
\end{definition}

\begin{remark} \label{rem.ipsymsiganalogwittsymsig}
The properties of the symmetric signature given in \cite[Subsection
5.5]{friedmanmcclure} remain valid (with $\Z$ coefficients instead of $F$
coefficients) for $\sigma^*_\ip$, with the same proofs,
provided that Proposition \ref{f1}(i), Theorem \ref{T: universal duality}, and
Theorem \ref{T: univ lef} are used instead of \cite[Proposition
5.15, Theorem 4.1, Theorem 4.5]{friedmanmcclure}.
For example, Prop. 5.20 of \cite{friedmanmcclure} asserts that if $n$ is
divisible by $4$, then the composition
\[ L^n (\rat [G]) \longrightarrow L^n (\rat)\stackrel{\cong}{\longrightarrow} W(\rat), \]
where $W(\rat)$ is the Witt group of $\rat$, takes the symmetric signature
of a Witt space $X$ with $\pi_1 (X)=G$ to the Witt class of the intersection
form on the middle dimensional, middle perversity intersection homology of $X$.
For IP-spaces $X,$ the analogous statement is:  
If $n$ is divisible by $4$, then the composition
\[ L^n (\intg [G]) \longrightarrow L^n (\intg) \stackrel{\cong}{\longrightarrow} \intg \]
takes $\sigma^*_\ip (X)$
to the ordinary signature of the intersection
form of $X$.
\end{remark}

\begin{remark} \label{rem.htpyinvsymsig}
We have only described the \emph{universal} symmetric signature of IP-spaces.
More generally, one can similarly construct the symmetric signature
$\sigma^*_\ip (r)\in L^n (\intg [G])$ of any reference map $r:X\to BG$ with $X$ IP and $G$ any discrete group.
Instead of using the intersection chains of the universal cover of $X$, one uses
the intersection chains of the cover of $X$ induced by $r$. This is done in
\cite{friedmanmcclure} for global Witt maps $r$. For such maps, 
\cite[Thm. 5.23]{friedmanmcclure} asserts that $\sigma^*_{\operatorname{Witt}} (r\circ f)=
\sigma^*_{\operatorname{Witt}} (r)$ for orientation preserving stratified homotopy equivalences $f:X'\to X$ between
Witt spaces. Similarly for orientation preserving stratified homotopy equivalences $f$ between IP-spaces,
\begin{equation} \label{equ.htpyinvipsymsig}
\sigma^*_{\ip} (r\circ f)=
\sigma^*_{\ip} (r). 
\end{equation}
\end{remark}

\section{The $L$-theory fundamental class}
\label{s8}

For an $n$-dimensional compact oriented topological
manifold $M$, Ranicki constructs an $L$-theory
fundamental class $[M]_{\mathbb{L}} \in
\syml_n (M)$ (\cite[Section 16]{ranicki}) which plays an important role in
surgery theory. It is an oriented homeomorphism invariant
whose image under the assembly map
\[ \syml_n (M) \rightarrow L^n (\intg [\pi_1 M]) \]
is the symmetric signature $\sigma^\ast (M)$.
$[M]_{\mathbb{L}}$ can be constructed in the following way.
There is an equivalence 
\[
M\STop\to \bQ_\STop 
\]
in the stable category, where $M\STop$ is the Thom spectrum and $\bQ_\STop$ is
the Quinn spectrum; see \cite[Appendix B]{LM} for details.
There is
an ad theory $\ad_\STopFun$ which is related to $\ad_\STop$ in the same way
that $\ad_\ipFun$ is related to $\ad_\ip$ (see the end of \cite[Section
8]{LM}), and the map given by forgetting the chain representative 
is an equivalence
\[
\bQ_\STopFun\xrightarrow{\simeq}\bQ_\STop.
\]
The symmetric signature gives a map\footnote{ \label{aa46}
Specifically, the end of \cite[Section 8]{LM} gives a morphism
$
\Aa_\mathrm{StopFun}\to \Aa_{e,*,1},
$
\cite[Section 10]{LM} gives a morphism
$
\sig: \Aa_{e,*,1}\to \Aa^\Z,
$
(which was denoted $\Sig$ in \cite{LM})
and the map \eqref{aa45} is, by definition, induced by the 
the composite of these two morphisms.} 
\begin{equation}
\label{aa45}
\sig:\bQ_\STopFun\to
\bQ^\Z_{\geq 0}={\mathbb{L}}^\bullet(\Z),
\end{equation}
so we have a map in the stable category
\[
\Sig:\bQ_\STop\to {\mathbb{L}}^\bullet(\Z).
\]
Now if $M$ is an  $n$-dimensional compact oriented topological
manifold then the identity map $M\to M$ represents an element 
in 
$(\Omega_\STop)_n(M)$.
Let $[M]_\STop$ denote the image of this element under the isomorphism
$(\Omega_\STop)_n(M)
\cong
M\STop_n(M)$.
We have

\begin{lemma}
\label{aa47}
The image of $[M]_\STop$ under the composite
\begin{equation}
\label{e12}
%(\Omega_\STop)_n(M)
%\cong
M\STop_n(M)\to (\bQ_\STop)_n(M)
\xrightarrow{\Sig} {\mathbb{L}}^\bullet(\Z)_n(M)
=\syml_n(M)
\end{equation}
is $[M]_{\mathbb{L}}$.
\end{lemma}

\begin{proof}
The construction of $[M]_{\mathbb{L}}$ is given in \cite[Propositions 16.15 and
16.16(ii)]{ranicki}.   The class $[M]_{\mathbb{L}}$ is denoted by
$\mathrm{sign}_M^{\textbf{L}^\bullet}(M)$ in \cite[Definition 8.11]{kmm}, and 
\cite[Proposition 13.3]{kmm} says that it is $S$-dual to the orientation 
class $u_{\textbf{L}^\bullet}(\nu)\in {\mathbb{L}}^\bullet(\Z)^k(T(\nu))$ 
(where $T(\nu)$ is the Thom complex of the normal bundle and $k$ is the 
dimension of $\nu$). 
By the proof of \cite[Proposition 13.2]{kmm}, $u_{\textbf{L}^\bullet}(\nu)$
is obtained by applying the composite
\[
M\STop\to\bQ_\STop\xrightarrow{\Sig} {\mathbb{L}}^\bullet(\Z)
\]
to the orientation class  $u_{M\STop}(\nu)\in M\STop^k(T(\nu))$ 
represented by the map
\[
T(\nu)\to \STop_k. 
\]
Thus it suffices to show that the $S$-dual of $u_{M\STop}(\nu)$ is 
the element we have called $[M]_\STop$.
The $S$-dual of $u_{M\STop}(\nu)$ is represented by the
composite
\[
S^{n+k}\xrightarrow{p} T(\nu)\xrightarrow{\delta} T(\nu)\wedge M_+
\to \STop_k\wedge M_+,
\]
where $p$ is the Pontrjagin-Thom collapse and $\delta$ is the Thom diagonal
(see for example \cite[Proposition 2.2]{rezk},  with the maps $\eta$ and 
$\phi$ given on page 7 of \cite{rezk})
and by \cite[page 19 and Example 6 on page 43]{stong} this composite 
represents 
$[M]_\STop$.
\end{proof}

\begin{remark}
The composite
\[
M\STop\to\bQ_\STop\xrightarrow{\Sig} {\mathbb{L}}^\bullet(\Z)
\]
is the {\it Sullivan-Ranicki orientation} (compare \cite[Remark 1.4]{LM2}).
\end{remark}

Our goal in this section is to prove

\begin{thm}
\label{t3}
For an $n$-dimensional compact oriented IP-space $X$ there is 
a fundamental class $[X]_{\mathbb{L}} \in
\syml_n (X)$ with the following properties:

{\rm (i)} $[X]_{\mathbb{L}}$ is an oriented PL homeomorphism invariant.

{\rm (ii)} If $X$ is a PL manifold then $[X]_{\mathbb{L}}$ is the same as the
fundamental class constructed by Ranicki.
\end{thm}

(We will show in the next section that $[X]_{\mathbb{L}}$ assembles to the
symmetric signature $\sigma_\ip^*(X)$ given by Definition \ref{d5}.)

\begin{remark}
For Witt spaces, a different method for constructing a fundamental class is
described in \cite{csw}.
\end{remark}

The rest of the section is devoted to the proof of Theorem \ref{t3}.  
We 
begin with the construction of $[X]_{\mathbb{L}}$. 

Define a category $\Aa_{\ip,Z}$ as follows.
An object of $\Aa_{\ip,Z}$ is an object $X$ of $\Aa_\ip$ together with a map of
topological spaces $X\to Z$.  A morphism from $X\to Z$ to $X'\to Z$ is a
morphism $X\to X'$ in $\Aa_\ip$ for which the diagram
\[
\xymatrix{
X
\ar[rr]
\ar[rd]
&
&
X'
\ar[ld]
\\
&Z&
}
\]
commutes.
There is a forgetful functor
\[
\Upsilon: \Aa_{\ip,Z}\to \Aa_\ip,
\]
and we define
$\ad_{\ip,Z}^k(K)\subset \pre_{\ip,Z}^k(K)$ to be the set of functors $F$ for
which the composite $\Upsilon\circ F$ is in $\ad_\ip^k(K)$.  The proof of
Theorem \ref{t1} shows that this is an ad theory.

Now we have a
functor $\Phi$ from spaces to spectra with 
\[
\Phi(Z)=
\bQ_{\ip,Z}.
\]
By Remark \ref{aa41},
the homotopy groups of $\bQ_{\ip,Z}$ are the same as the IP bordism
groups $(\Omega_\ip)_*(Z)$ defined by Pardon (\cite[Section 5]{pardon}).  
Pardon
proves that $(\Omega_\ip)_*$ is a homology theory, and in particular this shows
that $\Phi$ is homotopy invariant in the
sense of \cite[Section 1]{wwa}.  We therefore have an assembly map
\[
\alpha:Z_+\wedge \Phi(*)
\to
\Phi(Z)
\]
by \cite[Theorem 1.1 and Observation 1.2]{wwa} for all spaces $Z$ which have
the homotopy type of a CW complex.

\begin{thm}
\label{t2}
$\alpha$ is a weak equivalence.
\end{thm}
The proof is deferred to Section \ref{pt2}.
Now we can define $[X]_{\mathbb{L}}$.

\begin{definition} \label{def.xipxl}
Let $X$ be an $n$-dimensional compact oriented IP-space.

(i) Let $[X]_\ip\in \pi_n(\bQ_{\ip,X})$ be the image of
the class of the identity map $X\to X$ in $(\Omega_\ip)_n(X)$ under the
isomorphism
$
(\Omega_\ip)_n(X)
\cong
\pi_n(\bQ_{\ip,X}).
$

(ii) Let
$[X]_{\mathbb{L}}$ be the image of $[X]_\ip$ under the composite
\begin{equation}
\label{e13}
%(\Omega_\ip)_n(X)
%\cong
\pi_n(\bQ_{\ip,X})
\xleftarrow[\cong]{\alpha}
(\bQ_\ip)_n(X)
\xrightarrow{\Sig}
(\bQ^\Z_{\mathrm{rel},\geq 0})_n(X)
\xleftarrow{\cong}
\syml_n(X),
\end{equation}
where the last map is the isomorphism \eqref{e10}.
\end{definition}

It remains to prove parts (i) and (ii) of Theorem \ref{t3}.
For part (i) it suffices to show that if $f:X\to X'$ is an oriented PL
homeomorphism then $f_*([X]_\ip)=[X']_\ip$, and this in turn follows from the
fact that the map
\[
(I\times X)\cup_{1\times X}\, X'
\to
X',
\]
which is the identity on $X'$ and takes $(t,x)$ to $f(x)$, is a bordism between
$f$ and the identity map of $X'$.

For part (ii), we need to compare the composites \eqref{e12} and \eqref{e13} 
for $X$ a PL manifold $M$.  \cite[Section 14]{LM2} gives a morphism
\[
\Sigr:\Aa_{e,*,1}\to\Aa^\Z_{\mathrm{rel}}
\]
and as in Footnote \ref{aa46} (with $\Aa^\Z_{\mathrm{rel}}$ instead of $\Aa^\Z$)  this gives a map 
\[
\sig:\bQ_\STopFun\to \bQ^\Z_{\mathrm{rel},\geq 0}
\]
which in turn gives
\[
\Sig: \bQ_\STop\to \bQ^\Z_{\mathrm{rel},\geq 0}.
\]
By \cite[Proposition 14.4]{LM2}, 
\eqref{e12} is equal to the composite 
\begin{equation}
\label{e14}
%(\Omega_\STop)_n(M)
%\cong
M\STop_n(M)\to (\bQ_\STop)_n(M)
\xrightarrow{\Sig}
(\bQ^\Z_{\mathrm{rel},\geq 0})_n(M)
\xleftarrow{\cong}
\syml_n(M).
\end{equation}
Next we observe that for each space $Z$ there is an ad theory $\ad_{\STop,Z}$
defined analogously to $\ad_{\ip,Z}$.  We get a functor $\Xi$ from spaces to
spectra by letting
\[
\Xi(Z)=\bQ_{\STop,Z},
\]
and we have 
\[
\pi_*\Xi(Z)\cong(\Omega_\STop)_*(Z).
\]
$\Xi$ is homotopy invariant, because $\pi_*\Xi(Z)$ is a homology theory by
\cite[Chapter 4]{cf},
and the proof of Theorem \ref{t2} shows that
the assembly map for $\Xi$ is a weak equivalence. 

\begin{lemma}
\label{l3}
The composite \eqref{e14} is equal to the composite
\[
% (\Omega_\STop)_n(M)
% \cong
\pi_n\bQ_{\STop,M}
\xleftarrow{\alpha}
(\bQ_\STop)_n(M)
\xrightarrow{\Sig}
(\bQ^\Z_{\mathrm{rel},\geq 0})_n(M)
\xleftarrow{\cong}
\syml_n(M)
\]
where $\alpha$ is the assembly map for the functor $\Xi$.
\end{lemma}
 
We defer the proof for a moment.
Next we observe that everything we have said about topological manifolds and 
STop is also valid for PL manifolds and
SPL.
We write $\Xi'$ for the PL analog of $\Xi$. Now consider the
following diagram.
\[
\xymatrix{
\Xi(M)
&
\bQ_\STop\wedge M_+
\ar[rd]^{\Sig\wedge 1}
\ar[l]_-\alpha
&
\\
\Xi'(M)
\ar[u]
\ar[d]
&
\bQ_{\mathrm{SPL}}\wedge M_+
\ar[l]_-\alpha
\ar[r]^-{\Sig\wedge 1}
\ar[u]
\ar[d]
&
\bQ^\Z_{\mathrm{rel},\geq 0}\wedge M_+
\\
\Phi(M)
&
\bQ_\ip\wedge M_+
\ar[l]_-\alpha
\ar[ru]^{\Sig\wedge 1}
&
}
\]
The squares commute up to homotopy by the naturality of the assembly map, and 
the triangles commute by the definition of the maps $\Sig$.
The upper left-hand vertical map takes $[M]_\mathrm{PL}$ to $[M]_\STop$, and 
(by Lemma \ref{l3}) the upper composite takes $[M]_\STop$ to the fundamental 
class constructed by Ranicki.  The lower left-hand vertical map takes
$[M]_\mathrm{PL}$ to $[M]_\ip$, and the lower composite takes $[M]_\STop$ to
the fundamental class given by Definition \ref{def.xipxl}(ii).
\qed

\begin{proof}[Proof of Lemma \ref{l3}]
It suffices to show that the diagram
\begin{equation}
\label{e15}
\xymatrix{
(\Omega_\STop)_n(Z)
\ar[d]_\cong
&
M\STop_n(Z)
\ar[l]_i^\cong
\ar[d]_\cong^j
\\
\pi_n \bQ_{\STop,Z}
&
(\bQ_\STop)_n(Z)
\ar[l]^-\alpha
}
\end{equation}
commutes, where $Z$ is a space, $i$ is the standard isomorphism and $j$ is 
given by \cite[Appendix B]{LM}.\footnote{
The referee has asked us to remind the reader that \cite[Appendix B]{LM}
uses the folk theorem that that there is an isomorphism
$
\Omega_*(\STop)\to \pi_*(M\STop)
$
whose construction is similar to that on pages 19--20 of \cite{stong}.
}

First we recall the definition of $i$ (cf.\ \cite[pages 224--5]{dk}).  
The $k$-th space of the spectrum
$M\STop\wedge Z_+$ is $T\STop_k\wedge Z_+$, where $T\STop_k$ is the Thom
space.  The inclusion of the 0-section gives an embedding
\[
B\STop_k \to T\STop_k.
\]
Given a map
$f:S^{n+k}\to T\STop_k\wedge Z_+$, there is a homotopic map $f'$ 
for which the composite
\[
S^{n+k}\xrightarrow{f'} T\STop_k\wedge Z_+\xrightarrow{p_1} T\STop_k
\]
(where $p_1$ is the projection)
is transverse to the 0-section. Then the oriented topological manifold 
$(p_1\circ f')^{-1}(B\STop_k)$ is equal to
$(f')^{-1}(B\STop_k\times Z)$, and
$i$ takes the homotopy class of $f$ to the bordism class of the composite
\[
(f')^{-1}(B\STop_k\times Z)
\xrightarrow{f'}
B\STop_k\times Z
\xrightarrow{p_2}
Z.
\]

Next we observe that $i$ can be described using maps of spectra.
Let $S_\bullet(T\STop_k\wedge Z_+)$ be the singular complex, and let
$S_\bullet^\pitchfork(T\STop_k\wedge Z_+)$ be the sub-semisimplicial set
consisting of maps $g:\Delta^n\to T\STop_k\wedge Z_+$ for which the restriction
of $p_1\circ g$ to each face is transverse to the 0-section (cf.\
\cite[Appendix B]{LM}).  Let 
\[
(M\STop\wedge Z_+)^\pitchfork
\]
be the spectrum whose $k$-th space is the geometric realization 
$|S_\bullet^\pitchfork(T\STop_k\wedge Z_+)|$.  
Given a simplex $g$ of $S_\bullet^\pitchfork(T\STop_k\wedge Z_+)$, we obtain 
an element of $\ad_{\STop,Z}^k(\Delta^n)$ by taking each oriented simplex
$(\sigma,o)$ to $(g|_\sigma)^{-1}(B\STop\times Z)$;
this gives a natural transformation
\[
I: (M\STop\wedge Z_+)^\pitchfork
\to
\bQ_{\STop,Z}.
\]
Next, transversality implies that the
map 
\[
|S_\bullet^\pitchfork(T\STop_k\wedge Z_+)|
\to
|S_\bullet(T\STop_k\wedge Z_+)|
\]
is a weak equivalence (because by \cite[Lemma 16.3, Definition 3.6, Example
1.5 and Lemma 1.5]{may} each element of $\pi_n$ of the target is 
represented by an $n$-simplex whose faces are at the basepoint, and such a 
simplex can be deformed to one that is transverse, and similarly for 
homotopies), and 
so the map
\[
(M\STop\wedge Z_+)^\pitchfork
\to
M\STop\wedge Z_+
\]
is a weak equivalence.  The isomorphism
$i$ is induced by the composite
\[
M\STop\wedge Z_+
\xleftarrow{\simeq}
(M\STop\wedge Z_+)^\pitchfork
\xrightarrow{I}
\bQ_{\STop,Z}
\]
together with the isomorphism $\pi_n\bQ_{\STop,Z}\cong (\Omega_\STop)_n(Z)$.
%Hence the functor which takes 
%$Z$ to $(M\STop\wedge Z_+)^\pitchfork$ is a homotopy functor, and there is an
%assembly map
%\[
%M\STop^\pitchfork\wedge Z_+
%\xrightarrow{\alpha}
%(M\STop\wedge Z_+)^\pitchfork.
%\]

Now
consider the diagram
\[
\xymatrix{
M\STop\wedge Z_+
&
M\STop\wedge Z_+
\ar[l]_-=
\\
(M\STop\wedge Z_+)^\pitchfork
\ar[u]^\simeq
\ar[d]_I
&
M\STop^\pitchfork\wedge Z_+
\ar[u]^\simeq
\ar[l]_-\alpha
\ar[d]_I
\\
\bQ_{\STop,Z}
&
\bQ_\STop\wedge Z_+.
\ar[l]_-\alpha
}
\]
This homotopy commutes by naturality of the assembly map (since the assembly 
map for the functor $M\STop\wedge Z_+$ is the identity map).  On passage to 
homotopy groups, 
% the left-hand vertical composite induces the map $i$ of 
% diagram \eqref{e15}, and 
the right-hand vertical 
composite induces the map 
$j$ of diagram \eqref{e15}, and this shows that diagram \eqref{e15} commutes 
as required. 
\end{proof}

In \cite{pardon}, Pardon computes the IP bordism groups of a point to be
\begin{equation} \label{equ.pardoncompip}
\Omega^\ip_n (\pt) \cong
\begin{cases} \intg, & n\equiv 0(4) \\
 \intg/_2, & n\equiv 1(4), n>1 \\
 0, & \text{ otherwise.}
\end{cases} 
\end{equation}
The isomorphisms are given by the
signature (when $n\equiv 0(4)$) and the
de Rham invariant (when $n\equiv 1(4)$). These groups are very close to
\[
L^n (\intg) \cong
\begin{cases} \intg, & n\equiv 0(4) \\
 \intg/_2, & n\equiv 1(4) \\
 0, & \text{ otherwise.}
\end{cases} 
\]

\begin{thm}
The map $\Omega^\ip_n (\pt)\to \syml_n (\pt) = L^n (\intg)$, given by
(\ref{e13}) on a point, is an isomorphism for all
$n\not= 1$.
\end{thm}
\begin{proof}
If $n\equiv 2,3 (4)$, then both $\Omega^\ip_n (\pt)$ and $L^n (\intg)$ vanish and the claim holds.

Suppose that $n\equiv 0(4)$. Then by (\ref{equ.pardoncompip}), the signature is an isomorphism
$\Omega^\ip_n (\pt)\cong \intg$. The signature of a symmetric Poincar\'e chain complex is an
isomorphism $L^n (\intg)\cong \intg$. We shall show that our map
$\Omega^\ip_n (\pt)\to L^n (\intg)$ sends a generator to a generator. The complex projective
space $\cpnt$ represents a generator $[\cpnt]\in \Omega^\ip_n (\pt)$ since it is an IP-space
and has signature $1$.
Let $f$ be the map $f:\cpnt \to \pt$. The diagram
\begin{equation} \label{equ.ipcpnt}
\xymatrix{
\Omega^\ip_n (\cpnt) \ar[d]_{f_*} \ar[r] & \syml_n (\cpnt) \ar[d]^{f_*} \\
\Omega^\ip_n (\pt) \ar[r] & \syml_n (\pt) 
} 
\end{equation}
commutes, since $\Omega^\ip_n (-)\to \syml_n (-)$ is a natural transformation of 
homology theories (being induced by a spectrum level map).
Since the assembly map $\alpha$ is a natural transformation of functors and
$\cpnt$ is simply connected, the diagram
\[ \xymatrix{
\syml_n (\cpnt) \ar[d]_{f_*} \ar[r]^{\alpha} & L^n (\intg) \ar@{=}[d]^{f_* = \id} \\
\syml_n (\pt) \ar[r]_{\cong}^{\alpha} & L^n (\intg)
} \]
commutes as well.
By construction, the top horizontal arrow of diagram (\ref{equ.ipcpnt}) maps
$[\id_{\cpnt}]\in \Omega^\ip_n (\cpnt)$ to the fundamental class
$[\cpnt]_{\mathbb{L}} \in \syml_n (\cpnt)$. By our Theorem \ref{t3}(iii),
this class agrees with Ranicki's fundamental $\syml$-homology class for manifolds.
The latter class is known to assemble to the Mischenko-Ranicki symmetric
signature $\sigma^* (\cpnt)$, which is a generator of $L^n (\intg)$, since
the ordinary signature of $\cpnt$ is $1$.
Now $f_* [\id_{\cpnt}] = [\cpnt]\in \Omega^\ip_n (\pt)$ and thus by commutativity,
the bottom arrows of the two diagrams must also map $[\cpnt]$ to a generator and the
claim is proved. 

Suppose that $n\equiv 1(4)$ and $n>1$. Then according to (\ref{equ.pardoncompip}), 
the de Rham invariant is an isomorphism
$\Omega^\ip_n (\pt)\cong \intg/_2$. The de Rham invariant of a symmetric Poincar\'e chain complex is an
isomorphism $L^n (\intg)\cong \intg/_2$. To show that our map
$\Omega^\ip_n (\pt)\to L^n (\intg)$ sends the generator to the generator, we can use the same argument
as in the above case, replacing the signature by the de Rham invariant and replacing the complex projective
spaces by simply connected smooth manifolds $M^n$ with nontrivial de Rham invariant. In dimension $5$,
such a manifold is given by $M^5 = \operatorname{SU}(3)/\SO(3)$. If $A$ is a manifold of dimension
congruent $1$ mod $4$ and $B$ a manifold of dimension congruent $0$ mod $4$, then the de Rham invariant
of the product $A\times B$ is the de Rham invariant of $A$ multiplied by the signature of $B$.
Thus, the manifolds $M^n = M^5 \times \mathbb{CP}^{(n-5)/2},$ all have de Rham invariant $1$ and are
simply connected. The Mischenko-Ranicki symmetric signature $\sigma^* (M^n)$ is then the generator
of $L^n (\intg)$, since $M^n$ has nontrivial de Rham invariant.
\end{proof}

\begin{remark}
In degree $n=1$, there is a discrepancy: 
$\Omega^\ip_1 (\pt)=0,$ while $L^1 (\intg)\cong \intg/_2$.
\end{remark}

\section{The assembly of the fundamental class}
\label{gs6}

The goal of this section is to prove that 
$[X]_{\mathbb{L}}$ assembles to the symmetric signature
$\sigma_\ip^*(X)$.

Before giving the precise statement (Theorem \ref{nn13}) we need some 
preliminary work. 
In order to construct the relevant assembly map, we first need a functorial 
model for $\ad^{\Z[\pi_1 Z]}$;  $\ad^{\Z[\pi_1 Z]}$ is not a functor of $Z$
because the category of modules over 
$\Z[\pi_1 Z]$ is not a functor of $Z$ (see 
\cite[Section 13]{LM} for an explanation of this issue in the context of $R$
modules).\footnote{Briefly, the point is that the obvious way of trying
to make the category of $R$ modules a functor of $R$ 
doesn't give a functor because $R_3\otimes_{R_2}
(R_2\otimes_{R_1} M)$ is canonically isomorphic to, but not the same module
as, $R_3\otimes_{R_1} M$.
}  
Accordingly, in Subsection \ref{nn4} we
give a functorial model for the category of modules over
$\Z[\pi_1 Z]$, and in Subsection \ref{mm} we use this to construct an ad theory
$\ad^{\Z[\pi_1 Z]}_\mathbb F$ which is a functor of $Z$ and has a canonical map
to $\ad^{\Z[\pi_1 Z]}$ which is an isomorphism on bordism groups.
Subsection \ref{nn10} gives the statement of the main theorem and an outline of
the proof; the details of the proof are given in Subsections
\ref{ss5}--\ref{nn3}.

\begin{convention}
Throughout this section, $Z$ will denote a path-connected space.
\end{convention}

\subsection{A functorial model for the category of modules over 
$\Z[\pi_1 Z]$}
\label{nn4}

We use the terminology and notation of 
\cite[Subsection B.5]{LM2} (which the reader should consult before
continuing)\footnote{For the benefit of the reader who doesn't want to consult
\cite[Subsection B.5]{LM2}, here are the key definitions:

We define the category of {\it schematic free $R$ modules} as follows.  An
object is
a set $\mathbb M$.  This should be thought of as representing the free $R$
module generated by $\mathbb M$, which we denote by $R\langle \mathbb
M\rangle$.  We define a map $\mathbb M\to\mathbb M'$ to be a map of $R$-modules
$ R\langle \mathbb M\rangle \to R\langle \mathbb M'\rangle$.

We define the category of {\it schematic $R$ modules} as follows.
An object of this category is
a triple $(\mathbb M,\mathbb N,T)$, where $\mathbb M$ and
$\mathbb N$ are schematic free $R$ modules and $T$ is a map
$\mathbb N\to \mathbb M$.
Such a triple should be thought of as representing the quotient of
$R\langle \mathbb M \rangle$ by the image of $T$; we
write $R\langle(\mathbb M,\mathbb N,T)\rangle$ for this quotient.
}
together with:

\begin{notation}
When possible we denote a schematic $R$ {module} by a single letter
$\mathbb 
A$ rather than by a triple $({\mathbb M}, {\mathbb N}, T)$. (Then $R\langle 
\mathbb A \rangle$ denotes the quotient of the free $R$ module $R\langle 
\mathbb M \rangle$ by the image of the map $T:R\langle 
\mathbb N \rangle \to R\langle 
\mathbb M \rangle$.)
\end{notation}

To lighten the notation, we will denote $\Z[\pi_1 Z]$ by $\RR(Z)$ and
$\Z[\pi_1(Z,z)]$ by $\RR(Z)_z$ throughout 
this subsection.

Let us define a {\it schematic $\RR(Z)$ module}
to be a pair $(z,{\mathbb A})$ where $z$ is an element of $Z$ and
${\mathbb A}$ is a schematic $\RR(Z)_z$ {module}.  If we think of $\RR(Z)_z$ 
as
a subcategory of $\RR(Z)$ and of 
$\RR(Z)_z\langle {\mathbb A}\rangle$ as a module over this subcategory, then
the Kan extension\footnote{See Appendix \ref{add}.} 
of $\RR(Z)_z\langle {\mathbb A}\rangle$ is an $\RR(Z)$ 
module which will be
denoted $\RR(Z)\langle z,{\mathbb A}\rangle$.

Here is an explicit description of 
$\RR(Z)\langle z,{\mathbb A}\rangle$.
If we denote the set of path homotopy classes of paths from $z$ to $z'$ by 
${\mathfrak p}_{z,z'}$, 
then $\RR(Z)\langle z,{\mathbb A}\rangle$ is defined on objects by letting 
$\RR(Z)\langle z,{\mathbb A}\rangle_{z'}$ be the quotient of 
\[
\mathrm{Mor}_{\RR(Z)}(z,z')\otimes
\RR(Z)_z\langle {\mathbb A}\rangle
\cong
\bigoplus_{\delta\in {\mathfrak p}_{z,z'}}
\RR(Z)_z\langle {\mathbb A}\rangle
\]
by the following equivalence relation: if $(\delta,a)$ denotes the copy of an
element $a\in \RR(Z)_z\langle {\mathbb A}\rangle$
in the $\delta$ summand and $\gamma$ is the class of a loop at $z$ then 
$(\delta\gamma,a)\sim (\delta,\gamma a)$;
\footnote{Note that we denote composition of path homotopy classes by letting 
$\delta\gamma$ be ``first $\gamma$, then $\delta$'', analogously to 
composition of functions.}
it follows that for any $\delta\in {\mathfrak p}_{z,z'}$ the map which takes
$a$ to $(\delta,a)$ is an isomorphism
$\RR(Z)_z\langle {\mathbb A}\rangle
\to
\RR(Z)\langle z,{\mathbb A}\rangle_{z'}$.
The action of
the morphisms is given as follows: if $\epsilon\in{\mathfrak p}_{z',z''}$ is
thought of as a morphism from $z'$ to $z''$ then $\epsilon$ takes $(\delta,a)$
to $(\epsilon\delta,a)$.  
% In particular, it follows from what has been said
% that there is a canonical isomorphism of $\RR(Z)_z$ modules
% \begin{equation}
% \label{nn5}
% \RR(Z)_z\langle {\mathbb A}\rangle
% \to
% \RR(Z)\langle z,{\mathbb A}\rangle_z
% \end{equation}
% which takes $a$ to $(\iota, a)$, where $\iota$ is the trivial path at $z$.

We define a {map} from $(z,{\mathbb A})$ 
to $(z',{\mathbb A}')$
to be a homomorphism of $\RR(Z)$ modules $\RR(Z)\langle z,{\mathbb 
A}\rangle\to \RR(Z)\langle z',{\mathbb A'}\rangle$.

\begin{lemma}
\label{rrrr}
The functor from 
schematic $\RR(Z)$ {modules} to $\RR(Z)$ modules which takes
$(z,{\mathbb A})$ 
to $\RR(Z)\langle z,{\mathbb A}\rangle$ is an equivalence of categories.
\end{lemma}

\begin{proof}
The functor is the identity on morphism sets, so it's only necessary to show
that every $\RR(Z)$ module $\M$ is isomorphic to one of the form $\RR(Z)\langle
z,{\mathbb A}\rangle$.  Fix $z\in Z$.  By \cite[Lemma B.19]{LM2}, there is an 
$\mathbb A$ with an isomorphism $\RR(Z)_z\langle \mathbb A\rangle\to \M_z$, and
the explicit description of $\RR(Z)\langle
z,{\mathbb A}\rangle$ given above shows that this isomorphism induces an
isomorphism $\RR(Z)\langle
z,{\mathbb A}\rangle\to \M$.
\end{proof}

% The {\it zero object} in the category of schematic $\Z[\pi_1 Z]$ modules is the
% object $(\emptyset,\emptyset,T)$, where $T$ is the only possible map.

A {\it schematic chain complex} $\mathbb C$ over $\Z[\pi_1 Z]$ is a sequence of
schematic $\Z[\pi_1 Z]$ modules and maps. We write $\Z[\pi_1 Z]\langle \mathbb
C\rangle$ for the corresponding sequence of $\Z[\pi_1 Z]$ modules and maps.  A
map $\mathbb C\to \mathbb C'$ of schematic chain complexes is a map of
$\Z[\pi_1 Z]$ chain complexes $\Z[\pi_1 Z]\langle \mathbb
C\rangle \to \Z[\pi_1 Z]\langle \mathbb
C'\rangle$.

Now let $g:Z_1\to Z_2$ be a continuous function. 

\begin{notation}
\label{nn6}
For each $z\in Z$ let
\[
{g}_z:\RR(Z_1)_z\to \RR(Z_2)_{g(z)}
\]
be the induced homomorphism.
Let 
\[
g_*:\RR(Z_1)\to \RR(Z_2)
\]
be the induced functor.  
\end{notation}

We want to show that $g$ induces a functor
$g_\sch$ from
the category of schematic $\RR(Z_1)$ modules to the category of schematic
$\RR(Z_2)$ modules.  First we define $g_\sch$ on objects:
for a schematic $\RR(Z_1)$ {module} $(z,\mathbb A)$ we define 
$g_\sch(z,\mathbb A)$ to be
the schematic $\RR(Z_2)$ {module} $(g(z),({g}_z)_\sch{\mathbb A})$.  
To define ${g}_\sch$ on maps we need a lemma.

\begin{lemma}
\label{aa35}
There is a canonical isomorphism
\[
\RR(Z_2)\langle {g}_\sch(z,\mathbb A)\rangle
\cong
\mathrm{Kan}_{g_*} \RR(Z_1)\langle z,\mathbb A\rangle
\]
\end{lemma}

\begin{proof}
Let $i:\RR(Z_1)_z\to \RR(Z_1)$ and $j:\RR(Z_2)_{g(z)}\to \RR(Z_2)$ be the
inclusions.  We have
\begin{multline*}
\RR(Z_2)\langle {g}_\sch(z,\mathbb A)\rangle
=
\RR(Z_2)\langle g(z), ({g}_z)_\sch\mathbb A\rangle
=
\mathrm{Kan}_j(\RR(Z_2)_{g(z)}\langle ({g}_z)_\sch\mathbb A)\rangle)
\\
\cong
\mathrm{Kan}_j(\mathrm{Kan}_{ g_z} (\RR(Z_1)_z\langle \mathbb A \rangle))
\cong
\mathrm{Kan}_{j\circ  g_z}(\RR(Z_1)_z\langle \mathbb A \rangle)
\\
=
\mathrm{Kan}_{ g_*\circ i}(\RR(Z_1)_z\langle \mathbb A\rangle)
\cong
\mathrm{Kan}_{ g_*}(\mathrm{Kan}_i(\RR(Z_1)_z\langle \mathbb A \rangle))
\\
=
\mathrm{Kan}_{ g_*} \RR(Z_1)\langle z,\mathbb A\rangle,
\end{multline*}
where the equalities in the first line are definitions,
the first isomorphism in the second line is 
\cite[Equation (B.2) in Subsection B.5]{LM2}
and
Example \ref{aa29} and 
the second is Equation \eqref{aa28},
the first equality in the third line is obvious and the isomorphism is
Equation \eqref{aa28}, and the last equality follows from the definition of
$\RR(Z_1)\langle z,\mathbb A\rangle$.
\end{proof}

Now for a {map}  
\[
t:
(z,{\mathbb A})
\to
(z',{\mathbb A}')
\]
we define 
\[
{g}_\sch t:
{g}_\sch(z,{\mathbb A})
\to
{g}_\sch(z',{\mathbb A}')
\]
to be the {map} determined by the diagram
\[
\xymatrix{
\RR(Z_2)\langle {g}_\sch(z,{\mathbb A})\rangle
\ar[r]^-{{g}_\sch t}
\ar[d]_\cong
&
\RR(Z_2)\langle{g}_\sch(z',{\mathbb A}')\rangle
\ar[d]_\cong
\\
\mathrm{Kan}_{g_*} \RR(Z_1)\langle z,\mathbb A\rangle
\ar[r]^-{\mathrm{Kan}_{g_*}(t)}
&
\mathrm{Kan}_{g_*} \RR(Z_1)\langle z',{\mathbb A}'\rangle
}
\]
where the bottom arrow is induced by $t$.

With these definitions,  ${g}_\sch$ is a functor from schematic $\RR(Z_1)$ {modules} to 
schematic $\RR(Z_2)$ {modules}.
If $h$ is a continuous function $Z_2\to Z_3$ we have $(h\circ 
g)_\sch={h}_\sch\circ
{g}_\sch$, so the category of schematic $\RR(Z)$ {modules} is a functor of $Z$.

\subsection{A functorial model for $\ad^{\Z[\pi_1 Z]}$}
\label{mm}

\begin{notation}
\label{nn22}
As usual, let $W$ be the standard resolution of $\Z$ by $\Z[\Z/2]$ modules.
\end{notation}

\begin{definition}
(i) A {\it schematic quasi-symmetric complex of dimension $n$ over $\Z[\pi_1 
Z]$} is
a pair $(\mathbb C,\varphi)$, where $\mathbb C$ is a schematic chain complex 
over $\Z[\pi_1 Z]$ for which $\Z[\pi_1 Z]\langle \mathbb C\rangle$ is 
homotopy finite and $\varphi$ is a $\Z/2$-equivariant map
\[
W\to (\Z[\pi_1 Z]\langle \mathbb C\rangle)^t\otimes_{\Z[\pi_1 Z]}
\Z[\pi_1 Z]\langle \mathbb C\rangle
\]
of graded abelian groups which raises degrees by $n$.

(ii) We define a category $\Aa^{\Z[\pi_1 Z]}_\sch$ as follows.
The
objects are the schematic quasi-symmetric complexes over $\Z[\pi_1 Z]$.
A morphism $(\mathbb C,\varphi)\to (\mathbb C',\varphi')$ is a map of 
$\Z[\pi_1 Z]$
chain complexes $f:\Z[\pi_1 Z]\langle \mathbb C\rangle\to
\Z[\pi_1 Z]\langle \mathbb C'\rangle$ such that if $\dim
\varphi=\dim\varphi'$ then $(f^t\otimes f)\circ \varphi=\varphi'$.
\end{definition}

% $\Aa^{\Z[\pi_1 Z]}_\mathbb{F}$ is a balanced $\Z$-graded category, where $i$
% takes $(\mathbb C,\varphi)$ to $(\mathbb C,-\varphi)$.

There is a morphism
\[
\Theta:\Aa^{\Z[\pi_1 Z]}_\sch
\to 
\Aa^{\Z[\pi_1 Z]}
\]
of $\Z$-graded categories which takes $(\mathbb C,\varphi)$ to $(\Z[\pi_1
Z]\langle \mathbb C\rangle,\varphi)$; this is an equivalence of categories.

\begin{definition}
\label{nn7}
A $K$-ad with values in $\Aa^{\Z[\pi_1 Z]}_\sch$ is a pre $K$ ad 
$F$ for which $\Theta\circ F$ is a $K$-ad.
\end{definition}

We write $\ad^{\Z[\pi_1 Z]}_\sch(K)$ for the set of $K$-ads with values
in $\Aa^{\Z[\pi_1 Z]}_\sch$.

\begin{prop}
\label{nn11}
{\rm{(i)}} $\ad^{\Z[\pi_1 Z]}_\sch$ is an ad theory.

{\rm{(ii)}} $\Theta$ induces a morphism of ad theories which is an isomorphism
of bordism groups.
\end{prop}

This is an easy consequence of the fact that $\ad^{\Z[\pi_1 Z]}$ is an ad
theory and \cite[Lemma B.23]{LM2}.

Next we consider functoriality.
Now let $g:Z\to Z'$ be a continuous function.  Define a functor
\[
g^\mathrm{sym}_\sch:
\Aa^{\Z[\pi_1 Z]}_\sch
\to
\Aa^{\Z[\pi_1 Z']}_\sch
\]
by
\[
g^\mathrm{sym}_\sch(\mathbb C,\varphi)
=
(g_\sch(\mathbb C),\psi),
\]
where $\psi$ is the composite
\begin{multline*}
W\xrightarrow{\varphi} (\Z[\pi_1 Z]\langle \mathbb C\rangle)^t\otimes_{\Z[\pi_1 Z]}
\Z[\pi_1 Z]\langle \mathbb C\rangle
\xrightarrow{\iota\otimes\iota}
\\
(\mathrm{Kan}_{g_*}(\Z[\pi_1 Z]\langle \mathbb C\rangle))^t
\otimes_{\Z[\pi_1 Z']}
\mathrm{Kan}_{g_*}(\Z[\pi_1 Z]\langle \mathbb C\rangle)
\cong
\\
(\Z[\pi_1 Z']\langle g_\sch(\mathbb C)\rangle)^t
\otimes_{\Z[\pi_1 Z']}
\Z[\pi_1 Z']\langle g_\sch(\mathbb C)\rangle,
\end{multline*}
where $\iota$ is given in Appendix \ref{add} 
and the isomorphism is given by Lemma \ref{aa35}.

\begin{prop}
\label{nn8}
Let $g:Z\to Z'$ and 
$g':Z'\to Z''$ be continuous functions.
Then
\[
(g'g)^\mathrm{sym}_\sch
=
(g')^\mathrm{sym}_\sch
g^\mathrm{sym}_\sch.
\]
\end{prop}

\begin{proof}
First we show that $(g'g)^\mathrm{sym}_\sch
=
(g')^\mathrm{sym}_\sch g^\mathrm{sym}_\sch$ on objects.
Let $(\mathbb C,\varphi)$ be an object of $\Aa^{\Z[\pi_1 Z]}_\sch$.  Let 
$(g'g)^\mathrm{sym}_\sch(\mathbb C,\varphi)=((g'g)_\sch(\mathbb C),\psi)$
and let $(g')^\mathrm{sym}_\sch
g^\mathrm{sym}_\sch(\mathbb C,\varphi)=(g'_\sch g_\sch(\mathbb C),\psi')$.
We have already seen that $(g'g)_\sch(\mathbb C)=g'_\sch g_\sch(\mathbb C)$.
For the proof that $\psi=\psi'$, first let $(z,\mathbb A)$ be a schematic
$\Z[\pi_1 Z]$ module and consider the diagram
\[
\xymatrix{
\Z[\pi_1 Z]\langle(z,\mathbb A)\rangle
\ar[r]^-\iota
\ar[dd]_\iota
&
\mathrm{Kan}_{g_*}(\Z[\pi_1 Z]\langle(z,\mathbb A)\rangle)
\ar[r]^-\cong
&
\Z[\pi_1 Z']\langle g_\sch (z,\mathbb A)\rangle
\ar[d]_\iota
\\
&
&
\mathrm{Kan}_{g'_*} \Z[\pi_1 Z']\langle g_\sch (z,\mathbb A)\rangle
\ar[d]_\cong
\\
\mathrm{Kan}_{(g'g)_*}\Z[\pi_1 Z]\langle(z,\mathbb A)\rangle
\ar[rr]^-\cong
&
&
\Z[\pi_1 Z'']\langle g'_\sch g_\sch (z,\mathbb A)\rangle
}
\]
By definition, the upper left corner is a Kan extension, and by the universal 
property of the Kan extension (in Appendix \ref{add}) it suffices to show that
the clockwise and counterclockwise composites in the diagram agree when
restricted to $\Z[\pi_1(Z,z)]\langle \mathbb A\rangle$, and this is a routine
verification.  The proof that $\psi=\psi'$ is a straightforward consequence
of commutativity of the diagram.

Next we show that $(g'g)^\mathrm{sym}_\sch
=
(g')^\mathrm{sym}_\sch g^\mathrm{sym}_\sch$ on maps.  Let 
$f:(\mathbb C,\varphi)\to (\mathbb C',\varphi')$ be a map, which simply means
that $f$ is a map $\Z[\pi_1 Z]\langle\mathbb C\rangle\to \Z[\pi_1
Z]\langle\mathbb C'\rangle$ (with $(f^t\otimes f)\circ \varphi=\varphi'$ if
$\dim \varphi=\dim\varphi'$).
First we observe that, for any 
schematic
$\Z[\pi_1 Z]$ module
$(z,\mathbb A)$, 
the diagram
\[
\xymatrix{
\mathrm{Kan}_{(g'g)_*}\Z[\pi_1 Z]\langle(z,\mathbb A)\rangle
\ar[r]^-\cong
\ar[dd]_\cong
&
\mathrm{Kan}_{g'_*}\mathrm{Kan}_{g_*}\Z[\pi_1 Z]\langle(z,\mathbb A)\rangle
\ar[d]_\cong
\\
&
\mathrm{Kan}_{g'_*}\Z[\pi_1 Z']\langle g_\sch (z,\mathbb A)\rangle
\ar[d]_\cong
\\
\Z[\pi_1 Z'']\langle (g'g)_\sch (z,\mathbb A)\rangle
\ar[r]^-=
&
\Z[\pi_1 Z'']\langle g'_\sch g_\sch (z,\mathbb A)\rangle
}
\]
commutes, as the reader can check; 
the fact that $(g'g)^\mathrm{sym}_\sch(f)
=
(g')^\mathrm{sym}_\sch g^\mathrm{sym}_\sch(f)$ is an easy consequence of this
and the definitions.
\end{proof}

\begin{prop}
\label{nn9}
$g^\mathrm{sym}_\sch$ takes ads to ads.
\end{prop}

\begin{proof}
Let $F\in \ad^{\Z[\pi_1 Z]}_\sch(K)$, and write
\[
F(\sigma,o)=(\mathbb C_\sigma,\varphi_{\sigma,o})
\]
Next let $z\in Z$ and observe that $\Theta\circ F$ determines a pre $K$-ad
$(\Theta\circ F)_z$ with values in $\Aa^{\Z[\pi_1(Z,z)]}$, with
\[
(\Theta\circ F)_z(\sigma,o)=(\Z[\pi_1
Z]\langle \mathbb C_\sigma\rangle_z, (\varphi_{\sigma,o})_z)
\]
where $(\varphi_{\sigma,o})_z$ is the composite
\begin{multline*}
W\xrightarrow{\varphi_{\sigma,o}}
(\Z[\pi_1 Z]\langle \mathbb C_\sigma\rangle)^t\otimes_{\Z[\pi_1 Z]}
\Z[\pi_1 Z]\langle \mathbb C_\sigma\rangle
\\
\cong
(\Z[\pi_1 Z]\langle \mathbb C_\sigma\rangle)_z)^t 
\otimes_{\Z[\pi_1(Z,z)]}
\Z[\pi_1 Z]\langle \mathbb C_\sigma\rangle)_z.
\end{multline*}
Using Definition \ref{nn7} and Remarks \ref{aa5} and \ref{aa6}, we see that 
$(\Theta\circ F)_z$ is an ad.

Now let $G=\Theta\circ g^\mathrm{sym}_\sch\circ F$ and write
\[
G(\sigma,o)=(\Z[\pi_1 Z']\langle g_\sch(\mathbb C_\sigma)\rangle,\psi_{\sigma,o});
\] 
we need to show that $G\in \ad^{\Z[\pi_1 Z']}(K)$.  First we observe that, by
Lemma \ref{aa35} and the definition of $g^\mathrm{sym}_\sch$, $G$ is
isomorphic to the pre $K$-ad $G'$ with
\[
G'(\sigma,o)=(\mathrm{Kan}_{g_*}(\Z[\pi_1 Z]\langle \mathbb C_\sigma\rangle,
\psi'_{\sigma,o}),
\]
where $\psi'_{\sigma,o}$ is the composite
\begin{multline*}
W\xrightarrow{\varphi_{\sigma,o}}
(\Z[\pi_1 Z]\langle \mathbb C_\sigma\rangle)^t\otimes_{\Z[\pi_1 Z]}
\Z[\pi_1 Z]\langle \mathbb C_\sigma\rangle
\xrightarrow{\iota\otimes\iota}
\\
(\mathrm{Kan}_{g_*}(\Z[\pi_1 Z]\langle \mathbb C_\sigma\rangle))^t
\otimes_{\Z[\pi_1 Z']}
\mathrm{Kan}_{g_*}(\Z[\pi_1 Z]\langle \mathbb C_\sigma\rangle),
\end{multline*}
so it suffices to show $G'$ is an ad.  Now (recalling that we chose $z\in Z$),
$G'$ determines a pre $K$-ad $G'_{g(z)}$ with
\[
G'_{g(z)}(\sigma,o)
=
(\mathrm{Kan}_{g_*}(\Z[\pi_1 Z]\langle \mathbb C_\sigma\rangle)_{g(z)},
(\psi'_{\sigma,o})_{g(z)}),
\]
where $(\psi'_{\sigma,o})_{g(z)}$ is the composite
\begin{multline*}
W\xrightarrow{\psi'_{\sigma,o}}
(\mathrm{Kan}_{g_*}(\Z[\pi_1 Z]\langle \mathbb C_\sigma\rangle))^t
\otimes_{\Z[\pi_1 Z']}
\mathrm{Kan}_{g_*}(\Z[\pi_1 Z]\langle \mathbb C_\sigma\rangle)
\\
\cong
(\mathrm{Kan}_{g_*}(\Z[\pi_1 Z]\langle \mathbb C_\sigma\rangle))_{g(z)}^t
\otimes_{\Z[\pi_1(Z',g(z))]}
\mathrm{Kan}_{g_*}(\Z[\pi_1 Z]\langle \mathbb C_\sigma\rangle)_{g(z)}
\end{multline*}
and it suffices (using Remarks \ref{aa5} and \ref{aa6}) to show that
$G'_{g(z)}$ is an ad.

Finally, we observe that, by
Remark \ref{aa33}, $G'_{g(z)}$ is isomorphic to the pre $K$-ad $H$ with
\[
H(\sigma,o)=
(\Z[\pi_1(Z',g(z)]\otimes_{\Z[\pi_1(Z,z)]}
(\Z[\pi_1 Z]\langle \mathbb C_\sigma\rangle)_z,
\chi_{\sigma,o})
\]
where $\chi_{\sigma,o}$ is the composite 
\begin{multline*}
W\xrightarrow{(\varphi_{\sigma,o})_z}
(\Z[\pi_1 Z]\langle \mathbb C_\sigma\rangle)_z)^t 
\otimes_{\Z[\pi_1(Z,z)]}
\Z[\pi_1 Z]\langle \mathbb C_\sigma\rangle)_z
\\
\to
(\Z[\pi_1(Z',g(z)]\otimes_{\Z[\pi_1(Z,z)]}
(\Z[\pi_1 Z]\langle \mathbb C_\sigma\rangle)_z)^t
\\
\otimes_{\Z[\pi_1(Z',g(z)]}
\Z[\pi_1(Z',g(z)]\otimes_{\Z[\pi_1(Z,z)]}
(\Z[\pi_1 Z]\langle \mathbb C_\sigma\rangle)_z,
\end{multline*}
and $H$ is an ad by \cite[Lemma B.15]{LM2}, using the fact that $(\Theta\circ
F)_z$ is an ad.
\end{proof}

Combining Propositions \ref{nn8} and \ref{nn9} we see that that 
$\ad^{\Z[\pi_1 Z]}_\sch$ is a functor of $Z$. 

\subsection{Statement of the main theorem of this section, and outline of the 
proof}
\label{nn10}

Since $\ad^{\Z[\pi_1 Z]}_{\sch,\geq 0}$ is a homotopy invariant functor of $Z$, there 
is an assembly map
\[
\alpha: Z_+\wedge \bQ^{\Z}_{\sch,\geq 0}
\to
\bQ^{\Z[\pi_1 Z]}_{\sch,\geq 0}.
\]

Our main theorem in this section is

\begin{thm}
\label{nn13}
The image of $[X]_{\mathbb{L}}$ under the composite
\[
X_+\wedge \bQ_{\geq 0}^{\Z}
\xleftarrow{\simeq}
X_+\wedge \bQ_{\sch,\geq 0}^{\Z}
\xrightarrow{\alpha}
\bQ^{\Z[\pi_1 Z]}_{\sch,\geq 0}
\xrightarrow{\simeq}
\bQ^{\Z[\pi_1 Z]}_{\geq 0}
\]
is the symmetric signature $\sigma_\ip^*(X)$ given by Definition \ref{d5}.
\end{thm}

For the proof, we will use the following functors and natural transformations:
\[
\Phi
\xleftarrow{\simeq}
\Phi'
\xrightarrow{\sig}
\Psi''
\xleftarrow{\simeq}
\Psi'
\to
\Psi
\]
Here 
\begin{itemize}
\item
$\Phi$ is the functor 
\[
\Phi(Z)=\bQ_{\ip,Z}
\]
defined in Section \ref{s8}, 
\item
$\Phi'$, which will be defined in Subsection \ref{aa16}, is a version of $\Phi$ which
incorporates a choice of fundamental class (cf.\ the ad theory 
$\ad_\ipFun$ defined in Subsection \ref{ss1}),
\item
$\Psi$ is the functor 
\[
\Psi(Z)=\bQ^{\Z[\pi_1 Z]}_{\sch,\geq 0}
\]
\item
$\Psi'$ will be defined in Subsection \ref{nn21}; it is a functorial version of
the ad theory $\ad^{\Z[\pi_1 Z]}_\mathrm{fin}$ defined in Remark \ref{nn12}.
It is needed because there is no reasonable map 
from $\Psi$ to $\Psi''$ or from $\Psi''$ to $\Psi$,
\item
the natural transformation $\Psi'\to\Psi$ will be defined in Subsection
\ref{nn21}, and the natural transformation $\Psi'\to\Psi''$ will be defined in
Subsection \ref{nn25},
\item
$\Psi''$, which will be defined in Subsection \ref{nn17}, is a functorial 
version of $\bQ^{\Z[\pi_1 Z]}_{\mathrm{rel},\geq 0}$ (which was defined in  
Section 
\ref{symsig}),
\item
the natural transformation $\sig$, which is a version of the symmectric
signature, will be defined in Subsection \ref{aa3}.
\end{itemize}

Now consider the diagram
\[
%\begin{equation}
\label{nn14}
\scriptsize{
\xymatrix{
X_+\wedge \bQ_\ip
\ar[d]_\alpha^\simeq
&
X_+\wedge \Phi'(*)
\ar[l]_-\simeq
\ar[r]^-{1\wedge\sig}
\ar[d]^\alpha
&
X_+\wedge \Psi''(*)
\ar[d]_\alpha
&
X_+\wedge \Psi'(*)
\ar[l]_-\simeq
\ar[r]
\ar[d]_\alpha
&
X_+\wedge \bQ^\Z_{\sch,\geq 0}
\ar[d]_\alpha
\\
\bQ_{\ip,X}
&
\Phi'(X)
\ar[l]_-\simeq
\ar[r]^-\sig
&
\Psi''(X)
&
\Psi'(X)
\ar[l]_-\simeq
\ar[r]
&
\bQ_{\sch,\geq 0}^{\Z[\pi_1 X]}
}
}
\]
%\end{equation}

This diagram commutes up to homotopy by the naturality of the assembly map 
(\cite[Theorem 1.1]{wwa}).  Now Theorem \ref{nn13} is immediate from the
following two lemmas, which will be proved in Subsections \ref{nn29} and
\ref{nn3} respectively.

\begin{lemma}
\label{nn15}
The image of $\alpha^{-1}[X]_\ip$ along the top row of the diagram above
followed by the map 
$X_+\wedge \bQ^\Z_{\sch,\geq 0}
\to
X_+\wedge \bQ^\Z_{\geq 0}$
is 
$[X]_\mathbb{L}$. 
\end{lemma}

\begin{lemma}
\label{nn16}
The image of $[X]_\ip$ along
the bottom row followed by the map $\bQ_{\sch,\geq 0}^{\Z[\pi_1 X]}
\to
\bQ_{\geq 0}^{\Z[\pi_1 X]}$
 is $\sigma_\ip^*(X)$.
\end{lemma}

\subsection{Background for the functor $\Psi''$}
\label{ss5}

In this subsection and the next our goal is to give a functorial version of
$\ad^{\Z[\pi_1 Z]}_\mathrm{rel}$.  The second and third authors have done this
for rings with involution in \cite[Appendix B]{LM2} and we only need to
adapt that to the present context.

We need to make one important change in the framework of \cite[Appendix
B]{LM2}. Expressions of the form $R\otimes_{R^\mathrm{op}\otimes
R} N$ frequently occur there, where $N$ is an $R^\mathrm{op}\otimes
R$ module.  In corresponding situations in the present paper we will have a
$\Z[\pi_1 Z]^\mathrm{op}\otimes \Z[\pi_1 Z]$ module $\N$ and we will use the
expression
\[
\ZZ\otimes_{\Z[\pi_1 Z]}\N,
\]
where $\ZZ$ is the constant right $\Z[\pi_1 Z]$ module with value $\Z$ and we
think of $\N$ as a left $\Z[\pi_1 Z]$ module via the map
\[
\Z[\pi_1 Z]
\to
\Z[\pi_1 Z]^\mathrm{op}\otimes \Z[\pi_1 Z]
\]
which takes an object $z$ to $(z,z)$ and a path homotopy class $\delta$ to
$(\delta^{-1},\delta)$.
(We could give a precise analogue of $R\otimes_{R^\mathrm{op}\otimes
R}N$, but it would be isomorphic to $\ZZ\otimes_{\Z[\pi_1 Z]}
\N$ and the latter is easier to work with.)

The first issue that must be addressed in creating $\Psi''$ is that,
unfortunately, $\ad^{\Z[\pi_1
Z]}_\mathrm{rel}$ isn't even ``approximately'' functorial, that is, given a map
$g:Z\to Z'$ and an object $(C,D,\beta,\phi)$ of $\Aa^{\Z[\pi_1
Z]}_\mathrm{rel}$, there doesn't seem to be a reasonable way to create an
object $(C',D',\beta',\phi')$ of $\Aa^{\Z[\pi_1 Z']}_\mathrm{rel}$ (we could
let $C'=\mathrm{Kan}_{g_*} C$, but there's no reasonable candidate for
$D'$).  
\cite[Subsections B.1--B.4]{LM2} construct an ad theory $\ad^R_\mathrm{Rel}$
which is approximately functorial (that is, functorial up to isomorphism) and
has a map to $\ad^R_\mathrm{rel}$ which induces an isomorphism of bordism
groups; in this subsection we generalize this.

% The first step is the analog of \cite[Lemma B.2]{LM2}.  
% 
% \begin{definition}
% \label{aa7}
% Let $\P(Z)$ be the bimodule associated, as in Appendix \ref{add}, to the 
% identity functor of $\Z[\pi_1 Z]$.
% \end{definition}
% 
% \begin{lemma}
% \label{aa42}
% Let ${\mathcal M}$ be a right $\Z[\pi_1 Z]$ module and ${\mathcal N}$ a left $\Z[\pi_1 Z]$ module.
% The map
% \[
% a:{\mathcal P}(Z)\otimes_{\Z[\pi_1 Z]^{\mathrm{op}}\otimes \Z[\pi_1 Z]}
% ({\mathcal M}\otimes {\mathcal N})
% \to
% {\mathcal M}\otimes_{\Z[\pi_1 Z]} {\mathcal N}
% \]
% given by 
% $a(\delta\otimes m\otimes n)={\mathcal M}_\delta(m)\otimes n$
% is an isomorphism.
% \end{lemma}
% 
% This is immediate from Remark \ref{aa5} and \cite[Lemma B.2]{LM2}.
% 
% Now \cite[Definition B.3 and Lemma B.4]{LM2} 
% generalize in a straightforward way.  

\cite[Definition B.3 and Lemma B.4]{LM2} generalize to our context in a 
straighforward way.  The analogue of \cite[Definition B.5]{LM2} is

\begin{definition} 
A {\it Relaxed quasi-symmetric complex of dimension $n$ over $\Z[\pi_1 Z]$} 
is a quadruple  $(C,E,\gamma,\phi)$, where $C$ is a homotopy finite chain
complex over $\Z[\pi_1 Z]$, $E$ is a homotopy finite chain complex over
$\Z[\pi_1 Z]^{\mathrm{op}}\otimes \Z[\pi_1 Z]$ with a $\Z/2$ action for which
the generator acts quasi-linearly, $\gamma$ is a $\Z/2$ equivariant $\Z[\pi_1
Z]^{\mathrm{op}}\otimes \Z[\pi_1 Z]$ linear quasi-isomorphism $C^t\otimes C\to
E$, and $\phi$ is an $n$-dimensional element of $(\ZZ\otimes_{\Z[\pi_1 Z]} 
E)^{\Z/2}$.
\end{definition} 

\begin{remark}
\label{nn18}
\cite[Definition B.7]{LM2} generalizes
to give a category $\Aa^{\Z[\pi_1 
Z]}_\mathrm{Rel}$.
Analogously to \cite[Remark B.8]{LM2} we have a map
\[
\Lambda:\Aa^{\Z[\pi_1 
Z]}_\mathrm{Rel}\to
\Aa^{\Z[\pi_1 
Z]}_\mathrm{rel}
\]
which takes $(C,E,\gamma,\phi)$ to $(C,\ZZ\otimes_{\Z[\pi_1 Z]} E,\beta,\phi)$,
where $\beta$ is the following composite (in which the isomorphism is
Remark \ref{tens})
\[
C^t\otimes_{\Z[\pi_1 Z]} C
\cong
\ZZ\otimes_{\Z[\pi_1 Z]} (C^t\otimes C)
\xrightarrow{1\otimes \gamma}
\ZZ\otimes_{\Z[\pi_1 Z]} E. 
\]
\cite[Definitions B.10 and B.11]{LM2} generalize to give a set $\ad^{\Z[\pi_1 
Z]}_\mathrm{Rel}(K)$ for each $K$; note that (as in \cite[Definition 
B.11(ii)]{LM2}) every ad is isomorphic to a
balanced ad.  \cite[Theorem B.12]{LM2} generalizes to
the statement that this is an ad theory.
The proof of \cite[Proposition B.17]{LM2} generalizes to show that the map
\[
\Lambda:\ad^{\Z[\pi_1 
Z]}_\mathrm{Rel}\to
\ad^{\Z[\pi_1 
Z]}_\mathrm{rel}
\]
induces an isomorphism of bordism groups.
\end{remark}

Next we show that $\ad^{\Z[\pi_1 Z]}_\mathrm{Rel}$ is approximately functorial.

\begin{definition}
\label{nn20}
Let $g:Z\to Z'$ be a continuous function.  Define a functor 
\[
g_\mathrm{Rel}:\Aa^{\Z[\pi_1 Z]}_\mathrm{Rel} \to
\Aa^{\Z[\pi_1 Z']}_\mathrm{Rel}
\]
as follows. For an object $(C,E,\gamma,\phi)$ of $\Aa^{\Z[\pi_1
Z]}_\mathrm{Rel}$,  let
\[
g_\mathrm{Rel}((C,E,\gamma,\phi)
=
(\Kan_{g_*}(C),\Kan_{g_*\otimes g_*}(E),\gamma',\phi'),
\]
where $\gamma'$ is the composite
\[
(\Kan_{g_*}(C))^t
\otimes
\Kan_{g_*}(C)
\cong
\Kan_{g_*\otimes g_*} (C^t\otimes C)
\xrightarrow{\Kan_{g_*\otimes g_*} \gamma}
\Kan_{g_*\otimes g_*} (E)
\]
(which is a quasi-isomorphism by Remarks \ref{nn19} and \ref{aa33} and the
K\"unneth spectral sequence \cite[Theorem 5.6.4]{Weib}) 
and  $\phi'$ is the image of $\phi$ under the map
\[
% begin{multline*}
(\ZZ\otimes_{\Z[\pi_1 Z]}
 E)^{\Z/2}
\to
(\ZZ\otimes_{\Z[\pi_1 Z']}
\Kan_{g_*\otimes g_*}(E))^{\Z/2}
% \end{multline*}
\]
(where $\ZZ$ denotes the $\Z[\pi_1 Z]$ module in the first expression and the
$\Z[\pi_1 Z']$ module in the second) which takes the class of $1\otimes e$
(where $e\in E_{(z,z)}$) to the class of $1\otimes \id_{(z,z)}\otimes e$.
\end{definition}

The proof of \cite[Proposition B.15]{LM2} generalizes (using Remarks \ref{aa5},
\ref{aa6}, \ref{nn19}, and \ref{aa33}) to show that
$g_\mathrm{Rel}$ takes ads to ads.

\subsection{The functor $\Psi''$}
\label{nn17}

In this section we give a schematic version of $\ad^{\Z[\pi_1
Z]}_\mathrm{Rel}$.  First we need to define the concept of schematic 
$\Z[\pi_1 Z]^\mathrm{op}\otimes \Z[\pi_1 Z]$ module.

There is an isomorphism
\[
\kappa:\Z[\pi_1 Z]^\mathrm{op}\otimes
\Z[\pi_1 Z]\cong
\Z[\pi_1(Z\times Z)]
\]
which takes $\delta\otimes\delta'$ to $\delta^{-1}\times \delta'$, where 
$\delta$
and $\delta'$ are path homotopy classes in $Z$.

\begin{definition}
\label{aa31}
(i)
A {\it schematic {module} over $\Z[\pi_1 Z]^\mathrm{op}\otimes
\Z[\pi_1 Z]$} is a schematic {module}, in the sense already defined, over
$\Z[\pi_1(Z\times Z)]$.  

(iii) For a schematic $\Z[\pi_1 Z]^\mathrm{op}\otimes
\Z[\pi_1 Z]$ {module} $((z_1,z_2),\mathbb A)$, define
\[
(\Z[\pi_1 Z]^\mathrm{op}\otimes
\Z[\pi_1 Z])\langle (z_1,z_2),\mathbb A\rangle
\] 
to be the module induced from 
$\Z[\pi_1(Z\times Z)]\langle (z_1,z_2),\mathbb A\rangle$
by the isomorphism $\kappa$.
\end{definition}

\begin{definition}
(i)
A {\it schematic Relaxed quasi-symmetric complex of dimension $n$}
is a quadruple $(\mathbb C,\mathbb E,\gamma,\phi)$, where $\mathbb C$ is a
schematic $\Z[\pi_1 Z]$ chain complex, $\mathbb E$ is a schematic
$\Z[\pi_1 Z]^\mathrm{op}\otimes
\Z[\pi_1 Z]$
chain complex, $\Z[\pi_1 Z]\langle \mathbb
C\rangle$ is homotopy
finite,
$(\Z[\pi_1 Z]^\mathrm{op}\otimes
\Z[\pi_1 Z])\langle  \mathbb E\rangle$ is a
homotopy finite chain complex over $\Z[\pi_1 Z]^\mathrm{op}\otimes
\Z[\pi_1 Z]$ with a
$\Z/2$ action for which the generator acts quasi-linearly, $\gamma$ is a
$\Z/2$ equivariant
quasi-isomorphism $(\Z[\pi_1 Z]\langle \mathbb
C\rangle)^t\otimes \Z[\pi_1 Z]\langle \mathbb
C\rangle \to (\Z[\pi_1 Z]^\mathrm{op}\otimes
\Z[\pi_1 Z])\langle  \mathbb E\rangle$ of
$\Z[\pi_1 Z]^\mathrm{op}\otimes
\Z[\pi_1 Z]$ chain complexes, and $\phi$
is an $n$-dimensional
element of
$(\ZZ\otimes_{\Z[\pi_1 Z]}
(\Z[\pi_1 Z]^\mathrm{op}\otimes
\Z[\pi_1 Z])\langle
\mathbb E\rangle)^{\Z/2}$. 

(ii) We define a category $\Aa_\Rel^{\Z[\pi_1 Z]}$ as
follows.  The objects of $\Aa_\Rel^{\Z[\pi_1 Z]}$ are the schematic Relaxed 
quasi-symmetric complexes.
A morphism 
$(\mathbb C,\mathbb E,\gamma,\phi)\to (\mathbb
C',\mathbb E',\gamma',\phi')$
is a pair $(f_1:\Z[\pi_1 Z]\langle\mathbb C\rangle\to 
\Z[\pi_1 Z]\langle\mathbb C'\rangle,
f_2:(\Z[\pi_1 Z]^\mathrm{op}\otimes
\Z[\pi_1 Z])\langle \mathbb E\rangle
\to
(\Z[\pi_1 Z]^\mathrm{op}\otimes
\Z[\pi_1 Z])\langle \mathbb E'\rangle)$, where $f_1$
is a map of $\Z[\pi_1 Z]$ chain complexes
and $f_2$
is a
$\Z/2$ equivariant map of $\Z[\pi_1 Z]^\mathrm{op}\otimes
\Z[\pi_1 Z]$ chain complexes, such that
$f_2\gamma=\gamma'(f_1\otimes f_1)$,
and (if $\dim \phi=\dim\phi'$)
$(1\otimes f_2)_*(\phi)=\phi'$.
\end{definition}

There is a functor
\[
\Theta_\mathrm{Rel}:\Aa_\Rel^{\Z[\pi_1 Z]}
\to
\Aa_\mathrm{Rel}^{\Z[\pi_1 Z]}
\]
which takes $(\mathbb C,\mathbb E,\gamma,\phi)$ to
$(\Z[\pi_1 Z]\langle \mathbb
C\rangle, (\Z[\pi_1 Z]^\mathrm{op}\otimes
\Z[\pi_1 Z])\langle  \mathbb E\rangle,\gamma,\phi)$; this is an equivalence of
categories.  

\begin{definition}
\label{nn27}
A $K$-ad with values in $\Aa_\Rel^{\Z[\pi_1 Z]}$ is a
pre $K$-ad $F$ for which $\Theta_\mathrm{Rel}\circ F$ is an ad.  
\end{definition}

The proof of
\cite[Proposition B.22]{LM2} shows that $\ad_\Rel^{\Z[\pi_1 Z]}$ is an ad
theory and that $\Theta_\mathrm{Rel}$ induces an isomorphism of bordism groups.

Next we consider functoriality.  Let $g:Z\to Z'$ be a continuous function.

\begin{definition}
Define a functor
\[
g_\Rel:\Aa_\Rel^{\Z[\pi_1 Z]}\to \Aa_\Rel^{\Z[\pi_1 Z']}
\]
as follows.  For an object $(\mathbb C,\mathbb E,\gamma,\phi)$ of
$\Aa_\Rel^{\Z[\pi_1 Z]}$, let
\[
g_\Rel(\mathbb C,\mathbb E,\gamma,\phi)
=
(g_\sch\mathbb C, (g\times g)_\sch\mathbb E,\delta,\psi),
\]
where $g_\sch$ and $(g\times g)_\sch$ are defined in Subsection \ref{nn4},
and (letting $C=\Z[\pi_1 Z]\langle {\mathbb C}\rangle$ and $E=
(\Z[\pi_1 Z]^\mathrm{op}\otimes
\Z[\pi_1 Z])\langle {\mathbb E}\rangle$, and
using the notation of Definitions \ref{nn20} and \ref{aa31}(ii) and the 
isomorphism of Lemma \ref{aa35})
$\delta$ is defined by the diagram
\[
\xymatrix{
\Z[\pi_1 Z']\langle g_\sch{\mathbb C}\rangle^t\otimes 
\Z[\pi_1 Z']\langle g_\sch{\mathbb C}\rangle
\ar[r]^-\delta
\ar[d]_\cong
&
(\Z[\pi_1 Z']^{\mathrm{op}}\otimes \Z[\pi_1 Z'])\langle(g\times 
g)_\sch{\mathbb E}\rangle
\ar[d]_\cong
\\
(\Kan_{g_*} (C))^t\otimes (\Kan_{g_*} (C))
\ar[r]^-{\gamma'}
&
\Kan_{g_*\otimes g_*}(E)
}
\]
and $\psi$ is the image of $\phi'$ under the isomorphism
\begin{multline*}
(\ZZ\otimes_{\Z[\pi_1 Z']}
\Kan_{g_*\otimes g_*}E)^{\Z/2}
\\
\cong
(\ZZ\otimes_{\Z[\pi_1 Z']}
(\Z[\pi_1 Z']^{\mathrm{op}}\otimes \Z[\pi_1 Z'])
\langle(g\times g)_\sch {\mathbb E}\rangle)^{\Z/2}.
\end{multline*}
\end{definition}

The proof of \cite[Proposition B.25]{LM2} shows that $g_\Rel$ takes ads to ads.
The reader can check that if $g':Z'\to Z''$ is another continuous function then
$(g'g)_\Rel=g'_\Rel g_\Rel$, so $\ad_\Rel^{\Z[\pi_1 Z]}$ is a functor of $Z$ as
required. 

We can now define $\Psi''$ by 
\[
\Psi''(Z)=\bQ_\Relz^{\Z[\pi_1 Z]}.
\]

\subsection{Background for the map $\Psi'\rightarrow\Psi''$}
\label{nn24}

Recall Remark \ref{nn12}.  In the next subsection we will define $\Psi'$ to 
be a certain schematic version of $\bQ^{\Z[\pi_1 Z]}_\mathrm{fin}$.  In this 
subsection, as preparation for the construction of  the map 
$\Psi'\to\Psi''$, we give the nonschematic version of this map.

To do this, we will give a map
\[
\Pi:\Aa^{\Z[\pi_1 Z]}_\mathrm{fin}\to\Aa^{\Z[\pi_1 Z]}_\mathrm{Rel}
\]
(following \cite[Remarks B.6 and B.9]{LM2}). 

Recall Notation \ref{nn22}. For 
a chain complex $C$ over $\Z[\pi_1 Z]$, let $(C^t\otimes C)^W$ be the chain 
complex over $\Z[\pi_1 Z]^\mathrm{op}\otimes \Z[\pi_1 Z]$ with 
\[
((C^t\otimes C)^W)_{(z_1,z_2)}=
(C_{z_1}^t\otimes C_{z_2})^W
\]
When $C$ is finite, $(C^t\otimes C)^W$ is (additively) a direct sum of copies
of $C^t\otimes C$, and hence the natural map
\[
\ZZ\otimes_{\Z[\pi_1 Z]} ((C^t\otimes C)^W)
\to
(\ZZ\otimes_{\Z[\pi_1 Z]} (C^t\otimes C))^W
\]
is an isomorphism because tensor product preserves direct sums.

Now define $\Pi$ by
\[
\Pi(C,\varphi)
=
(C,(C^t\otimes C)^W,\gamma,\phi),
\]
where $\gamma:C^t\otimes C\to (C^t\otimes C)^W$ is induced by the augmentation
$W\to \Z$ and $\phi$ is the image of $\varphi$ under the composite
\begin{multline*}
((C^t\otimes C)^W)^{\Z/2}
\cong
((\ZZ\otimes_{\Z[\pi_1 Z]}
(C^t\otimes C))^W)^{\Z/2}
\\
\cong
(\ZZ\otimes_{\Z[\pi_1 Z]} ((C^t\otimes
C)^W))^{\Z/2}
\end{multline*}
(where the first isomorphism is Remark \ref{tens}).

$\Pi$ takes ads to ads, 
and the proof of \cite[Proposition B.17]{LM2}
shows that it induces an isomorphism of bordism groups.

\begin{remark}
\label{nn33}
Consider the following diagram, where the left vertical arrow was given in
Remark \ref{nn12}, the map $\Lambda$ in Remark \ref{nn18}, and the lower
horizontal map in Remark \ref{nn32}:
\[
\xymatrix{
\Aa^{\Z[\pi_1 Z]}_\mathrm{fin}
\ar[r]^\Pi
\ar[d]
&
\Aa^{\Z[\pi_1 Z]}_\mathrm{Rel}
\ar[d]_\Lambda
\\
\Aa^{\Z[\pi_1 Z]}
\ar[r]
&
\Aa^{\Z[\pi_1 Z]}_\mathrm{rel}.
}
\]
The diagram commutes up to natural isomorphism, so the following diagram
homotopy commutes, by 
the analog of \cite[Proposition 14.5]{LM2}.
\[
\xymatrix{
\bQ^{\Z[\pi_1 Z]}_\mathrm{fin}
\ar[r]^\Pi
\ar[d]
&
\bQ^{\Z[\pi_1 Z]}_\mathrm{Rel}
\ar[d]_\Lambda
\\
\bQ^{\Z[\pi_1 Z]}
\ar[r]
&
\bQ^{\Z[\pi_1 Z]}_\mathrm{rel}.
}
\]
\end{remark}

\subsection{The functor $\Psi'$ and the map
$\Psi'\rightarrow\Psi$}
\label{nn21}

Our first task is to give a suitable schematic version $\Aa^{\Z[\pi_1
Z]}_\mathrm{fin,\sch}$ of $\Aa^{\Z[\pi_1
Z]}_\mathrm{fin}$.  The obvious choice would be the subcategory of
$\Aa^{\Z[\pi_1 Z]}_\sch$ consisting of objects $(\mathbb C,\varphi)$ for which
$\Z[\pi_1 Z]\langle\mathbb C\rangle$ is finite, but there is no reasonable way
to give a map from this to $\Aa^{\Z[\pi_1 Z]}_\Rel$; the difficulty is that 
given an object $(\mathbb C,\varphi)$ there is no reasonable choice for
$\mathbb E$.  We will therefore build a choice for $\mathbb E$ into the
definition of the category $\Aa^{\Z[\pi_1 Z]}_\mathrm{fin,\sch}$.

\begin{definition}
\label{nn23}
Let $C$ be a chain complex over $\Z[\pi_1 Z]$ (resp., $\Z[\pi_1
Z]^\mathrm{op}\otimes \Z[\pi_1 Z]$).  A {\it schematic model} for $C$ is a pair
$(\mathbb C, \lambda)$, where 
$\mathbb C$ is a schematic chain complex over $\Z[\pi_1 Z]$ (resp., $\Z[\pi_1
Z]^\mathrm{op}\otimes \Z[\pi_1 Z]$) and $\lambda$  is an isomorphism
$ \Z[\pi_1 Z]\langle\mathbb C\rangle
\to C$ (resp., $(\Z[\pi_1
Z]^\mathrm{op}\otimes \Z[\pi_1 Z])\langle\mathbb C\rangle
\to C$).
\end{definition}

\begin{definition}
Define a category $\Aa^{\Z[\pi_1 Z]}_{\mathrm{fin},\sch}$ as follows.  
An object is a quadruple
\[
(\mathbb C,\varphi,\mathbb E,\lambda)
\]
where $(\mathbb C,\varphi)$ is an object of $\Aa^{\Z[\pi_1 
Z]}_\sch$ with $\Z[\pi_1 Z]\langle \mathbb C\rangle$ finite and 
$(\mathbb E,\lambda)$ is a schematic model for $((\Z[\pi_1 Z]\langle \mathbb 
C\rangle)^t\otimes \Z[\pi_1 Z]\langle \mathbb C\rangle)^W$.
The morphisms $(\mathbb C,\varphi,\mathbb E,\lambda)\to (\mathbb
C',\varphi',\mathbb E',\lambda')$ are the morphisms $(\mathbb C,\varphi)\to
(\mathbb C',\varphi')$ in $\Aa^{\Z[\pi_1 Z]}_\sch$.  
\end{definition}

Note that the definition of morphism does not involve the schematic model 
$(\mathbb
E,\lambda)$.  Because of this, the functor
\[
\Theta_\mathrm{fin}:
\Aa^{\Z[\pi_1 Z]}_{\mathrm{fin},\sch}
\to
\Aa^{\Z[\pi_1 Z]}_{\mathrm{fin}}
\]
which takes $(\mathbb C,\varphi,\mathbb E,\lambda)$ to $(\Z[\pi_1 Z]\langle
\mathbb C\rangle,\varphi)$
is an equivalence of categories.  

\begin{definition}
\label{nn28}
An ad with values in $\Aa^{\Z[\pi_1
Z]}_{\mathrm{fin},\sch}$ is a pread whose composite with 
$\Theta_\mathrm{fin}$ is an ad.  
\end{definition}

The proof of Proposition \ref{nn11} shows that this gives an ad theory and
that the map $\ad^{\Z[\pi_1 Z]}_{\mathrm{fin},\sch} 
\to \ad^{\Z[\pi_1 Z]}_{\mathrm{fin}}$ induces an isomorphism of bordism groups.

Next we show that $\ad^{\Z[\pi_1 Z]}_{\mathrm{fin},\sch}$ is a functor of $Z$.
Let $g:Z\to Z'$ be a continuous map, and define a functor
\[
g_{\mathrm{fin},\sch}:
\Aa^{\Z[\pi_1 Z]}_{\mathrm{fin},\sch}
\to
\Aa^{\Z[\pi_1 Z']}_{\mathrm{fin},\sch}
\]
by
\[
g_{\mathrm{fin},\sch}(\mathbb C,\varphi,\mathbb E,\lambda)
=
(g^\mathrm{sym}_\sch(\mathbb C,\varphi),(g\times g)_\sch \mathbb E, \lambda')
\]
where $g^\mathrm{sym}_\sch$ was defined in Subsection \ref{mm} and $\lambda'$ 
is the following composite
(where the first isomorphism is given by Lemma \ref{aa35})
\begin{multline*}
(\Z[\pi_1 Z']^\mathrm{op}\otimes \Z[\pi_1 Z'])\langle(g\times g)_\sch 
\mathbb E\rangle
\cong
\Kan_{g_*\otimes g_*}((\Z[\pi_1 Z]^\mathrm{op}\otimes \Z[\pi_1 Z])
\langle\mathbb E\rangle)
\\
\xrightarrow{\Kan_{g_*\otimes g_*}(\lambda)}
\Kan_{g_*\otimes g_*}(((\Z[\pi_1 Z]\langle \mathbb 
C\rangle)^t\otimes \Z[\pi_1 Z]\langle \mathbb C\rangle)^W)
\\
\cong
(\Kan_{g_*\otimes g_*}((\Z[\pi_1 Z]\langle \mathbb 
C\rangle)^t\otimes \Z[\pi_1 Z]\langle \mathbb C\rangle))^W
\\
\cong
(
(\Z[\pi_1 Z']\langle g_\sch\mathbb 
C\rangle)^t\otimes \Z[\pi_1 Z']\langle g_\sch\mathbb C\rangle
)^W.
\end{multline*}

$g_{\mathrm{fin},\sch}$ takes ads to ads by Proposition \ref{nn9}.  We leave it
to the reader to check that if $g':Z'\to Z''$ then $(g'g)_{\mathrm{fin},\sch}
=g'_{\mathrm{fin},\sch}g_{\mathrm{fin},\sch}$.

Now we can define
\[
\Psi'(Z)=\bQ^{\Z[\pi_1 Z]}_{\mathrm{fin},\sch,\geq 0}.
\]

The functor
\[
\Psi'\to\Psi
\]
is induced by the map
\[
\Aa^{\Z[\pi_1 Z]}_{\mathrm{fin},\sch}
\to
\Aa^{\Z[\pi_1 Z]}_\sch
\]
which takes $(\mathbb C,\varphi,\mathbb E,\lambda)$
to
$(\mathbb C,\varphi)$.

\subsection{The map $\Psi'\to \Psi''$}
\label{nn25}

In order to define the map $\Psi'\to\Psi''$, we first define
\[
\Pi_\sch:\Aa^{\Z[\pi_1 Z]}_{\mathrm{fin},\sch}
\to
\Aa^{\Z[\pi_1 Z]}_\Rel
\]
as follows.  

On objects, $\Pi_\sch$ is defined by
\[
\Pi_\sch(\mathbb C,\varphi,\mathbb E,\lambda)
=
(\mathbb C,\mathbb E,\gamma,\phi);
\]
here $\gamma$ is the composite
\begin{multline*}
(\Z[\pi_1 Z]\langle \mathbb C\rangle)^t\otimes \Z[\pi_1 Z]\langle 
\mathbb C\rangle
\to
((\Z[\pi_1 Z]\langle \mathbb C\rangle)^t\otimes \Z[\pi_1 Z]\langle 
\mathbb C\rangle)^W
\\
\xrightarrow{\lambda^{-1}}
(Z[\pi_1 Z]^\mathrm{op}\otimes \Z[\pi_1 Z])\langle \mathbb E\rangle,
\end{multline*}
(where the first map is induced by the augmentation $W\to \Z$),
and (writing $C$ for $\Z[\pi_1 Z]\langle \mathbb C\rangle$ and $E$ for $(Z[\pi_1
Z]^\mathrm{op}\otimes \Z[\pi_1 Z])\langle \mathbb E\rangle$) $\phi$ is the
image of $\varphi$ under the composite 
\begin{multline*}
((C^t\otimes C)^W)^{\Z/2}
\cong
((\ZZ\otimes_{\Z[\pi_1 Z]}
(C^t\otimes C))^W)^{\Z/2}
\\
\cong
(\ZZ\otimes_{\Z[\pi_1 Z]} ((C^t\otimes
C)^W))^{\Z/2}
%\\
\xrightarrow{1\otimes \lambda^{-1}}
(\ZZ\otimes_{\Z[\pi_1 Z]} E)^{\Z/2}
\end{multline*}
(where the first isomorphism is Remark \ref{tens}).  

To define $\Pi_\sch$ on morphisms, recall that a morphism $(\mathbb
C,\varphi,\mathbb E,\lambda)\to (\mathbb
C',\varphi',\mathbb E',\lambda')$ is just a map of chain complexes $f:\Z[\pi_1
Z]\langle \mathbb C\rangle\to \Z[\pi_1
Z]\langle \mathbb C'\rangle$ (such that $(f\otimes f)\circ
\varphi=\varphi'$ if $\dim \varphi=\dim\varphi'$).  We define $\Pi_\sch$ of
such a morphism to be the morphism $(f,f_1)$, where 
$f_1$ is the composite
\begin{multline*}
(Z[\pi_1 Z]^\mathrm{op}\otimes \Z[\pi_1 Z])\langle \mathbb E\rangle
\xrightarrow{\lambda}
((Z[\pi_1 Z]\langle \mathbb C\rangle)^t
\otimes
Z[\pi_1 Z]\langle \mathbb C\rangle)^W
\\
\xrightarrow{(f\otimes f)^W}
((Z[\pi_1 Z]\langle \mathbb C'\rangle)^t
\otimes
Z[\pi_1 Z]\langle \mathbb C'\rangle)^W
\\
\xrightarrow{(\lambda')^{-1}}
(Z[\pi_1 Z]^\mathrm{op}\otimes \Z[\pi_1 Z])\langle \mathbb E'\rangle.
\end{multline*}

To see that $\Pi_\sch$ takes ads to ads, consider the diagram
\[
%\begin{equation}
%\label{nn26}
\xymatrix{
\Aa^{\Z[\pi_1 Z]}_{\mathrm{fin},\sch}
\ar[r]^-{\Pi_\sch}
\ar[d]_{\Theta_\mathrm{fin}}
&
\Aa^{\Z[\pi_1 Z]}_\Rel
\ar[d]_{\Theta_\mathrm{Rel}}
\\
\Aa^{\Z[\pi_1 Z]}_{\mathrm{fin}}
\ar[r]^-\Pi
&
\Aa^{\Z[\pi_1 Z]}_\mathrm{Rel},
}
%\end{equation}
\]
where $\Theta_\mathrm{fin}$ was defined in Subsection \ref{nn21},
$\Theta_\mathrm{Rel}$ in Subsection \ref{nn17}, and
$\Pi$ in Subsection \ref{nn24}.
The diagram commutes up to natural isomorphism.
We need to show that
if $F$ is an ad with values in $\Aa^{\Z[\pi_1 Z]}_{\mathrm{fin},\sch}$
then $\Pi_\sch\circ F$ is an ad, and for this it suffices by Definition
\ref{nn27} 
to show that
$\Theta_\mathrm{Rel}\circ\Pi_\sch\circ F$ is an ad, and for this in turn it
suffices to show that $\Pi\circ\Theta_\mathrm{fin}\circ F$ is an ad.  But
$\Theta_\mathrm{fin}\circ F$ is an ad by Definition \ref{nn28}, and $\Pi$ takes
ads to ads.

Now we can define the map
\[
\Psi'\to\Psi''
\]
to be the map 
\[
\bQ^{\Z[\pi_1 Z]}_{\mathrm{fin},\sch,\geq 0}
\to
\bQ^{\Z[\pi_1 Z]}_{\Rel,\geq 0}
\]
induced by $\Pi_\sch$.

To see that this map is an equivalence, we only need to observe that in the
diagram above the two vertical maps and the lower horizontal map all induce
isomorphisms of bordism groups.

\subsection{Background for $\Phi'$}
\label{aa2}

Recall that $\Phi(Z)=\bQ_{\ip,Z}$.
We need to create a functor $\Phi'$ which is equivalent 
to $\Phi$ and has a
symmetric signature map to $\Psi''$. For this we need a suitable ad 
theory for each $Z$, and for this in turn we need a category $\Aa$ for
each $Z$.   We have seen in Subsection 
\ref{ss1} that $\bQ_\ipFun$ is equivalent to $\bQ_\ip$ and has a symmetric 
signature map to $\bQ^\Z_{\mathrm{rel},\geq 0}$, which suggests that we could 
let the objects of $\Aa$ be the pairs $(f:X\to Z,\xi)$, where $\xi$ is a 
representative for the fundamental class.  However, this does not map to 
$\Aa^{\Z[\pi_1 Z]}_{\Rel}$, because an object of $\Aa^{\Z[\pi_1 Z]}_{\Rel}$ 
has the form $(\mathbb C,\mathbb E,\gamma,\phi)$ and a pair $(f:X\to Z,\xi)$ 
does not give a canonical choice for $\mathbb C$ and $\mathbb E$.  This 
suggests that, as in Subsection \ref{nn21}, we should include schematic 
models $(\mathbb C,\lambda_1)$ and $(\mathbb E,\lambda_2)$ as part of the 
structure of an object of $\Aa$ (we'll say later what these are models of).  
% Unfortunately, this still doesn't work, because the symmetric signature 
% functor obtained in this way does not take ads to ads, because it does not 
% preserve well-behavedness (we have not mentioned well-behavedness explicitly in
% this section, but it is included in the definition of $\ad^{\Z[\pi_1
% Z]}_\mathrm{Rel}$, since that is generalized from \cite[Definitions B.10 and
% B.11]{LM2}, and hence it is included in the definition of $\ad^{\Z[\pi_1
% Z]}_{\Rel}$).
Unfortunately, this still isn't functorial (cf.\ the proof of Lemma
\ref{aa36}(i)).
For functoriality, we need to add 
another ingredient, namely the intersection chains that are compatible with a 
certain open cover of $X$.

In this subsection we work out the theory of this additional ingredient, and in
the next subsection we combine this with the other ingredients to create
$\Phi'$.

First we need some terminology and notation.

Let $\uu$ be an open covering of a space $X$.  
For each perversity $\bar p$
we write
$IS_*^{\bar p,\uu} ({X};\Z)$
for the subcomplex of $IS_*^{\bar p} ( X;\Z)$ 
generated by the 
subcomplexes
$IS_*^{\bar p} (U;\Z)$ 
with $U\in \uu$ (this is somewhat different from the 
notation in 
\cite[Section 6]{friedmanmcclure}).  

\begin{remark}
\label{aa30}
\cite[Proposition
2.9]{FTrans} shows that the inclusion
\[
IS_*^{\bar p,\uu} ({X};\Z)\to IS_*^{\bar p} ( X;\Z)
\]
is a quasi-isomorphism.
\end{remark}

\begin{definition}
\label{aa22}
Let $f:X\to Z$ be an object of $\Aa_{\ip,Z}$.
The open cover $\uu_X$ of $X$ consists of the contractible open sets. 
\end{definition}

Define a category
$\Aa_{\ip,Z,\Fun}$ as follows.  The objects are pairs $(f:X\to Z,\xi)$, 
where $f$ is an object of $\Aa_{\ip,Z}$ 
and 
$\xi$ is an element of $IS_*^{\bar 0,\uu_X} ({X};\Z)$ whose image in 
$IS^{\bar 0}_* (X,\partial X;\Z)$
represents the fundamental class
$\Gamma_X\in IH^{\bar 0}_n(X,\partial X;\Z)$; there is also an empty object 
of dimension $n$ for each $n$. 
The 
morphisms $(f:X\to Z,\xi)\to(f':X'\to 
Z,\xi')$ are morphisms $g:X\to X'$ in $\Aa_{\ip,Z}$ which, if the dimensions are
equal, have the property that 
$g_*\xi=\xi'$.

\begin{definition}
\label{aa38}
Define 
$\ad_{\ip,Z,\Fun}(K)\subset \pre_{\ip,Z,\Fun}^k(K)$ to be the set of functors 
$F$ such that

\quad (a) $F$ is balanced,

\quad (b) the composite of $F$ with the forgetful functor 
$\Aa_{\ip,Z,\Fun}\to \Aa_\ip$ is an element of
$\ad_\ip^k(K)$, and

\quad (c) for each oriented cell $(\sigma,o)$ of $K$, the equation
\[
\partial \xi_{\sigma,o}
=
\sum \xi_{\sigma',o'}
\]
holds, where $\sigma'$ runs through the cells of $\partial \sigma$ and
$o'$ is the orientation for which the incidence number $[o,o']$ is $(-1)^k$.
\end{definition}

\begin{prop}
%\label{fp1}
\label{aa23}
$\ad_{\ip,Z,Fun}$ is a connective ad theory.
\end{prop}
\begin{proof}
We only need to check parts (f) and (g) of \cite[Definition
3.10]{LM}.

For the proof of (f), we may assume that $K$ and $K'$ are as in the
corresponding part of the proof of
Theorem \ref{t1}.  Let $F$ be a $K'$-ad and write
\[
F(\sigma,o)=(f_\sigma:X_{\sigma,o}\to Z,\xi_{(\sigma,o)}).
\]
Let $Y_\sigma$ be the
underlying $\partial$-IP-space of $X_{\sigma,o}$ (forgetting the orientation).
Let
\[
Y_\tau=\colim_{\sigma\in K'} Y_\sigma.
\]
The proof of Theorem \ref{t1} shows that $Y_\tau$ is a $\partial$-IP-space and
that an orientation $o$ of $\tau$ determines an orientation of $Y_\tau$; we let
$X_{\tau,o}$ be $Y_\tau$ with this orientation.
Let $f_\tau:X_{\tau,o}\to Z$ be the map induced by the $f_\sigma$.
As in the proof of Proposition \ref{fp1}, let
\[
\xi_{\tau,o}=\sum_{\sigma\in K'} \xi_{\sigma,o_\sigma}.
\]
$\xi_{\tau,o}$ is in $IS^{\bar 0,\uu_{X_{\tau,o}}}_*(X_{\tau,o};\Z)$, because
by Lemma \ref{aa11} every contractible open set of $X_{\sigma,o_\sigma}$ is
contained in a contractible open subset of $X_{\tau,o}$.
The desired result now follows from the proofs of Theorem \ref{t1} and 
Proposition \ref{fp1}. 

For the proof of (g), let $F$ be a $K$-ad, and write
\[
F(\sigma,o)=(f_\sigma:X_{\sigma,o}\to Z, \xi_{\sigma,o}).
\]
As in the proof of Proposition \ref{fp1}, we only need to specify $J(F)$ for
oriented cells of the form $(\sigma\times I,o\times o')$ where $o'$ is the
standard orientation.  
Let 
\[
g_\sigma:X_{\sigma,o}\times I\to Z
\]
be the composite of $f_\sigma$ with the
projection.
Let $s:\Delta^1\to I$ be the standard oriented
homeomorphism. Recall the relative barycentric subdivision
construction from \cite[Section 16]{munkres}.  For each $\sigma$ we can 
apply iterated barycentric
subdivision to $\xi_{\sigma,o}\times s$, while holding fixed any faces that
land in $X_{\sigma,o}\times \{0,1\}$, to obtain an element
$\eta_{\sigma,o}\in IS^{\bar 0,\uu_{X_{\sigma,o}}}(X_{\sigma,o}\times I,\Z)$.  
Now let
\[
J(F)(\sigma\times I,o\times o')=
(g_\sigma:X_{\sigma,o}\times I\to Z, \eta_{\sigma,o}).
\]
\end{proof}

The forgetful functor 
\[
\Upsilon:\Aa_{\ip,Z,\Fun}\to \Aa_{\ip,Z}
\]
gives a morphism
$\ad_{\ip,Z,\Fun}\to \ad_{\ip,Z}$ of ad theories.

\begin{prop}
The map 
\[
\bQ_\Upsilon:\bQ_{\ip,Z,\Fun}\to \bQ_{\ip,Z}
\]
induced by $\Upsilon$ is a weak equivalence.
\end{prop}

\begin{proof}
We need to show that the map of bordism groups
\[
\Omega_\Upsilon:(\Omega_{\ip,Z,\Fun})_*\to (\Omega_{\ip,Z})_*
\]
is an isomorphism.

To see that it is onto, let $f:X\to Z$ be an element of $\ad_{\ip,Z}(*)$.  
Let $\xi\in IS^{\bar 0}(X;\Z)$ represent the
fundamental class, and apply iterated barycentric subdivision to 
$\xi$ to get an element $\eta\in IS^{\bar 0,\uu_X}_*(X;\Z)$.  Then $f:X\to Z$ 
is equal to $\Upsilon(f:X\to Z,\eta)$.

Next we claim that if $(f:X\to Z,\xi_1)$ and
$(f:X\to Z,\xi_2)$ are elements of $\ad_{\ip,Z,\Fun}(*)$ with the same
$f$ then they represent the same element of $(\Omega_{\ip,Z,\Fun})_*$.  
Let $g:X\times I\to 
Z$ be the composite of $f$ with the projection.
The proof of \cite[Lemma 8.2]{LM} gives an element 
$\eta\in IS^{\bar 0}(X\times I;\Z)$ with $\partial \eta=\xi_2-\xi_1$.  
Applying iterated barycentric subdivision, holding fixed any faces that land 
in $X\times \{0,1\}$, gives an element $\eta'\in IS^{\bar 0,\uu_{X\times 
I}}(X\times I;\Z)$ with $\partial \eta'=\xi_2-\xi_1$.  
Now let $F$ be the $I$-ad 
which takes the cells 0, 1 and $I$ (with their standard orientations) 
respectively to $(f,\xi_1)$, $(f,\xi_2)$, and $(g,\eta')$; then $F$ is the 
desired bordism.

Now to see that $\Omega_\Upsilon$ is 1-1, let $(f_1:X_1\to Z,\xi_1)$ and 
$(f_2:X_2\to Z,\xi_2)$ be elements of $\ad_{\ip,Z,\Fun}(*)$, and let
$F\in \ad_{\ip,Z}(I)$ be an ad which takes 0 and 1 (with their standard
orientations) to $f_1$ and $f_2$.  Write $g:Y\to Z$ for the image of $I$ (with
its standard orientation) under $F$. 
Let $\eta\in IS^{\bar 0}(Y;\Z)$ be a representative of the fundamental class 
of $Y$, and apply iterated barycentric subdivision to get an element 
$\eta'\in IS^{\bar 0,\uu_Y}(Y;\Z)$.  Then we can write
\[
\partial \eta'=\theta_1-\theta_2,
\]
where $\theta_i\in IS^{\bar 0,\uu_{X_i}}(X_i;\Z)$ is a representative for the
fundamental class of $X_i$.  Now we have an ad $G$ in $\ad_{\ip,Z,\Fun}(I)$
which takes 0, 1 and $I$ (with their standard orientations) respectively to 
$(f_1:X_1\to Z,\theta_1)$, 
$(f_2:X_2\to Z,\theta_2)$, and $(f:Y\to Z,\eta')$. This shows
that $(f_1:X_1\to Z,\theta_1)$ and
$(f_2:X_2\to Z,\theta_2)$ represent the same element of
$(\Omega_{\ip,Z,\Fun})_*$, and by the previous paragraph we see that
$(f_1:X_1\to Z,\xi_1)$ and
$(f_2:X_2\to Z,\xi_2)$ represent the same element of
$(\Omega_{\ip,Z,\Fun})_*$.
\end{proof}

\begin{remark}
\label{aa37}
We make 
$\Aa_{\ip,Z,\mathrm{Fun}}$ a functor of $Z$ by letting $g:Z\to W$ take
$(f:X\to Z,\xi)$ to $(g\circ f:X\to W,\xi)$; this takes
ads to ads by Definition \ref{aa38}.
\end{remark}

\subsection{The functor $\Phi'$}
\label{aa16}

We need a preliminary definition.

\begin{definition}
\label{aa25}
Let $f:X\to Z$ be an object of $\Aa_{\ip,Z}$. 

(i) For each $z\in Z$, let $\tilde Z_z\to 
Z$ be the universal cover
consisting of path homotopy classes of paths starting at $z$.

% (ii)
% Let 
% \[
% f_z:f^*(\tilde{Z}_z)\to Z
% \] 
% be the composite of the covering map with $f$.  Let
% \[
% f_{z,z'}:f^*(\tilde{Z}_z)\times f^*(\tilde{Z}_{z'})\to Z\times Z
% \]
% be the composite of the covering map with $f\times f$.
%and let $\tilde X_z$ be the pullback of $\tilde Z_z$ to $X$.

(ii) Let
$C_{f:X\to Z}$ be the chain complex of left $\Z[\pi_1 Z]$ modules with 
\[
(C_{f:X\to Z})_z=IS^{\bar n,\uu_{f^*(\tilde{Z}_z)}}_*(f^*(\tilde{Z}_z);\Z).
\]

(iii)
Let $E_{f:X\to Z}$ be the chain complex of left $\Z[\pi_1
Z]^\mathrm{op}\otimes \Z[\pi_1 Z]$ modules 
with 
\[
(E_{f:X\to Z})_{(z,z')}=IS^{Q_{\bar n,\bar n},
\uu_{f^*(\tilde{Z}_z)\times f^*(\tilde{Z}_{z'})}}_*(f^*(\tilde{Z}_z)\times 
f^*(\tilde{Z}_{z'});\Z).
\]
\end{definition}

Now define a category $\Aa_{\ip,Z,\Fun,\sch}$ as follows.
The objects are 
sextuples
\[
(f:X\to Z,\xi, 
{\mathbb C},\lambda_1, \mathbb E,\lambda_2), 
\]
where 
$(f:X\to Z,\xi)$ is an
object of $\Aa_{\ip,Z,\Fun}$, 
$({\mathbb C},\lambda_1)$ is a schematic model for
$C_{f:X\to Z}$, 
and $(\mathbb E,\lambda_2)$ is a schematic model for
$E_{f:X\to Z}$.
A morphism 
$(f:X\to Z,\xi, 
{\mathbb C},\kappa, \mathbb E,\lambda)
\to 
(f':X'\to Z,\xi', 
{\mathbb C}',\kappa', {\mathbb E}',\lambda')$
is a morphism $(f:X\to Z,\xi)\to (f':X'\to Z,\xi')$ 
in $\Aa_{\ip,Z,\Fun}$.

There is a forgetful functor $\Aa_{\ip,Z,\Fun,\sch}\to
\Aa_{\ip,Z,\Fun}$, and we define 
$\ad_{\ip,Z,\Fun,\sch}(K)\subset 
\pre_{\ip,Z,\Fun,\sch}(K)$
to be the set of functors $F$ such that
the composite of $F$ with the forgetful functor is an element of
$\ad_{\ip,Z,\Fun}^k(K)$.
Lemma \ref{rrrr} shows that the forgetful functor is an equivalence of
categories, and then 
the proof of Proposition \ref{nn11} shows that
$\ad_{\ip,Z,\Fun,\sch}$ is an ad theory and that
the forgetful functor induces an isomorphism of bordism groups
$(\Omega_{\ip,Z,\Fun,\sch})_*\to (\Omega_{\ip,Z,\Fun})_*$.

In order to make 
$\ad_{\ip,Z,\Fun,\sch}$ a functor of $Z$ we need a lemma.  Recall
Notation \ref{nn6}.

\begin{lemma}
\label{aa36}
Let $g:Z\to W$. 

{\rm{(i)}} For each $z\in Z$ there is a canonical isomorphism
\[
% \begin{equation}
%\label{aa34}
\Z[\pi_1(W,g(z))]\otimes_{\Z[\pi_1(Z,z)]}
(C_{f:X\to Z})_z
\cong
(C_{g\circ f:X\to W})_{g(z)},
% \end{equation}
\]
and for each pair $(z,z')$ there is a canonical isomorphism
\begin{multline*}
\Z[\pi_1(W\times W,(g(z),g(z')))]\otimes_{\Z[\pi_1(Z\times Z,(z,z'))]}
(E_{f:X\to Z})_{(z,z')}
\\
\cong
(E_{g\circ f:X\to W})_{(g(z),g(z'))}.
\end{multline*}

{\rm{(ii)}} There are canonical isomorphisms
\[
\mathrm{Kan}_{g_*} C_{f:X\to Z}
\cong
C_{g\circ f:X\to W}
\]
and 
\[
\mathrm{Kan}_{g_*\otimes g_*} E_{f:X\to Z}
\cong
E_{g\circ f:X\to W}
\]
\end{lemma}

\begin{proof}
We give the proofs for $C_{f:X\to Z}$; the proofs for $E_{f:X\to Z}$
are similar.

Proof of (i).
% An element of $\tilde X_z$ has the form $(x,[\delta]$ where $\delta$ is a 
% path starting at $f(x)$ and $[\ ]$ denotes the equivalence class.
%For $z\in Z$ We write $\tilde Z_{g(z)}$ for $g^*(\tilde W_{g(z)})$ and 
%$\tilde X_{g(z)}$ for $(g\circ f)^* \tilde W_{g(z)}$.
The left action of $\pi_1(Z,z)$ on $\tilde Z_z$ is given by
\[
\delta\cdot\epsilon
=
\epsilon\delta^{-1}.
\]
There is a 
homeomorphism
\[
\pi_1(W,g(z))\times_{\pi_1(Z,z)} \tilde Z_z
\to
g^*\tilde W_{g(z)}
\]
which takes $(\alpha,\epsilon)$ to $(\epsilon(1), (g\circ\epsilon)\alpha^{-1})$.
Pulling back to $X$ gives a homeomorphism
\begin{multline*}
\pi_1(W,g(z))\times_{\pi_1(Z,z)} f^*\tilde{Z}_z
\approx
f^*(\pi_1(W,g(z))\times_{\pi_1(Z,z)} \tilde Z_z)
\xrightarrow{\approx}
f^*g^*\tilde W_{g(z)}
\\
\approx
(g\circ f)^*\tilde W_{g(z)}.
%=\tilde X_{g(z)}.
\end{multline*}
Since (by Definition \ref{aa22}) the restriction of $f^*\tilde{Z}_z$ to 
each $U\in \uu_{f^*\tilde{Z}_z}$ is trivial, the canonical map
\[
\Z[\pi_1(W,g(z))]\otimes_{\Z[\pi_1(Z,z)]}
IS^{\bar n}_* (T; \Z)
\to
IS^{\bar n}_*(\pi_1(W,g(z))\times_{\pi_1(Z,z)} T;\Z)
\]
is an isomorphism for each $T\in \uu_{f^*\tilde{Z}_z}$.  Now the following 
composite (where the first and fourth isomorphisms are given by 
\cite[Proposition 6.3]{friedmanmcclure}) is the required isomorphism.
\begin{multline*}
\Z[\pi_1(W,g(z))]\otimes_{\Z[\pi_1(Z,z)]}
(C_{f:X\to Z})_z
\\
=
\Z[\pi_1(W,g(z))]\otimes_{\Z[\pi_1(Z,z)]}
IS_*^{\bar n,\uu_{f^*\tilde{Z}_z}}(f^*\tilde{Z}_z;\Z)
\\
\cong
\Z[\pi_1(W,g(z))]\otimes_{\Z[\pi_1(Z,z)]}
\colim IS_*^{\bar n}(T; \Z)
\\
\cong
\colim(
\Z[\pi_1(W,g(z))]\otimes_{\Z[\pi_1(Z,z)]}
IS_*^{\bar n}(T; \Z))
\\
\cong
\colim IS_*^{\bar n}(\pi_1(W,g(z))\times_{\pi_1(Z,z)} T;\Z)
\\
\cong
IS_*^{\bar n,\uu_{(g\circ f)^*\tilde W_{g(z)}}}((g\circ f)^*\tilde W_{g(z)};\Z)
%\\
=(C_{g\circ f:X\to W})_{g(z)}
\end{multline*}

Proof of (ii).
As $z$ varies, the following composite (where the second map is the 
isomorphism from part (i)) gives a map of $\Z[\pi_1 Z]$ modules:
\[
(C_{f:X\to Z})_z
\to
\Z[\pi_1(W,g(z))]\otimes_{\Z[\pi_1(Z,z)]}
(C_{f:X\to Z})_z
\xrightarrow{\cong}
(C_{g\circ f:X\to W})_{g(z)},
\]
and this induces a map of $\Z[\pi_1 W]$ modules
\[
h:\mathrm{Kan}_{g_*} C_{f:X\to Z}
\to
C_{g\circ f:X\to W}.
\]
It suffices to show that $h_{g(z)}$ is an isomorphism for any $z$.
By Remark \ref{aa33}, it suffices to show that the composite
\begin{multline*}
\Z[\pi_1(W,g(z)]\otimes_{\Z[\pi_1(Z,z)]} (C_{f:X\to Z})_z
\cong
(\mathrm{Kan}_{g_*} (C_{f:X\to Z}))_{g(z)}
\\
\xrightarrow{h_{g(z)}}
(C_{g\circ f:X\to W})_{g(z)}
\end{multline*}
is an isomorphism.  But this composite is the same as the map in part (i),
since both are $\Z[\pi_1(W,g(z)]$ linear and they agree on $(C_{f:X\to
Z})_z$.
\end{proof}

Now for a map $g:Z\to W$ we define a functor
\[
g_{\Fun,\sch}:
\Aa_{\ip,Z,\Fun,\sch}
\to
\Aa_{\ip,W,\Fun,\sch}
\]
by
\begin{multline*}
g_{\Fun,\sch}(f:X\to Z,\xi, {\mathbb C},\lambda_1, \mathbb E,\lambda_2)
\\
=
(g\circ f:X\to W,\xi, g_\sch{\mathbb C},
\lambda_1',(g\times g) _\sch\mathbb E,\lambda_2'),
\end{multline*}
where $\lambda_1'$ is the composite
\begin{multline*}
\Z[\pi_1 W]\langle g_\sch\mathbb C\rangle
\cong
\mathrm{Kan}_{g_*}\Z[\pi_1 Z]\langle \mathbb C\rangle
\xrightarrow{\Kan_{g_*}\lambda_1}
\mathrm{Kan}_{g_*}C_{f:X\to Z}
\\
\cong
C_{g\circ f:X\to W}
\end{multline*}
(where the first isomorphism is Lemma
\ref{aa35} and the second is Lemma \ref{aa36})
and $\lambda_2$ is the composite 
\begin{multline*}
(\Z[\pi_1 W]^\mathrm{op}\otimes \Z[\pi_1 W])\langle
(g\times g)_\sch\mathbb E\rangle
\\
\cong
\mathrm{Kan}_{g_*\otimes g_*}(\Z[\pi_1 Z]^\mathrm{op}\otimes \Z[\pi_1
Z])\langle \mathbb E\rangle
\\
\xrightarrow{\mathrm{Kan}_{g_*\otimes g_*}\lambda_2}
\mathrm{Kan}_{g_*\otimes g_*}E_{f:X\to Z}
\cong
E_{g\circ f:X\to W}
\end{multline*}

$g_{\Fun,\sch}$ takes ads to ads by Remark \ref{aa37}, and thus
$\ad_{\ip,Z,\Fun,\sch}$ is a functor of $Z$ as required.

We now define 
\[
\Phi'(Z)=\bQ_{\ip,Z,\Fun,\sch}
\]

\subsection{Symmetric signature}
\label{aa3}

First we give the analog of a definition from Section \ref{symsig}.

\begin{definition}
\label{nn34}
For an object $f:X\to Z$ of $\Aa_{\ip,Z}$,
define $A_{f:X\to Z}$ 
to be the chain complex of 
left $\Z[\pi_1 Z]$ modules with
\[
(A_{f:X\to Z})_z=
IS_*^{\bar 0,\uu_{f^*Z_z}} (f^*Z_z;\Z).
\]
\end{definition}

For each $z\in Z$, the canonical map
\[
\Z\otimes_{\Z[\pi_1(Z,z)]} (A_{f:X\to Z})_z
=
\Z\otimes_{\Z[\pi_1(Z,z)]} IS_*^{\bar 0,\uu_{f^*Z_z}} (f^*Z_z;\Z)
\to
IS_*^{\bar 0,\uu_X} (X;\Z)
\]
is an isomorphism, because it is the same as the following composite
(where the first and fourth isomorphisms are given by \cite[Proposition
6.3]{friedmanmcclure}):
\begin{multline*}
\Z\otimes_{\Z[\pi_1(Z,z)]} IS_*^{\bar 0,\uu_{f^*Z_z}} (f^*Z_z;\Z)
\cong
%\Z\otimes_{\Z[\pi_1(Z,z)]} \colim_{V\in \uu_z} IS_*^{\bar 0} (V;\Z)
\Z\otimes_{\Z[\pi_1(Z,z)]} \colim IS_*^{\bar 0} (V;\Z)
\\
\cong
%\colim_{V\in \uu_z} (\Z\otimes_{\Z[\pi_1(Z,z)]} IS_*^{\bar 0} (V;\Z))
\colim (\Z\otimes_{\Z[\pi_1(Z,z)]} IS_*^{\bar 0} (V;\Z))
\cong
%\colim_{U\in \uu} IS_*^{\bar 0} (U;\Z)
\colim IS_*^{\bar 0} (U;\Z)
\cong
IS_*^{\bar 0,\uu_X} (X;\Z).
\end{multline*}
Recall that we write $\ZZ$ for the 
constant $\Z[\pi_1 Z]$ module with value $\Z$.
Using Remark \ref{aa5}, we obtain an isomorphism
\[
\ZZ\otimes_{\Z[\pi_1 Z]} A_{f:X\to Z}
\cong
\Z\otimes_{\Z[\pi_1(Z,z)]} (A_{f:X\to Z})_z
\cong
IS_*^{\bar 0,\uu_X} (X;\Z)
\]
which the reader can check is independent of the choice of $z$.

\begin{definition}
\label{aa9}
Define a functor 
\[
\sig:\Aa_{\ip,Z,\Fun,\sch}\to\Aa^{\Z[\pi_1 Z]}_{\Rel}
\]
by 
\[
\sig(f:X\to Z,\xi, {\mathbb C},\lambda_1, \mathbb E,\lambda_2)
=({\mathbb C}, \mathbb E,\gamma,\phi),
\]
where $\gamma$ is the composite
\begin{multline*}
(\Z[\pi_1 Z]\langle {\mathbb C}\rangle)^t
\otimes
\Z[\pi_1 Z]\langle {\mathbb C}\rangle
\xrightarrow{\lambda_1\otimes \lambda_1}
(C_{f:X\to Z})^t\otimes C_{f:X\to Z}
\xrightarrow{\times}
E_{f:X\to Z}
\\
\xrightarrow{\lambda_2^{-1}}
(\Z[\pi_1 Z]^{\mathrm{op}}\otimes \Z[\pi_1 Z])\langle {\mathbb
E}\rangle
\end{multline*}
and $\phi$ is the image
of $\xi$ under the following composite (where the arrow is induced by the
diagonal map):
\begin{multline*}
IS_*^{\bar 0,\uu_X} (X;\Z)
\cong
\ZZ\otimes_{\Z[\pi_1 Z]} A_{f:X\to Z}
\to \ZZ\otimes_{\Z[\pi_1 Z]} E_{f:X\to Z}.
\\
\xrightarrow{1\otimes \lambda_2^{-1}}
\ZZ\otimes_{\Z[\pi_1 Z]} (\Z[\pi_1 Z]^{\mathrm{op}}\otimes \Z[\pi_1 Z])\langle
{\mathbb
E}\rangle
\end{multline*}
\end{definition}

The definitions show that
$\sig$ is a natural transformation as $Z$ varies. 

\begin{prop}
\label{aa14}
If $K$ is strict and $F\in\ad_{\ip,Z,\Fun,\sch}^k(K)$ then 
$\sig\circ F\in (\ad^{\Z[\pi_1 Z]}_{\Rel,\geq 0})^k(K)$.
\end{prop}

\begin{proof}
Let $F\in \ad_{\ip,Z,\Fun,\sch}(K)$.  
%Let us write $\Lambda$ for the 
%following composite (see Subsection \ref{nn17} for the first map and Remark 
%\ref{nn18} for the second):
%\[
%\Aa^{\Z[\pi_1 Z]}_{\Rel}
%\xrightarrow{\Theta_\mathrm{Rel}}
%\Aa^{\Z[\pi_1 Z]}_{\mathrm{Rel}}
%\to
%\Aa^{\Z[\pi_1 Z]}_{\mathrm{rel}}.
%\]
By Definition \ref{nn27}
we need to show that $\Theta_\mathrm{Rel}\circ \sig\circ F$ is an ad, and for
this
we need to show (see Remark \ref{nn18}) that $\Theta_\mathrm{Rel}\circ
\sig\circ F$ is 
closed and well-behaved and that its composite with the map 
\[
\Lambda:\Aa^{\Z[\pi_1 Z]}_{\mathrm{Rel}}
\to
\Aa^{\Z[\pi_1 Z]}_{\mathrm{rel}}
\] 
satisfies the analog of \cite[part (i)(b) of Definition 12.10]{LM2}.

$\Theta_\mathrm{Rel}\circ \sig\circ F$ is closed by property (c) of 
Definition \ref{aa38}.

For the proof that $\Theta_\mathrm{Rel}\circ \sig\circ F$
is well-behaved, first note that
by Definition \ref{aa38}(a), $F$ is balanced, so we can write
\[
F(\sigma,o)=(f_\sigma:X_{\sigma,o}\to Z,\xi_{\sigma,o},\mathbb
C_\sigma, (\lambda_1)_\sigma,\mathbb E_\sigma,(\lambda_2)_\sigma).
\]
Also write $Y_\sigma$ for the underlying $\partial$-IP-space of 
$X_{\sigma,o}$ (forgetting the orientation).
It suffices to show that the 
functors
$C_{f_\sigma:Y_\sigma\to Z}$ and 
$E_{f_\sigma:Y_\sigma\to Z}$ are well-behaved.  We will prove this
for the first, the proof for the second is similar.

Choose $z\in Z$. 
It suffices to show that the functor 
$IS_i^{\bar{n},\uu_{f_\sigma^*(\tilde 
Z_z)}}(f_\sigma^*(\tilde 
Z_z);\Z)$ is well-behaved.  
Let $\sigma\in K$, and
for each $\tau\subset \sigma$ let 
$\mathcal W_\tau$ be
the collection of all intersections of sets in $\uu_{Y_\tau}$, and for each
$U\in \mathcal W_\tau$ let $\tilde U$ be the inverse image of $U$ in 
$f_\tau^*(\tilde Z_z)$.
The map
\[
\colim_{U\in \mathcal W_\tau}
IS_i^{\bar{n}}(\tilde U;\Z)
\to
IS_i^{\bar{n},\uu_{f_\tau^*(\tilde Z_z)}}(f_\tau^*(\tilde Z_z);\Z)
\]
is an isomorphism for each $\tau$ by \cite[Proposition
6.3]{friedmanmcclure}, so 
the well-behavedness will be an immediate
consequence of Lemma \ref{aa15}(iii), with 
$
\mathcal S=\cup_{\tau\subset \sigma} \mathcal W_\tau
$
and 
$
\mathcal S'=\cup_{\tau\in\partial\sigma} \mathcal W_\tau,
$
once we verify the hypotheses of that Lemma.
Hypothesis (a) follows from Definition \ref{aa22}, and hypothesis (b) is
part (iii) of Lemma \ref{aa39} below.  For hypothesis (c), it suffices to show
that for each $U\in \S$ the map
\[
IS_i^{\bar{n}}(U;\Z)/\colim_{V\in \S \text{ with } V\subsetneq U}
IS_i^{\bar{n}}(V;\Z)
\to
S_i(U;\Z)/\colim_{V\in \S \text{ with } V\subsetneq U}
S_i(V;\Z)
\]
is a monomorphism (because the target of this map is free over $\Z$) and this
follows from parts (iii) and (ii) of Lemma \ref{aa39}.
This concludes the proof that $\sig\circ F$ is well-behaved.

To see that $\Lambda\circ \Theta_\mathrm{Rel}\circ \sig\circ F$ satisfies the
analog of \cite[part (i)(b) of Definition 12.10]{LM2},
we first observe that the map
\[
IS_*^{\bar{n},\uu_{f_\sigma^*(\tilde Z_z)}}(f_\sigma^*(\tilde Z_z);\Z)
\to
IS_*^{\bar{n}}(f_\sigma^*(\tilde Z_z);\Z)
\]
is a quasi-isomorphism by Remark \ref{aa30},
and that, if $\S'$ is as above, the composite
\[
\colim_{\tau\subset\partial\sigma}
IS_*^{\bar{n},\uu{f_\tau^*(\tilde Z_z)}}(f_\tau^*(\tilde Z_z);\Z)
\cong
\colim_{U\in\S'}
IS_*^{\bar{n}}(\tilde U;\Z)
\to
IS_*^{\bar{n}}(\partial f_\sigma^*(\tilde Z_z);\Z)
\]
is a  quasi-isomorphism by Lemma \ref{aa39}(iv). 
It therefore suffices to show that the map
\begin{multline*}
H_*(\Hom_{\Z[\pi_1(Z,z)]}(IS_*^{\bar{n}}(f_\sigma^*(\tilde 
Z_z);\Z),\Z[\pi_1(Z,z)])
\\
\xrightarrow{\backslash ((1\otimes\gamma_\sigma)_*)^{-1}([\phi_\sigma])}
IH_{\dim\sigma-k-*}^{\bar{n}}(f_\sigma^*(\tilde Z_z),\partial 
f_\sigma^*(\tilde Z_z);\Z)
\end{multline*}
is an isomorphism for each $\sigma$.  The rest of the proof is the same as the
corresponding part of the proof of Proposition \ref{fp3}, except that 
$Y_\sigma$ is replaced by $f_\sigma^*(\tilde Z_z)$.
\end{proof}

With notation as in the proof of well-behavedness in Proposition \ref{aa14}, 
say that a subset
$\T$ of $\S$ is {\em saturated} if for every $U\in\T$ we have $\{V\,|\,V\in 
\S \text{ and } 
V\subset U\}\subset \T$.  
For every saturated $\T$, 
write $X_\T$ for $\colim_{U\in \T} U$, and give this the filtration which
restricts to the filtration  of
Proposition \ref{p2}
on each $U\in \T$ (note that an open subset of a $\partial$-pseudomanifold is
a $\partial$-pseudomanifold by \cite[Lemma 2.7.7]{Fr}).  Write
$\tilde X_\T$ for $\colim_{U\in\T} \tilde U$.

\begin{lemma}
\label{aa39}
% Let $F\in\ad_{\ip,Z,\Fun}^k(K)$, and write 
% $F(\sigma,o)=(f_\sigma:X_\sigma\to Z,\uu_\sigma,\xi_{\sigma,o})$.

{\rm{(i)}} For every pair of saturated sets $\T'\subset \T$, there is a
neighborhood
$W$ of $X_{\T'}$ in $X_\T$
such that the inclusion $X_{\T'}\to W$ is a
stratified deformation retract, where the retraction $r:W\to X_{\T'}$ has the
property that $r(x)\in U$ whenever $x\in U\in \T$.

{\rm{(ii)}} For every saturated $\T$ and every $i\in \Z$, the intersection
of $IS^{\bar n}_i (X_K;\Z)$ and
$\sum_{U\in \T} S_i(U;\Z)$ (considered as subgroups of
$S_i(X_K;\Z)$)
is
$\sum_{U\in \T} IS^{\bar n}_i(U)$.

{\rm{(iii)}} If $K$ is strict then for every saturated $\T$ the map
\[
\colim_{U\in \T} IS^{\bar n}_*(U;\Z)
\to
\sum_{U\in \T} IS^{\bar n}_*(U;\Z)
\]
is an isomorphism.

{\rm{(iv)}} For every saturated $\T$ the map
\[
\colim_{U\in \T} IS_*^{\bar{n}}(\tilde U;\Z)
\to
IS_*^{\bar{n}}(\tilde X_\T)
\]
is a quasi-isomorphism.
\end{lemma}

The proof is the same as the proof of Lemma \ref{aa11}.

\subsection{Proof of Lemma \ref{nn15}}
\label{nn29}

We need to show that $\alpha^{-1}[X]_\ip$ is taken to $[X]_\mathbb{L}$ by the
composite
\begin{multline*}
X_+\wedge \bQ_\ip
\xleftarrow{\simeq}
X_+\wedge \bQ_{\ip,*,\Fun,\sch}
\xrightarrow{1\wedge\sig}
X_+\wedge \bQ_\Relz^{\Z}
\xleftarrow{\simeq}
X_+\wedge \bQ^\Z_{\mathrm{fin},\sch,\geq 0}
\\
\to
X_+\wedge \bQ^\Z_{\sch,\geq 0}
\to
X_+\wedge \bQ^\Z_{\geq 0}.
\end{multline*}
This is the same as the composite
\begin{multline*}
X_+\wedge \bQ_\ip
\xleftarrow{\simeq}
X_+\wedge \bQ_{\ip,*,\Fun,\sch}
\xrightarrow{1\wedge\sig}
X_+\wedge \bQ_\Relz^{\Z}
\to
X_+\wedge \bQ_{\mathrm{Rel},\geq 0}^\Z
\\
\xleftarrow{\simeq}
X_+\wedge \bQ^\Z_{\mathrm{fin},\geq 0}
\to
X_+\wedge \bQ^\Z_{\geq 0},
\end{multline*}
which by Remark \ref{nn33} is homotopic to the composite
\begin{multline*}
\label{nn30}
X_+\wedge \bQ_\ip
\xleftarrow{\simeq}
X_+\wedge \bQ_{\ip,*,\Fun,\sch}
\xrightarrow{1\wedge\sig}
X_+\wedge \bQ_\Relz^{\Z}
\to
X_+\wedge \bQ_{\mathrm{Rel},\geq 0}^\Z
\\
\xrightarrow{1\wedge \Lambda}
X_+\wedge \bQ_{\mathrm{rel},\geq 0}^\Z
\xleftarrow{\simeq}
X_+\wedge \bQ^\Z_{\geq 0}.
\end{multline*}

Now Lemma \ref{nn15} follows from Definition \ref{def.xipxl}(ii) and

\begin{lemma}
\label{aa43}
The diagram
\[
\xymatrix{
\bQ_{\ip,*,\Fun,\sch}
\ar[r]^-{\sig}
\ar[dd]
&
\bQ^{\Z}_{\mathrm{Rel},\sch,\geq 0}
\ar[d]
\\
&
\bQ^{\Z}_{\mathrm{Rel},\geq 0}
\ar[d]_\Lambda
\\
\bQ_{\ipFun}
\ar[r]^-\sig
&
\bQ^{\Z}_{\mathrm{rel},\geq 0}
}
\]
commutes up to homotopy, where the upper map was constructed in Subsection
\ref{aa3}, the lower map in Subsection \ref{ss2}, and the left map is induced by
the forgetful functor.
\end{lemma}

\begin{proof}
By the analog of \cite[Proposition 14.5]{LM2}, it suffices to give a natural 
quasi-isomorphism
$\nu$ from the composite
\[
\Aa_{\ip,*,\Fun,\sch}\xrightarrow{\sig}
\Aa_{\mathrm{Rel},\sch}^\Z
\to
\Aa_{\mathrm{Rel}}^\Z
\xrightarrow{\Lambda}
\Aa_{\mathrm{rel}}^\Z
\]
to the composite
\[
\Aa_{\ip,*,\Fun,\sch}
\to
\Aa_{\ipFun}
\xrightarrow{\sig}
\Aa_{\mathrm{rel}}^\Z.
\]
By Definition \ref{aa25}, $C_{f:X\to *}=IS_*^{\bar n,\uu_X}(X;\Z)$ and 
$E_{f:X\to *}=IS_*^{Q_{\bar n,\bar n},\uu_{X\times X}}(X\times X;\Z)$, and by
Definition \ref{nn34} we have $A_{f:X\to *}=IS_*^{\bar 0,\uu_X}(X;\Z)$.
Thus the first composite takes the object
\[
(X\to *,\xi,\mathbb C,\lambda_1,\mathbb E,\lambda_2)
\]
to
\[
(\Z\langle \mathbb C\rangle,\Z\langle \mathbb E\rangle,\beta,\varphi),
\]
where $\beta$ is the map
\begin{multline*}
(\Z\langle \mathbb C\rangle)^t\otimes \Z\langle \mathbb C\rangle
\xrightarrow{\lambda_1\otimes\lambda_1}
(IS_*^{\bar n,\uu_X}(X;\Z))^t\otimes IS_*^{\bar n,\uu_X}(X;\Z)
\xrightarrow{\times}
IS_*^{Q_{\bar n,\bar n},\uu_{X\times X}}(X\times X;\Z)
\\
\xrightarrow{\lambda_2^{-1}}
\Z\langle \mathbb E\rangle
\end{multline*}
and $\varphi$ is the image of $\xi$ under the diagonal map
\[
IS_*^{\bar 0,\uu_X}(X;\Z)
\to
IS_*^{Q_{\bar n,\bar n},\uu_{X\times X}}(X\times X;\Z).
\]
The second composite takes $(X\to *,\xi,\mathbb C,\lambda_1,\mathbb
E,\lambda_2)$ to
\[
(IS^{\bar n}_*(X;\Z), IS^{Q_{\bar{n},\bar{n}}}_*(X\times
X;\Z),\beta',\varphi'),
\]
where $\beta'$ is the map
\[
(IS_*^{\bar n}(X;\Z))^t\otimes IS_*^{\bar n}(X;\Z)
\xrightarrow{\times}
IS_*^{Q_{\bar n,\bar n}}(X\times X;\Z)
\]
and $\varphi'$ is the image of $\xi$ under the map
\[
IS_*^{\bar 0,\uu_X}(X;\Z)
\to
IS_*^{\bar 0}(X;\Z)
\to
IS_*^{Q_{\bar n,\bar n}}(X\times X;\Z).
\]

The natural quasi-isomorphism is given by the maps
\[
\Z\langle \mathbb C\rangle
\xrightarrow{\lambda_1}
IS^{\bar n,\uu_X}_*(X;\Z)
\to
IS^{\bar n}_*(X;\Z)
\]
and 
\[
\Z\langle \mathbb E\rangle
\xrightarrow{\lambda_2}
IS^{Q_{\bar{n},\bar{n}},\uu_{X\times X}}_*(X\times X;\Z)
\to
IS^{Q_{\bar{n},\bar{n}}}_*(X\times
X;\Z).
\]
\end{proof}

\subsection{Proof of Lemma \ref{nn16}}
\label{nn3}

% \begin{lemma}
%\label{nn16}
% The image of $[X]_\ip$ along
% the bottom row followed by the map $\bQ_{\sch,\geq 0}^{\Z[\pi_1 X]}
% \to
% \bQ_{\geq 0}^{\Z[\pi_1 X]}$
 % is $\sigma_\ip^*(X)$.
% \end{lemma}

Let $X$ be a compact connected oriented IP-space of dimension $n$. Recall
Definition \ref{def.xipxl}, and let $[X\xrightarrow{\mathrm{id}} X]$ denote 
the class of the identity map
in $(\Omega_\ip)_n(X)$.  We need to show that the image of
$[X\xrightarrow{\mathrm{id}} X]$ under the
following composite is $\sigma_\ip^*(X)$. 
\begin{multline*}
(\Omega_{\ip,X})_n
\xleftarrow{\simeq}
(\Omega_{\ip,X,\Fun,\sch})_n
\xrightarrow{\sig}
(\Omega_\Relz^{\Z[\pi_1 X]})_n
\xleftarrow{\simeq}
(\Omega_{\mathrm{fin},\sch,\geq 0}^{\Z[\pi_1 X]})_n
\to
(\Omega_{\sch,\geq 0}^{\Z[\pi_1 X]})_n
\\
\to
(\Omega_{\geq 0}^{\Z[\pi_1 X]})_n
=
L^n(\Z[\pi_1 X])
\end{multline*}

Using Remark \ref{nn33} we see that this is the same as the composite
\begin{multline*}
\label{nn35}
(\Omega_{\ip,X})_n
\xleftarrow{\simeq}
(\Omega_{\ip,X,\Fun,\sch})_n
\xrightarrow{\sig}
(\Omega_\Relz^{\Z[\pi_1 X]})_n
\to
(\Omega_{\mathrm{Rel},\geq 0}^{\Z[\pi_1 X]})_n
\xrightarrow{\Lambda}
(\Omega_{\mathrm{rel},\geq 0}^{\Z[\pi_1 X]})_n
\\
\xleftarrow{\cong}
(\Omega_{\geq 0}^{\Z[\pi_1 X]})_n
=
L^n(\Z[\pi_1 X])
\end{multline*}

Recall that the first map in this composite is induced by the forgetful
functor, so the image of $[X\xrightarrow{\mathrm{id}} X]$ under this map is the
class of any object of the form
\[
(\mathrm{id}:X\to X,\xi,\mathbb C,\lambda_1,\mathbb E,\lambda_2).
\]
By Definition \ref{d5} it suffices to show

% Let $X$ be a compact connected oriented IP-space of dimension $n$.  Choose
% an object of $\Aa_{\ip,Z,\Fun,\sch}$ of the form
% \[
% (\mathrm{id}:X\to X,\uu,\xi,x,\mathbb C,\lambda_1,\mathbb E,\lambda_2).
% \]
% Such an object is a $*$-ad and
% hence it represents an element 
% \[
% [(\mathrm{id}:X\to X,\uu,\xi,x,\mathbb C,\lambda_1,\mathbb E,\lambda_2)]
% \in (\Omega_{\ip,Z,\Fun,\sch})_n
% \]
% By Remark \ref{aa41} there is a canonical isomorphism 
% $(\Omega_{\ip,Z,\Fun,\sch})_n
% \cong
% \pi_n \mathbb \bQ_{\ip,Z,\Fun,\sch}$.

\begin{lemma}
\label{aa40}
The image of the class of $(\mathrm{id}:X\to X,\xi,\mathbb 
C,\lambda_1,\mathbb E,\lambda_2)$ under the composite
\[
%\begin{multline*}
(\Omega_{\ip,X,\Fun,\sch})_n
\xrightarrow{\sig}
(\Omega_\Relz^{\Z[\pi_1 X]})_n
\to
(\Omega_{\mathrm{Rel},\geq 0}^{\Z[\pi_1 X]})_n
\xrightarrow{\Lambda}
(\Omega_{\mathrm{rel},\geq 0}^{\Z[\pi_1 X]})_n
%\end{multline*}
\]
is the class of the object 
$(C_X,D_X,\beta_X,\varphi)$
given by Definition \ref{fd1}.
\end{lemma}

\begin{proof}
The image of $(\mathrm{id}:X\to X,\xi,\mathbb C,\lambda_1,\mathbb
E,\lambda_2)$ under the composite is represented by the following $*$-ad 
(where $\beta$ is given by Definition \ref{aa9} and Remark \ref{nn18}, and 
$\phi$ is given by Definition \ref{aa9})
\[
F=(\Z[\pi_1 X]\langle \mathbb C\rangle,
\ZZ\otimes_{\Z[\pi_1 X]}
(\Z[\pi_1 X]^\mathrm{op}\otimes \Z[\pi_1 X])\langle \mathbb E\rangle,
\beta,\phi).
\]
% and $\sigma^*_\ip(X)$ is represented by 
% \[
% G=(C_X,\Z\otimes_{\Z[\pi_1 X]} E_X,\beta_X,\varphi).
% \]
The following maps give a morphism of ads from this to
$(C_X,D_X,\beta_X,\varphi)$:
\[
\Z[\pi_1 X]\langle \mathbb C\rangle
\xrightarrow{\lambda_1}
C_{\mathrm{id}:X\to X}
\to 
C_X
\]
and
\begin{multline*}
\ZZ\otimes_{\Z[\pi_1 X]}
(\Z[\pi_1 X]^\mathrm{op}\otimes \Z[\pi_1 X])\langle \mathbb E\rangle
\xrightarrow{1\otimes \lambda_2}
\ZZ\otimes_{\Z[\pi_1 X]}
E_{\mathrm{id}:X\to X}
\\
\to
\ZZ\otimes_{\Z[\pi_1 X]} E_X.
\end{multline*}
The result now follows from \cite[Lemma 13.5]{LM2}.
\end{proof}

\section{The Stratified Novikov Conjecture}
\label{s9}

Let $G$ be a discrete group and $BG$ its classifying space.
Recall that the \emph{strong Novikov conjecture} for $G$ asserts that the
assembly map
\[ \alpha: \syml_n (BG) \longrightarrow L^n (\intg [G]) \]
is rationally injective. 
The $\rat$-localization of any spectrum is a graded Eilenberg-MacLane spectrum,
see e.g. \cite[Theorem 7.11]{rudyak}. Thus there is a splitting
\begin{equation} \label{equ.qsplit}
{\mathbb L}^\bullet(\Z) \otimes \rat \simeq \prod_{j\geq 0} K(\rat, 4j). 
\end{equation}
A morphism $\phi:E\to F$ between two ring $\rat$-local spectra $E,F$ is a ring morphism
if and only if $\phi_*: \pi_* (E)\to \pi_* (F)$ is a ring homomorphism (\emph{loc. cit.}).
The product on $\pi_* ({\mathbb L}^\bullet)$ is
\[ L^p (\intg)\otimes L^q (\intg) \longrightarrow  L^{p+q} (\intg), \]
given by sending chain complexes $C,D$ to their tensor product $C\otimes D$.
If $p$ and $q$ are divisible by $4$ (which is the only nontrivial case after tensoring with $\rat$),
then $L^p (\intg) \cong \intg$, a generator is given by the symmetric signature
$\sigma^* (\cplx P^{p/2})\in L^p (\intg)$, and the product sends
$\sigma^* (\cplx P^{p/2})\otimes \sigma^* (\cplx P^{q/2})$ to
$\sigma^* (\cplx P^{p/2} \times \cplx P^{q/2})$.
Under the isomorphism induced by the splitting (\ref{equ.qsplit}) on $\pi_*$, this product induces a
product on the homotopy groups of the graded Eilenberg-MacLane spectrum, which underlies a
spectrum level ring structure. By the above remark,
the equivalence (\ref{equ.qsplit}) is then a morphism of ring spectra.
It induces natural isomorphisms of homology theories
\[
S_X:\syml_n (X)\otimes \rat \stackrel{\cong}{\longrightarrow}
 \bigoplus_{j\geq 0} H_{n-4j} (X;\rat). 
\]
and cohomology theories
\[
S^X:(\syml)^n (X)\otimes \rat \stackrel{\cong}{\longrightarrow}
 \prod_{j\geq 0} H^{n+4j} (X;\rat). 
\]
Since $S^X$ is induced by a morphism of ring spectra, it maps the unit
$1\in (\syml)^0 (\pt)\otimes \rat$ to the unit $1\in H^0 (\pt;\rat)$ and preserves products.\\

An $n$-dimensional compact IP-space $X$ possesses characteristic classes
$L_j (X)\in H_j (X;\rat),$ which are the Poincar\'e duals of the Hirzebruch
$L$-classes when $X$ is a smooth manifold. These classes have been
introduced by Goresky and MacPherson in \cite{gm1}
(at least for spaces without odd codimensional strata, but the method works
whenever one has a self-dual intersection chain sheaf, see \cite{banagl-tiss}).
Goresky and MacPherson adapt a method of Thom, which exploits the bordism invariance of the
signatures of transverse inverse images of maps to spheres.
For singular $X$, these classes need not lift to the cohomology of $X$
under capping with the (ordinary) fundamental class. We shall denote the total
$L$-class by $L(X)$. An inclusion $Y\hookrightarrow X$ of PL spaces
is called \emph{normally nonsingular} if $Y$ possesses a regular neighborhood in $X$
which can be given the structure of a disk-block-bundle over $Y$.
\begin{lemma} \label{lem.lfundhitslclass}
Let $X$ be an $n$-dimensional compact IP-space.
The isomorphism $S_X$ maps
the rational $L$-theory fundamental class $[X]_{\mathbb{L}}$ to the
Goresky-MacPherson $L$-class $L(X)$. 
\end{lemma}
\begin{proof}
Let $Y^m$ be a compact IP-space and
$j: Y^m \hookrightarrow X^n$ a normally nonsingular inclusion with trivial
normal bundle $\nu$. Let $D\nu = D^{n-m} \times Y$ be
the total space of its disk bundle, $S\nu = S^{n-m-1} \times Y$
the total space of its sphere bundle. Note that $D\nu$ is a $\partial$-IP space.
Let $u\in H^{n-m} (D\nu,S\nu;\rat)$ be the Thom class of the normal bundle.
The composition
\[ H_k (X;\rat) \rightarrow H_k (X,X-Y;\rat)
  \stackrel{\cong}{\leftarrow} H_k (D\nu,S\nu;\rat) \hspace{2cm}\]
\[  \hspace{1cm} \overset{u\cap -}{\underset{\cong}{\rightarrow}} H_{k-n+m} (D\nu;\rat)
  \overset{\pi_\ast}{\underset{\cong}{\rightarrow}} H_{k-n+m} (Y;\rat)
  \stackrel{(-1)^s}{\to} H_{k-n+m} (Y;\rat), \]
where $s = \smlhf (n-m+1)(n-m)$ 
and $\pi: D\nu \to Y$ is the bundle projection,
defines a map
\[ j^!: H_k (X;\rat) \longrightarrow H_{k-n+m} (Y;\rat). \]
If $Z$ is any topological space and $R$ any coefficient ring, let $\epsilon_* :H_0 (Z;R)
\to R$ be the augmentation map.
By the Thom-Goresky-MacPherson construction, the $L$-classes are
uniquely characterized by the following two properties
(\cite[Proposition 8.2.11]{banagl-tiss}):
\begin{itemize}
\item If $j:Y^m \hookrightarrow X^n$ is a normally nonsingular inclusion with trivial
normal bundle, then
\[ L_{k-n+m} (Y) = j^! L_k (X). \]
\item $\epsilon_* L_0 (X)=\sigma (X),$ the signature of $X$.
\end{itemize}
Thus the lemma is proven if we show
\begin{enumerate}
\item If $j:Y^m \hookrightarrow X^n$ is a normally nonsingular inclusion with trivial
normal bundle, then
\[ (S_Y [Y]_{\mathbb{L}})_{k-n+m} = j^! (S_X [X]_{\mathbb{L}})_k; \] and
\item $\epsilon_* (S_X [X]_{\mathbb{L}})_0 =\sigma (X).$ 
\end{enumerate}
We turn to (1). Let $\nu$ be the normal bundle of $j$.
We write $S^0$ as $S^0 = \{ y_-, y_+ \}$.
Let $e\in H^1 (D^1, S^0;\rat)$ be the element obtained as the image of
the unit $1\in H^0 (y_+;\rat)$ under the composition
\[ H^0 (y_+;\rat) \stackrel{\cong}{\longleftarrow} H^0 (S^0, y_-;\rat) 
 \stackrel{\delta}{\longrightarrow} H^1 (D^1, S^0;\rat), \]
where the left arrow is an excision isomorphism and $\delta$ is the connecting
homomorphism of the triple $(D^1, S^0, y_-)$. The $(n-m)$-fold cross-product
$e\times \cdots \times e$ yields an element $e^{n-m} \in H^{n-m} (D^{n-m}, S^{n-m-1};\rat)$.
The Thom class $u$ arising in the definition of the map $j^!$ is then given by
$u = e^{n-m} \times 1 \in H^{n-m} (D\nu, S\nu;\rat),$ where $1\in H^0 (Y;\rat)$.
Analogous classes in $\mathbb{L}^\bullet$-homology and $\Omega_{\ip}$ can be
constructed in a similar fashion:
If $A$ is any abelian group, we shall briefly write $A_\rat$ for $A\otimes_{\intg} \rat$.
Let $e_\BL \in (\syml)^1 (D^1, S^0)_\rat$ be the element obtained as the image of
the unit $1\in (\syml)^0 (y_+)_\rat$ under the composition
\[ (\syml)^0 (y_+)_\rat \stackrel{\cong}{\longleftarrow} (\syml)^0 (S^0, y_-)_\rat 
 \stackrel{\delta}{\longrightarrow} (\syml)^1 (D^1, S^0)_\rat. \]
The $(n-m)$-fold cross-product
$e_\BL \times \cdots \times e_\BL$ yields an element $e^{n-m}_\BL \in (\syml)^{n-m} (D^{n-m}, S^{n-m-1})_\rat$.
Set
$u_\BL = e^{n-m}_\BL \times 1 \in (\syml)^{n-m} (D\nu, S\nu)_\rat,$ where 
$1\in (\syml)^0 (Y)_\rat$ is the unit. \\

We claim that
\begin{equation} \label{equ.selise}
S^{(D^1, S^0)} (e_\BL)=e.
\end{equation}
To see this, we observe that as $S$ is a natural transformation of cohomology theories, the diagram
\[ \xymatrix{
(\syml)^0 (y_+)_\rat \ar[d]_S^{\cong} & (\syml)^0 (S^0, y_-)_\rat 
  \ar[l]_{\cong}^{\operatorname{exc}} \ar[d]_S^{\cong} \ar[r]^{\delta} &
   (\syml)^1 (D^1, S^0)_\rat \ar[d]_S^{\cong} \\
\prod_{j\geq 0} H^{4j} (y_+;\rat) & 
  \prod_{j\geq 0} H^{4j} (S^0, y_-;\rat) \ar[l]_{\cong}^{\operatorname{exc}} \ar[r]^{\delta} & 
  \prod_{j\geq 0} H^{4j+1} (D^1, S^0;\rat)
} \]
commutes and hence
\[ S(e_\BL) = S \delta \operatorname{exc}^{-1} (1) = \delta \operatorname{exc}^{-1} S(1) =
    \delta \operatorname{exc}^{-1} (1) = e. \]

We show next that
\begin{equation} \label{equ.sulisu}
S^{(D\nu, S\nu)} (u_\BL) =u. 
\end{equation}
Since $S^{(D\nu, S\nu)}$ is induced by a morphism of ring spectra, it preserves products on cohomology.
Thus, using (\ref{equ.selise}),
\begin{align*}
S (u_\BL) &= S(e_\BL \times \cdots \times e_\BL \times 1) \\
&= S(e_\BL)\times \cdots \times S(e_\BL) \times S(1) \\
&= e\times \cdots \times e \times 1 \\
& = u.
\end{align*}

Let $e_\ip \in (\Omega_\ip)^1 (D^1, S^0)_\rat$ be the element obtained as the image of
the unit $1\in (\Omega_\ip)^0 (y_+)_\rat$ under the composition
\[ (\Omega_\ip)^0 (y_+)_\rat \stackrel{\cong}{\longleftarrow} (\Omega_\ip)^0 (S^0, y_-)_\rat 
 \stackrel{\delta}{\longrightarrow} (\Omega_\ip)^1 (D^1, S^0)_\rat. \]
The $(n-m)$-fold cross-product
$e_\ip \times \cdots \times e_\ip$ yields an element $e^{n-m}_\ip \in (\Omega_\ip)^{n-m} (D^{n-m}, S^{n-m-1})_\rat$.
Set
$u_\ip = e^{n-m}_\ip \times 1 \in (\Omega_\ip)^{n-m} (D\nu, S\nu)_\rat,$ 
where $1\in (\Omega_\ip)^0 (Y)_\rat$ is the unit.
The cap-product of $\mu :=[\id_{(D^{n-m}, S^{n-m-1})}]\in 
(\Omega_\ip)_{n-m} (D^{n-m}, S^{n-m})_\rat$
with $e^{n-m}_\ip$ is given by
\begin{equation} \label{enmcapmu}
e^{n-m}_\ip \cap \mu = (-1)^s [\pt \hookrightarrow D^{n-m}] \in 
    (\Omega_\ip)_0 (D^{n-m})_\rat, 
\end{equation}
as we shall now verify.
Set $\mu_1 = [\id_{(D^1, S^0)}]\in (\Omega_\ip)_1 (D^1, S^0)_\rat$. Then $\mu$ is the $(n-m)$-fold
cross product
\[ \mu = \mu_1 \times \cdots \times \mu_1 \]
and thus
\begin{eqnarray*} 
e^{n-m}_\ip \cap \mu &= &(e_\ip \times \cdots \times e_\ip)\cap (\mu_1 \times \cdots \times \mu_1) \\
&=& (-1)^{s_1} (e_\ip \cap \mu_1)\times \cdots \times (e_\ip \cap \mu_1), 
\end{eqnarray*}
where $s_1=\frac{1}{2} (n-m)(n-m-1)$. Let $i: (S^0,\varnothing)\hookrightarrow (S^0, y_-)$ be the canonical
inclusion.
To compute $e_\ip \cap \mu_1,$ we consider the diagram
\[ \xymatrix{
\Omega_\ip^1 (D^1, S^0)_\rat  & \otimes & 
   (\Omega_\ip)_1 (D^1,S^0)_\rat \ar[d]_{\partial} \ar[r]^{\cap} & (\Omega_\ip)_0 (D^1)_\rat \\
\Omega_\ip^0 (S^0)_\rat \ar[u]^{\delta'}   & \otimes & (\Omega_\ip)_0 (S^0)_\rat \ar[d]_{i_*} \ar[r]^{\cap}
      & 
    (\Omega_\ip)_0 (S^0)_\rat \ar[u]^{\iota} \\
\Omega_\ip^0 (S^0,y_-)_\rat \ar[d]_{\cong}^{\operatorname{exc}} \ar[u]^{i^*}  & \otimes & (\Omega_\ip)_0 (S^0,y_-)_\rat 
   \ar[r]^{\cap} 
     & (\Omega_\ip)_0 (S^0)_\rat \ar@{=}[u] \\
\Omega_\ip^0 (y_+)_\rat  
  & \otimes & (\Omega_\ip)_0 (y_+)_\rat \ar[u]^{\cong}_{\operatorname{exc}} \ar[r]^{\cap} & 
  (\Omega_\ip)_0 (y_+)_\rat, \ar[u]^{\iota_+}
} \]
whose middle and bottom portion commute, while the top portion anti-commutes, since
\[ (\delta' a)\cap \alpha = (-1)^{\deg (\delta' a)} a\cap \partial \alpha = -a\cap \partial \alpha \]
for an IP-cobordism class $a$ of degree $0$.
The image of $\mu_1$ under $i_* \partial$ is $[i:(S^0,\varnothing)\hookrightarrow (S^0, y_-)]$, while
the image of $[\id_{y_+}]\in (\Omega_\ip)_0 (y_+)$ under the excision isomorphism is
$[(y_+, \varnothing)\hookrightarrow (S^0,y_-)]$. Now
\[ [(y_+, \varnothing)\hookrightarrow (S^0,y_-)] = [i] \in (\Omega_\ip)_0 (S^0,y_-)_\rat \]
via the bordism
$W = I\sqcup I$ (disjoint union of two intervals) and $F:W\to S^0$ defined by
mapping the first copy of $I$ by the constant map to $y_+$ and mapping the second
copy of $I$ to $y_-$. Then the disjoint union $\{ y_+ \} \sqcup S^0$ is contained in $\partial W$,
$F$ restricted to $\{ y_+ \} \sqcup S^0$ agrees with the disjoint union of the 
inclusion $y_+ \hookrightarrow S^0$ and the identity map $S^0 \to S^0$, while
$F$ maps $\partial W - (\{ y_+ \} \sqcup S^0)$ to $y_-$. Hence $(W,F)$ is a valid bordism.
Consequently,
\begin{eqnarray*}
 e_\ip \cap \mu_1 & = & \delta' i^* \operatorname{exc}^{-1} (1) \cap \mu_1 \\
 & = & -\iota \iota_+ (1\cap \operatorname{exc}^{-1}  i_* \partial (\mu_1)) \\
 & = & -\iota \iota_+ (\operatorname{exc}^{-1}  i_* \partial (\mu_1)) \\
 & = & -\iota \iota_+ [\id_{y_+}] \\
& = & -[\{ y_+ \} \hookrightarrow D^1]
\end{eqnarray*}
and so, with $s_2 = n-m,$
\begin{eqnarray*} 
 e^{n-m}_\ip \cap \mu 
& = & (-1)^{s_1}   (-1)^{s_2}[\{ y_+ \} \hookrightarrow D^1] \times \cdots \times  [\{ y_+ \} \hookrightarrow D^1] \\
& = & (-1)^s [\pt \hookrightarrow D^{n-m}].
\end{eqnarray*}
For any pair $(W,V)$ of IP-spaces, $n=\dim W,$ let $\si: (\Omega_\ip)_n (W,V)_\rat \to 
\syml_n (W,V)_\rat$ be the composition
\[ (\Omega_\ip)_n(W,V)
\xleftarrow[\cong]{A}
(\bQ_\ip)_n(W,V)
\xrightarrow{\Sig}
(\bQ^\Z_{\mathrm{rel},\geq 0})_n(W,V)
\xleftarrow{\cong}
\syml_n(W,V) \]
(which is just (\ref{e13}) in the absolute case), tensored with $\id_\rat$.
By Section \ref{sfm}, we also have a corresponding multiplicative transformation of cohomology theories
$\si: \Omega^*_\ip (W,V)\to (\syml)^* (W,V)$.
Thus the arguments used above for the transformation $S$ can be applied to $\si$ to show that
\begin{equation} \label{equ.siuipisu}
\si (u_\ip) =u_\BL.
\end{equation}
If $\phi:E\to F$ is any morphism of ring spectra, then the diagram
\[ \xymatrix{
E^p (Z)\otimes E_q (Z) \ar[r]^{\phi_* \otimes \phi_*} \ar[d]_{\cap_E} &
  F^p (Z)\otimes F_q (Z) \ar[d]^{\cap_F} \\
E_{q-p} (Z) \ar[r]_{\phi_*} & F_{q-p} (Z)
} \]
commutes (see e.g. \cite[Prop. III.9.1 (v)]{adams}.)
Therefore, in view of (\ref{equ.sulisu}) and (\ref{equ.siuipisu}),
the following diagram commutes:
\[ \xymatrix{
(\Omega_\ip)_n (X)_\rat \ar[r]^{\si} \ar[d] & \syml_n (X)_\rat \ar[r]^{S_{X}}_{\cong} \ar[d] &
   \bigoplus_q H_{n-4q} (X;\rat) \ar[d] \\
(\Omega_\ip)_n (X,X-Y)_\rat \ar[r]^{\si}  & \syml_n (X,X-Y)_\rat \ar[r]^>>>>{S_{(X,X-Y)}}_>>>>{\cong} &
   \bigoplus_q H_{n-4q} (X,X-Y;\rat)  \\
(\Omega_\ip)_n (D\nu,S\nu)_\rat \ar[r]^{\si} \ar[d]_{u_\ip \cap -} \ar[u]^{\cong}_{\operatorname{exc}} & 
     \syml_n (D\nu, S\nu)_\rat 
    \ar[r]^>>>>>{S_{(D\nu,S\nu)}}_>>>>>{\cong} \ar[d]_{u_{\mathbb{L}} \cap -} \ar[u]^{\cong}_{\operatorname{exc}} &
   \bigoplus_q H_{n-4q} (D\nu,S\nu;\rat) \ar[d]_{u\cap -} \ar[u]^{\cong}_{\operatorname{exc}} \\
(\Omega_\ip)_m (Y\times D^{n-m})_\rat \ar[r]^{\si} \ar[d]^{\pi_*}_{\cong} & 
    \syml_m (Y\times D^{n-m})_\rat 
  \ar[r]^>>>>>{S_{Y\times D^{n-m}}}_>>>>>{\cong} \ar[d]^{\pi_*}_{\cong} &
   \bigoplus_q H_{m-4q} (Y\times D^{n-m};\rat) \ar[d]^{\pi_*}_{\cong} \\
(\Omega_\ip)_m (Y)_\rat \ar[r]^{\si} & \syml_m (Y)_\rat \ar[r]^{S_{Y}}_{\cong}  &
   \bigoplus_q H_{m-4q} (Y;\rat) \\
} \]
The left column, multiplied by $(-1)^s,$
defines a map $j^!_\ip: (\Omega_\ip)_n (X)_\rat \to (\Omega_\ip)_m (Y)_\rat$.
The image of $[X]_\ip$ in $(\Omega_\ip)_n (X,X-Y)_\rat$ equals the image of 
$[\id_{(D\nu, S\nu)}] \in (\Omega_\ip)_n (D\nu, S\nu)_\rat$ under the excision isomorphism;
the required bordism is given by the $\partial$-IP-space $W$ obtained
from gluing the cylinder $X\times I$ to the cylinder $D\nu \times I$ along the canonical
inclusion $D\nu \times \{ 1 \} \hookrightarrow X\times \{ 0 \}$. 
The map $F:W\to X$ is defined by $F(x,t)=x$ for $(x,t)\in X\times I$ and
$(x,t)\in D\nu \times I$. Note that 
$F$ maps $\partial W - (X\times \{ 1 \} \sqcup D\nu \times \{ 0 \})$ to $X-Y$, 
whence $(W,F)$ is indeed a viable bordism. 
Using the cross product on IP-bordism
\[ (\Omega_\ip)_{n-m} (D^{n-m}, S^{n-m-1})_\rat \otimes (\Omega_\ip)_m (Y)_\rat
  \stackrel{\times}{\longrightarrow} (\Omega_\ip)_n (D\nu, S\nu)_\rat, \]
we may express the element $[\id_{(D\nu, S\nu)}]$ as
\[  [\id_{(D\nu, S\nu)}] = \mu \times [Y]_\ip. \]
Thus, using (\ref{enmcapmu}), we find that
\begin{eqnarray*}
\pi_* (u_\ip \cap [\id_{(D\nu, S\nu)}])
& = & \pi_* ((e^{n-m}_\ip \times 1)\cap (\mu \times [Y]_\ip)) \\
& = & \pi_* ((-1)^{\deg (1)\deg (\mu)} (e^{n-m}_\ip \cap \mu)\times (1 \cap [Y]_\ip)) \\
& = & \pi_* ((e^{n-m}_\ip \cap \mu)\times (1 \cap [Y]_\ip)) \\
& = & (-1)^s \pi_* ([\pt \hookrightarrow D^{n-m}] \times [Y]_\ip) \\
& = & (-1)^s [Y]_\ip.
\end{eqnarray*}
This proves that
\[ j^!_\ip [X]_\ip = [Y]_\ip. \]
By the commutativity of the above diagram,
\[ j^! S_X [X]_\BL = j^! S_X \si [X]_\ip = S_Y \si j^!_\ip [X]_\ip = S_Y \si [Y]_\ip = S_Y [Y]_\BL, \]
which proves property (1).\\

It remains to establish property (2).
The map $f:X\to \pt$ from $X$ to a point induces a homomorphism
\[ f_*: \bigoplus_q H_{n-4q} (X;\rat) \longrightarrow
    \bigoplus_q H_{n-4q} (\pt;\rat) \]
such that $(S_X [X]_{\mathbb{L}})_0 =f_* S_X [X]_{\mathbb{L}}.$
If the dimension $n$ is not divisible by $4$, then $\bigoplus_q H_{n-4q} (\pt;\rat)=0$
and thus $\epsilon_* f_* S_X [X]_{\mathbb{L}} =0=\sigma (X)$, that is, (2) holds.
Assume that $n$ is divisible by $4$, so that
$\bigoplus_q H_{n-4q} (\pt;\rat)=H_0 (\pt;\rat)$.
Using the commutative diagram
\[ \xymatrix{
\syml_n (X)_\rat \ar[r]^>>>>{S_{X}}_>>>>{\cong} \ar[d]_{f_*} &
   \bigoplus_q H_{n-4q} (X;\rat) \ar[d]^{f_*} \\
\syml_n (\pt)_\rat \ar[r]^>>>>{S_{\pt}}_>>>>{\cong}  &
   \bigoplus_q H_{n-4q} (\pt;\rat) 
} \]
we can write 
\begin{equation} \label{equ.l1}
 f_* S_X [X]_{\mathbb{L}} =
  S_\pt f_* [X]_{\mathbb{L}}. 
\end{equation}
Let $\{ 1 \} = \pi_1 (\pt)$ denote the trivial fundamental group of the point.
The associated assembly map
\[ \syml_n (\pt) = \syml_n (B\{ 1 \}) \stackrel{\alpha_{\{ 1 \}}}{\longrightarrow}
L^n (\intg [\{ 1 \}]) \]
is an isomorphism.
Recall (\cite[Prop. 7.2]{ranickialt1}) that when $n$ is divisible by $4$, there is an isomorphism
$\sigma: L^n (\intg [\{ 1 \}])\cong \intg$ given by the signature $\sigma$
of a symmetric algebraic Poincar\'e complex.
The diagram
\[ \xymatrix{
\syml_n (\pt)_\rat \ar[r]^>>>>{S_{\pt}}_>>>>{\cong} \ar@{=}[d] &
   \bigoplus_q H_{n-4q} (\pt;\rat) \ar@{=}[d] \\
\syml_n (B\{ 1 \})_\rat \ar[d]_{\alpha_{\{ 1 \}}}^{\cong} & 
    H_{0} (\pt;\rat) \ar[d]^{\epsilon_*}_{\cong} \\ 
L^n (\intg [\{ 1 \}])_\rat \ar[r]^{\sigma}_{\cong} & \rat
} \]
commutes, as the calculation
\begin{eqnarray*}
\epsilon_* S_{\pt} [\pt]_{\mathbb{L}}
& = & \epsilon_* (L^* (\pt)\cap [\pt]_\rat ) =
   \epsilon_* (1\cap [\pt]_\rat) = \epsilon_* [\pt]_\rat = 1 \\
& = & \sigma (\sigma^* (\pt)) = \sigma (\alpha_{\{ 1 \}} [\pt]_{\mathbb{L}}),
\end{eqnarray*}
using e.g. \cite[Remark 16.17(i)]{ranicki}, shows.
(The formula $S_M [M]_{\mathbb{L}} = L^* (M)\cap [M]_\rat$ holds for
any closed smooth oriented $n$-manifold, where $L^* (M)\in H^* (M;\rat)$ is the
Hirzebruch $L$-class and $[M]_\rat \in H_n (M;\rat)$ the rational fundamental class.
Note that $\sigma (\sigma^* (\pt))=1,$ since the composition
\[ \Omega^{\SO}_0 (\pt) \stackrel{\sigma^*}{\longrightarrow} 
  \syml_0 (\pt) = L^0 (\intg)\stackrel{\sigma}{\longrightarrow} \intg \]
is the signature homomorphism
$\sigma: \Omega^{\SO}_0 (\pt) \to \intg,$ which sends the generator $[\id:\pt \to \pt]$
to the signature of a point, which equals $1$.)
Using this diagram, we obtain
\begin{equation} \label{equ.l2}
\epsilon_* S_\pt f_* [X]_{\mathbb{L}} 
  = \sigma \alpha_{\{ 1 \}}  f_* [X]_{\mathbb{L}}. 
\end{equation}
Let $G=\pi_1 (X)$ be the fundamental group of $X$ and $r:X\to BG$ a classifying map for the
universal cover of $X$. 
The commutative diagram of spaces
\[ \xymatrix{
X \ar[r]^f \ar[d]_r & \pt \ar@{=}[d] \\
BG \ar[r] & B\{ 1 \}
} \]
induces a commutative diagram
\[ \xymatrix{
\syml_n (X)_\rat \ar[r]^{f_*} \ar[d]_{r_*} &
   \syml_n (\pt)_\rat \ar@{=}[d] \\
\syml_n (BG)_\rat \ar[r] & \syml_n (B\{ 1 \})_\rat.  
} \]
Since the assembly map, obtained e.g. by applying \cite[Thm. 1.1]{wwa} to
the homotopy functor $Z\mapsto \mathbb{L}^\bullet (\intg \pi_1 (Z))$, is a natural transformation,
the diagram
\[ \xymatrix{
\syml_n (BG)_\rat \ar[r] \ar[d]_{\alpha_G} & \syml_n (B\{ 1 \})_\rat 
     \ar[d]^{\alpha_{\{ 1 \}}} \\ 
L^n (\intg [G])_\rat \ar[r]^{f_*} & L^n (\intg [\{ 1 \}])_\rat
} \]
commutes. Using these two squares, we get
\begin{equation} \label{equ.l3}
 f_* \alpha_G r_* [X]_{\mathbb{L}} = \alpha_{\{ 1 \}} f_* [X]_{\mathbb{L}}. 
\end{equation}
The ordinary signature information is contained in the symmetric signature
by 
\begin{equation} \label{equ.l4}
\sigma f_* \sigma^*_\ip (X)=\sigma (X),
\end{equation}
as follows from Remark \ref{rem.ipsymsiganalogwittsymsig}.
Putting equations (\ref{equ.l1}), (\ref{equ.l2}), (\ref{equ.l3}) and 
(\ref{equ.l4}) together, we compute
\begin{eqnarray*}
\epsilon_* (S_X [X]_{\mathbb{L}})_0
& = & \epsilon_* f_* S_X [X]_{\mathbb{L}} =
  \epsilon_* S_\pt f_* [X]_{\mathbb{L}} =
    \sigma \alpha_{\{ 1 \}} f_* [X]_{\mathbb{L}} \\
& = & \sigma f_* \alpha_G r_* [X]_{\mathbb{L}} =
    \sigma f_* \sigma^*_\ip (X)=\sigma (X),
\end{eqnarray*}
as was to be shown.
\end{proof}
Let $G=\pi_1 (X)$ be the fundamental group and $r:X\to BG$ a classifying map for the
universal cover of $X$. The map $r$ induces a homomorphism
\[ H_* (X;\rat) \longrightarrow H_* (BG;\rat) \]
on homology.
The \emph{higher signatures of $X$} are the rational numbers
\[ \langle a, r_* L(X) \rangle,~ a\in H^* (BG;\rat). \]
\begin{thm} \label{thm.novikov}
Let $X$ be an $n$-dimensional compact IP-space whose fundamental group
$G=\pi_1 (X)$ satisfies the strong Novikov conjecture.
Then the higher signatures of $X$ are stratified homotopy invariants.
\end{thm}
\begin{proof}
Let $X$ and $X'$ be $n$-dimensional compact IP-spaces with fundamental group $G$ and 
$f:X'\to X$ an orientation preserving stratified homotopy equivalence. If $r:X\to BG$ is a classifying map for the
universal cover of $X$, then $r' = r\circ f: X'\to BG$ is a classifying map for the
universal cover of $X'$. We must prove that
\[ r'_* L(X') = r_* L(X) \in H_* (BG;\rat). \]
By Theorem \ref{t3}(ii), the assembly map
\[ \syml_n (X)\longrightarrow L^n (\intg [G]) \]
maps $[X]_{\mathbb{L}}$ to $\sigma_\ip^*(X)$. 
By naturality, the assembly map for the space $X$ factors through $BG$
(see also \cite[p. 216]{davis}):
\[ \xymatrix{
\syml_n (X) \ar[r] \ar[d]_{r_*} & L^n (\intg [G]) \\
\syml_n (BG) \ar[ru]_{\alpha} &
} \]
(similarly for $X'$). Using this factorization, we may write 
\[ \alpha r_* [X]_{\mathbb{L}} = \sigma_\ip^*(X),~ \alpha r'_* [X']_{\mathbb{L}} = \sigma_\ip^*(X'). \]
By the stratified homotopy invariance (\ref{equ.htpyinvipsymsig}) of the IP symmetric signature,
\[ \sigma_\ip^* (X)= \sigma_\ip^* (r) = \sigma_\ip^* (rf)=\sigma_\ip^* (r')=\sigma_\ip^* (X'). \]
As $\alpha$ is by assumption rationally injective, it follows that
\[ r_* [X]_{\mathbb{L}} = r'_* [X']_{\mathbb{L}} \in \syml_n (BG)\otimes \rat. \]
Using the commutative diagram
\[ \xymatrix{
\syml_n (X) \otimes \rat \ar[r]^{r_*} \ar[d]_{S_X}^{\cong} & \syml_n (BG)\otimes \rat \ar[d]^{S_{BG}}_{\cong} \\
\bigoplus_j H_{n-4j} (X;\rat) \ar[r]^{r_*} & \bigoplus_j H_{n-4j} (BG;\rat)
} \]
(and the analogous diagram for $X'$),
together with Lemma \ref{lem.lfundhitslclass}, we deduce 
\begin{eqnarray*}
r_* L(X) & = & r_* S_X [X]_{\mathbb{L}} =S_{BG} r_* [X]_{\mathbb{L}} \\
 & = & S_{BG} r'_* [X']_{\mathbb{L}}
 = r'_* S_{X'} [X']_{\mathbb{L}} = r'_* L(X').
\end{eqnarray*}
\end{proof}
An analytic version of Theorem \ref{thm.novikov} has been proven by
Albin-Leichtnam-Mazzeo-Piazza in \cite{ALMPnovikov}. 
The scope of their theorem is in fact larger, as it applies even to those
non-Witt spaces, for which a so-called analytic self-dual mezzoperversity exists.
It was shown in \cite{ablmp} that such perversity data corresponds topologically
to the Lagrangian structures of Banagl as introduced in \cite{banagl-mem}.
A comparison of the analytic argument to our argument
shows that the role of our $\syml_n (X)$ is played in the analytic context
by $K_* (X)$. The role of the isomorphisms $S_X$ is played by the Chern character.
The group $L^n (\intg [G])$ corresponds to $K_* (C^*_r G),$ while our assembly map
$\alpha$ corresponds to the assembly map
$K_* (BG)\to K_* (C^*_r G)$ used in the analytic argument.

\section{Multiplicativity and commutativity}
\label{mc}

Recall from Definition \ref{d4} that
the symmetric signature map
\[
\Sig:\bM_\ip \to \bM^\Z_{\mathrm{rel},\geq 0}
\]
is the following composite in the homotopy category of spectra:
\[
\bM_\ip \xleftarrow{\simeq}
\bM_\ipFun
\xrightarrow{\sig}
\bM^\Z_{\mathrm{rel},\geq 0}.
\]
In this section we show that this composite is weakly equivalent to a composite
of ring maps between commutative ring spectra.
Specifically, we show the following. Recall Remarks \ref{r2} and \ref{r7}.

\begin{thm}
\label{t4}
There are
symmetric ring spectra
$\mathbf A$, $\mathbf B$ and $\mathbf C$, a commutative symmetric ring 
spectrum $\mathbf D$,
and a strictly commutative diagram
\begin{equation}
\label{e16}
\xymatrix{
\bM_\ip
&
{\mathbf A}
\ar[l]_-\simeq
\ar[r]^-\simeq
&
\bM^{\mathrm{comm}}_\ip
\\
\bM_\ipFun
\ar[u]_\simeq
\ar[d]_\sig
&
{\mathbf B}
\ar[l]_-\simeq
\ar[r]^-\simeq
\ar[u]_\simeq
\ar[d]
&
\bM^{\mathrm{comm}}_\ipFun
\ar[u]_\simeq
\ar[d]
\\
\bM_{{\mathrm{rel},\geq 0}}^\Z
&
\mathbf C
\ar[l]_-\simeq
\ar[r]^-\simeq
&
\mathbf D
}
\end{equation}
in which the horizontal arrows, the upper vertical arrows and the lower right
vertical arrow are ring maps. 
\end{thm}

\begin{remark}
\label{r9}
$\mathbf D$ is weakly equivalent to 
$(\bM_{{\mathrm{rel},\geq 0}}^\Z)^\mathrm{comm}$
by \cite[Remark 18.3]{LM2}.
\end{remark}

The rest of this section is devoted to the proof of Theorem \ref{t4}.
The top half of the diagram has already been constructed in Remark \ref{r7}.
For the lower half we will use the method of the proof of \cite[Theorem
1.3]{LM2} (it will be straightforward to check that the maps 
$\bM_\ipFun\leftarrow {\mathbf B}\to \bM^{\mathrm{comm}}_\ipFun$ given by the 
proof in this section are the same as those given by Remark \ref{r7}).

\begin{remark}
\label{r8}
In order to apply the proof of \cite[Theorem 1.3]{LM2} without change we would
need to know (by analogy with the paragraph before \cite[Definition 15.5]{LM2})
that the cross product gave a natural quasi-isomorphism
from the functor
\[
(\Aa_\ipFun)^{\times l}
\xrightarrow{\sig^{\times l}}
(\Aa^\Z_{{ \mathrm{rel}}})^{\times l}
\xrightarrow{\otimes}
\Aa^\Z_{{ \mathrm{rel}}}
\]
to the functor 
\[
(\Aa_\ipFun)^{\times l}
\to
\Aa_\ipFun
\xrightarrow{\sig}
\Aa^\Z_{{ \mathrm{rel}}}
\]
(where the unmarked arrow is the product in $\Aa_\ipFun$).
But this is not the case, for the simple reason that the cross product does not
give a map 
\[
IS_*^{Q_{\bar n, \bar n}}(X\times X;\Z)^{\otimes l}
\to
IS_*^{Q_{\bar n, \bar n}}(X^{\times l}\times X^{\times l};\Z)
\]
(cf.\ \cite[Lemma 6.4.1]{Fr}).  Our first task is to provide a
suitable substitute, which will be given in Proposition \ref{p4}.
\end{remark}

\begin{definition}
Let $Y_1,\ldots,Y_k$ be stratified PL $\partial$-pseudomanifolds and give 
$Y_1\times\cdots\times Y_k$ the product stratification.  Define
a perversity $Q_k$ on $Y_1\times\cdots\times Y_k$ by 
\[
Q_k(S_1\times\cdots\times S_k)=
\begin{cases}
0 \text{ if all $S_i$ are regular},
\\
2s-2+\sum \bar{n}(S_i)\text{ otherwise},
\end{cases}
\]
where the $S_i$ are strata and $s$ is the number of $S_i$ that are singular. 
\end{definition}

In particular, $Q_1=\bar n$ and $Q_2=Q_{\bar n, \bar n}$.

\begin{lemma}
\label{cross}
The cross product induces a quasi-isomorphism 
\[
IS_*^{Q_j}(Y_1\times\cdots\times Y_j;\Z)
\otimes
IS_*^{Q_k}(Y_{j+1}\times\cdots\times Y_{j+k};\Z)
\to
IS_*^{Q_{j+k}}(Y_1\times\cdots\times Y_{j+k};\Z).
\]
\end{lemma}

This is immediate from \cite[Theorem 6.4.6 and Remark 6.4.7]{Fr}. 
We need a more general version of this.

\begin{definition}
(i) Let $A$ be a finite totally ordered set.  A {\it partition} $\rho$ of $A$ is
a collection $B_1,\ldots,B_k$ of disjoint subsets of $A$ such that $\cup
B_i=A$ and $a<a'$ whenever
$a\in B_i$, $a'\in B_{i'}$ with $i<i'$.

(ii)
Let $X_1,\ldots,X_l$ be stratified PL $\partial$-pseudomanifolds and let 
\[
\rho=\{B_1,\ldots,B_k\}
\]
be a partition of $\{1,\ldots,l\}$.  Let 
$
Y_i=\prod_{j\in B_i} X_j
$
and give $Y_i$ the product stratification.
Define
\[
IS_*^{\rho}(X_1\times\ldots\times X_l;\Z)
\]
to be
\[
IS_*^{Q_k}(Y_1\times\ldots\times Y_k;\Z).
\]
\end{definition}

\begin{lemma}
\label{l4}
The cross product induces a quasi-isomorphism
\[
IS_*^{\rho}(X_1\times\cdots\times X_l;\Z)
\otimes
IS_*^{\rho'}(X_{l+1}\times\cdots\times X_{l+m};\Z)
\to
IS_*^{\rho\cup\rho'}(X_1\times\cdots\times X_{l+m};\Z).
\]
\end{lemma}

This is immediate from Lemma \ref{cross}.

\begin{definition}
Let $\rho=\{B_1,\ldots,B_j\}$ and $\rho'=\{C_1,\ldots,C_k\}$ be two partitions
of a set $A$.  Then $\rho'$ is a {\it refinement} of $\rho$ if each $C_i$ is
contained in some $B_i$.
\end{definition}

\begin{lemma}
\label{l5}
Let $\rho$ and $\rho'$ be partitions of $\{1,\ldots,l\}$.  If $\rho'$ is a
refinement of $\rho$ then 
\[
IS_*^{\rho}(X_1\times\cdots\times X_l;\Z)
\subset
IS_*^{\rho'}(X_1\times\cdots\times X_l;\Z)
\]
and the inclusion is a quasi-isomorphism.
\end{lemma}

\begin{proof}
First we claim
that the perversity that gives $IS_*^{\rho}$
is $\leq$ the perversity that gives $IS_*^{\rho'}$; the inclusion follows from
this.  It suffices to check the
case where one piece in $\rho$ is divided into two pieces in $\rho'$; this
means that one of the factors $Y_i$ is replaced by two factors $Y'$ and
$Y''$.  If $Y'$ and $Y''$ are not both singular it's easy to verify the
claim.  Otherwise we need to show 
\[
\bar n(Y'\times Y'')
\leq \bar n(Y')+\bar n(Y'') +2,
\]
and this follows by checking the cases where $\codim Y'$ and $\codim Y''$ are
both even, or both odd, or one is even.

To show the
quasi-isomorphism it suffices to show that
\[
IS_*^{\bar n}(X_1\times\cdots\times X_l;\Z)
\hookrightarrow
IS_*^{\rho}(X_1\times\cdots\times X_l;\Z)
\]
is a quasi-isomorphism for every $\rho$.  This in turn follows by induction 
from the following commutative diagram, where we let $\rho=\{B_1,\ldots,B_k\}$
and $\rho_1=\{B_2,\ldots,B_k\}$.
\[
\xymatrix{
IS_*^{\bar n}(\prod_{i\in B_1} X_i;\Z)
\otimes
IS_*^{\bar n}(\prod_{i\notin B_1} X_i;\Z)
\ar[r]^-\times
\ar[d]
&
IS_*^{\bar n}(X_1\times\cdots\times X_l;\Z)
\ar[d]
\\
IS_*^{\bar n}(\prod_{i\in B_1} X_i;\Z)
\otimes
IS_*^{\rho_1}(\prod_{i\notin B_1} X_i;\Z)
\ar[r]^-\times
&
IS_*^{\rho}(X_1\times\cdots\times X_l;\Z)
}
\]
Here the horizontal arrows are quasi-isomorphisms by \cite[Theorem
6.4.6 and Remark 6.4.7]{Fr} and the left vertical arrow is a 
quasi-isomorphism by the inductive
hypothesis, so the right vertical arrow is a quasi-isomorphism as required.
\end{proof}

Next is the analogue of Lemma \ref{l1} for this situation.  Recall
\cite[Definition 15.3(i)]{LM2}.

\begin{lemma}
\label{l6}
{\rm (i)}
Let $l\geq 1$ and let $\rho=\{B_1,\ldots,B_k\}$ be a partition of 
$\{1,\ldots,l\}$.  Let $\hat{\rho}$ be the partition
$\{B_1,\ldots,B_k,B_1,\ldots,B_k\}$ of $\{1,\ldots,l\}\coprod\{1.\ldots,l\}$.
Let $(X_i,\xi_i)$ for $1\leq i\leq l$ be objects of $\Aa_\ipFun$.  Give each 
$X_i$ the stratification of
Proposition \ref{p2}
and give $X_1\times\cdots\times X_l$ and 
$(X_1\times\cdots\times X_l)\times (X_1\times\cdots\times X_l)$
the product stratifications.  
Let $ \iota$ be the inclusion map  
$$IS^{Q_{\bar{n},\bar{n}}}_*((X_1\times\cdots\times X_l)\times 
(X_1\times\cdots\times X_l);\Z)
\subset
IS^{\hat{\rho}}_*((X_1\times\cdots\times X_l)\times (X_1\times\cdots\times
X_l);\Z)$$
given by Lemma \ref{l5}. Then
$(C,D,\beta,\varphi)$ is an object of $\Aa^\Z_{ \mathrm{rel}}$,
where

\quad $C=IS_*^{\rho}(X_1\times\cdots\times X_l;\Z)$,

\quad $D=IS^{\hat{\rho}}_*((X_1\times\cdots\times X_l)\times (X_1\times\cdots\times
X_l);\Z)$, 

\quad $\beta$ is the cross product followed by the inclusion $\iota$, and

\quad $\varphi$ is the image of
$\xi_1\times\cdots\times \xi_l$
%\in IS_*^{\bar{0}}(X_1\times\cdots\times X_l;\Z)$ 
under the map induced by the diagonal
\begin{multline*}
IS_*^{\bar{0}}(X_1\times\cdots\times X_l;\Z)
\to
IS^{Q_{\bar{n},\bar{n}}}_*((X_1\times\cdots\times X_l)\times 
(X_1\times\cdots\times X_l);\Z)
\end{multline*}
followed by  $\iota$.

{\rm (ii)} For $1\leq i\leq l$, let $f_i:(X_i,\xi_i)\to (X_i',\xi_i')$ be a 
morphism in $\Aa_\ipFun$. Let $(C,D,\beta,\varphi)$ and 
$(C',D',\beta',\varphi')$ be the objects of $\Aa^\Z_{ \mathrm{rel}}$ 
corresponding to the $l$-tuples $\{(X_i,\xi_i)\}$ and $\{(X'_i,\xi_i')\}$.  
Then the $f_i$
induce a morphism $(C,D,\beta,\varphi)\to (C',D',\beta',\varphi')$.
\end{lemma}

\begin{proof}
Part (i) follows from Lemma \ref{l5} and the fact (shown in the proof of 
Lemma \ref{l1}(i)) that $IS_*^{\bar{n}}(X_1\times\cdots\times X_l;\Z)$ is 
homotopy finite.
Part (ii) follows from the proof of Lemma \ref{l1}(ii).
\end{proof}

Now we can give the statement promised in Remark \ref{r8}.
Let 
\[
\sig_\rho:(\Aa_\ipFun)^{\times l}\to \Aa^\Z_{{ \mathrm{rel}}}
\]
be the functor given by Lemma \ref{l6}, and recall \cite[Definition
12.14]{LM2}.

\begin{prop}
\label{p4}
Let $\{B_1,\ldots,B_k\}$ be a partition of $\{1,\ldots,l\}$, let
$\rho_i$ be a partition of $B_i$ for $1\leq i\leq k$, let $\rho$ denote
$\rho_1\cup\cdots\cup\rho_k$, and let $\rho'$ be a
refinement of $\rho$.
There is a natural quasi-isomorphism from
\[
(\Aa_\ipFun)^{\times l}
\xrightarrow{\prod \sig_{\rho_i}}
(\Aa^\Z_{{\mathrm{rel}}})^{\times k}
\xrightarrow{\otimes}
\Aa^\Z_{{\mathrm{rel}}}
\]
to 
\[
(\Aa_\ipFun)^{\times l}
\xrightarrow{\sig_{\rho'}}
\Aa^\Z_{{ \mathrm{rel}}}
\]
given by the maps
\begin{multline*}
IS_*^{\rho_1}(\prod_{i\in B_1} X_i;\Z)
\otimes\cdots\otimes
IS_*^{\rho_k}(\prod_{i\in B_k} X_i;\Z)
\\
\xrightarrow{\times}
IS_*^{\rho}(X_1\times\cdots \times X_l;\Z)
\to
IS_*^{\rho'}(X_1\times\cdots \times X_l;\Z)
\end{multline*}
and
\begin{multline*}
IS^{\hat{\rho_1}}_*((\prod_{i\in B_1} X_i)\times (\prod_{i\in B_1} X_i);\Z)
\otimes\cdots\otimes
IS^{\hat{\rho_k}}_*((\prod_{i\in B_k} X_i)\times (\prod_{i\in B_k} X_i);\Z)
\\
\xrightarrow{\times}
IS^{\widehat{\rho}}_*(\prod_{j=1}^k((\prod_{i\in B_f} X_i)\times (\prod_{i\in
B_f} X_i));\Z)
\cong
IS^{\widehat{\rho}}_*((\prod_{i=1}^l X_i)\times (\prod_{i=1}^l X_i);\Z)
\\
\to
IS^{\widehat{\rho'}}_*((\prod_{i=1}^l X_i)\times (\prod_{i=1}^l X_i);\Z).
\end{multline*}
\end{prop}

This is immediate from Lemmas \ref{l4} and \ref{l5}.
Next we need the analogue of \cite[Definition
15.4]{LM2}.  Recall \cite[Definition
15.3(ii)]{LM2}.  

\begin{definition}
\label{m67}
Let $j\geq 0$ and let $r:\{1,\ldots,j\}\to \{u,v\}$.  Let $\Aa_i$ denote $\Ag$
if $r(i)=u$ and $\Ar$ if $r(i)=v$.

(i) Let $1\leq m\leq j$. A surjection 
\[
h:\{1,\ldots,j\}\to\{1,\ldots,m\}
\]
is {\it adapted} to $r$ if 
$r$ is constant on 
each set $h^{-1}(i)$ and  
$h$ is  injective on $r^{-1}(v)$.

(ii) Given a surjection 
\[
h:\{1,\ldots,j\}\to\{1,\ldots,m\}
\]
which is adapted to $r$, 
and a partition $\rho_i$ of $h^{-1}(i)$ for $1\leq i\leq m$,
define
\[
(h,\rho_1,\ldots,\rho_m)_\blacklozenge: 
\Aa_1\times\cdots\times\Aa_j
\to
(\Ar)^{\times m}
\]
by
\[
(h,\rho_1,\ldots,\rho_m)_\blacklozenge(x_1,\ldots,x_j)
=(i^{\epsilon}y_1,\ldots,y_m),
\]
where
$i^\epsilon$ is the sign that arises from putting the objects
$x_1,\ldots,x_j$ into the order $x_{\theta(h)^{-1}(1)},\ldots,
x_{\theta(h)^{-1}(j)}$
and 
\[
y_i=
\begin{cases}
\sig_{\rho_i}(\{x_p\}_{p\in h^{-1}(i)}) & \text{if $h^{-1}(i)\subset 
r^{-1}(u)$},
\\
x_{h^{-1}(i)} & \text{if $h^{-1}(i)\in r^{-1}(v)$}.
\end{cases}
\]

(iii)
A {\it datum} of type $r$ is a tuple
\[
(h,\rho_1,\ldots,\rho_m,\eta),
\]
where $h$ is a surjection which is 
adapted to $r$, $\rho_i$ is a partition of $h^{-1}(i)$, and $\eta$ is an 
element of $\Sigma_j$ with the property
that $h\circ\eta=h$.

(iv) Given a datum 
\[
\bfd=(h,\rho_1,\ldots,\rho_m,\eta),
\]
of type $r$,
define 
\[
\bfd_\bs :
\Aa_1\times\cdots\times\Aa_j
\to
\Ar
\]
to be the composite
\begin{multline*}
\Aa_1\times\cdots\times\Aa_j
\xrightarrow{\eta}
\Aa_{\eta^{-1}(1)}\times\cdots\times\Aa_{\eta^{-1}(j)}
=
\Aa_1\times\cdots\times\Aa_j
\\
\xrightarrow{(h,\rho_1,\ldots,\rho_m)_\blacklozenge}
(\Ar)^{\times m}
\xrightarrow{\otimes}
\Ar,
\end{multline*}
where $\eta$ permutes the factors with the usual sign.
\end{definition}

Finally, we have the analogue of \cite[Definition 15.5]{LM2}.

\begin{definition}
\label{m63}
For data of type $r$, define
\[
(h,\rho_1,\ldots,\rho_m,\eta)
\leq
(h',\rho'_1,\ldots,\rho'_{m'},\eta')
\]
if for each $i\in\{1,\ldots,m\}$ there is a $p\in\{1,\ldots,m'\}$ such that 
$\eta^{-1}(h^{-1}(i))$ is 
contained in 
${\eta'}^{-1}({h'}^{-1}(p))$ and $\eta'\eta^{-1}$ takes each piece of the
partition $\rho_i$ to a union of pieces of the partition $\rho'_p$.
\end{definition}

With these changes, the proof of \cite[Theorem 1.3]{LM2} goes through to 
construct the lower half of Diagram \eqref{e16}.  This completes the proof of
Theorem \ref{t4}.
% For the associativity property of the monad,  need that the natural and 
% product filtrations on the product of $\partial$ pseudomanifolds are the same

\section{Multiplicativity of the $L$-theory fundamental class}
\label{sfm}

In this section we prove

\begin{thm}
\label{t5}
\label{fundmult}
Let $X$ and $Y$ be compact oriented IP spaces. Then
\[
[X\times Y]_{\mathbb{L}}
=
[X]_{\mathbb{L}}
\times
[Y]_{\mathbb{L}}.
\]
\end{thm}

\begin{remark}
We will use the results of Section \ref{mc}, but one could give a simpler
proof of Theorem \ref{t5} using only part of the machinery of Section
\ref{mc}.
\end{remark}

The first step in the proof of Theorem \ref{t5} is to observe that we can 
replace the 
spectra $\bQ$ in Definition \ref{def.xipxl} by the equivalent symmetric 
spectra $\bM$.  Each of the symmetric spectra $\bM_{\ip, Z}$, $Z_+\wedge
\bM_\ip$, $Z_+\wedge \bM^\Z_{\mathrm{rel},\geq 0}$ and $Z_+\wedge \bM^\Z_{\geq 0}$ is semistable 
(\cite[Definition 5.6.1]{hss}) by \cite[Corollary 17.9(i)]{LM} and 
\cite[Examples 4.2 and 4.7]{schwede}, and hence their ``true'' (i.e., derived)
homotopy groups agree with their homotopy groups by \cite[Example
5.5]{schwede}.
Thus 
for a compact oriented IP space $Z$ of dimension $l$ the class
$[Z]_{\mathbb{L}}$ is 
the image of $[Z]_\ip$ under the composite
\begin{equation}
\label{e17}
(\Omega_\ip)_l(Z)
\xrightarrow{\cong}
\pi_l \bM_{\ip, Z}
\xleftarrow[\cong]{\alpha}
\pi_l(Z_+\wedge \bM_\ip)
\xrightarrow{\Sig}
\pi_l(Z_+\wedge \bM^\Z_{\mathrm{rel},\geq 0})
\xleftarrow{\cong}
\pi_l(Z_+\wedge \bM^\Z_{\geq 0})
\end{equation}
(where the first map is given by Remark \ref{aa41} and \cite[Proposition
17.7]{LM}).

Next we observe that the functors in \eqref{e17} have product operations.  For
the first functor (and for any spaces $X$ and $Y$), Cartesian product induces 
a map
\[
(\Omega_\ip)_m(X)
\otimes 
(\Omega_\ip)_n(Y)
\to
(\Omega_\ip)_{m+n}(X\times Y)
\]
by \cite[Lemma 2.11.7]{Fr}.
For the second functor,
Cartesian product induces 
\[
\Aa_{\ip,X}
\times
\Aa_{\ip,Y}
\to
\Aa_{\ip,X\times Y}
\]
and this induces a map
\[
\bM_{\ip,X}
\wedge
\bM_{\ip,Y}
\to
\bM_{\ip,X\times Y}
\]
which gives the desired product. 
The third, 
fourth and fifth functors in
\eqref{e17} have products because $\bM_\ip$, $\bM^\Z_{\mathrm{rel},\geq 0}$ and $\bM^\Z_{\geq 0}$
are ring spectra.

It therefore suffices to show that the maps in the composite \eqref{e17} 
preserve products.  For the second map this follows from Proposition 
\ref{p7}, for the third map from Theorem \ref{t4}, and for the fourth map from
\cite[Remark 13.2]{LM2}.  We will denote the first map by $\chi$, so it remains
to show

\begin{lemma}
\label{L2}
The map
\[
\chi:
(\Omega_\ip)_*(Z)
\xrightarrow{\cong}
\pi_* \bM_{\ip, Z}
\]
preserves products.
\end{lemma}

The rest of this section gives the proof of this lemma.

Recall that for a spectrum $\bQ$ or a symmetric spectrum $\mathbf M$ we write
$Q_k$ and 
$M_k$ for the $k$-th space.
The map $\chi$ can be written as the composite
\begin{equation}
\label{e18}
(\Omega_\ip)_l(Z) 
\xrightarrow{\cong}
\pi_l (Q_{\ip,Z})_0
\to
\pi_{l+k} (Q_{\ip,Z})_{k}
\to
\pi_{l+k}(M_{\ip,Z})_k
\to
\pi_l \bM_{\ip,Z}
\end{equation}
for $k\geq 1$, where the first arrow is given by Remark \ref{aa41}, the second
is the suspension map (\cite[page 44]{LM} and Appendix \ref{aa1}), the
third is described in \cite[top of page 53]{LM}, and the fourth is given by the
definition of the homotopy groups of a spectrum.

If $f:W\to Z$ is a map from an $l$-dimensional compact oriented IP space to 
$Z$, and if 
$p\leq l+2$,
let us write $f^{[l+1]}$ (resp.,
$f^{[p,l+2-p]}$) for the $\Delta^{l+1}$-ad (resp.,
$\Delta^p\times\Delta^{l+2-p}$-ad) 
which 
takes the top cell
with its canonical orientation to 
$f$ and 
all other cells to $\emptyset\to Z$.  Then $f^{[l+1]}$ is a simplex in
$(M_{\ip,Z})_1$ with all faces at the basepoint, so it
determines an element of $\pi_{l+1}(M_{\ip,Z})_1$ which we will denote by
$\underline{f}^{[l+1]}$.  Similarly, $f^{[p,l+2-p]}$ determines an element of 
$\pi_{l+2}(M_{\ip,Z})_2$ which we will denote by $\underline{f}^{[p,l+2-p]}$.
From 
\cite[Sections 15 and 17]{LM} (but using the signs in Appendix \ref{aa1})
we see that
\begin{equation}
\label{ee1}
\chi([f]) \text{\ is represented by\ } \underline{f}^{[l+1]}\in
\pi_{l+1}(M_{\ip,Z})_1 
\end{equation}
and 
\begin{equation}
\label{ee2}
\chi([f]) \text{\ is represented by\ } -\underline{f}^{[0,l+2]}\in
\pi_{l+2}(M_{\ip,Z})_2. 
\end{equation}

Now let $g:U\to X$ and $h:V\to Y$ be maps from compact oriented IP spaces of
dimensions $m$ and $n$ respectively; we need to show that
\begin{equation}
\label{ee3}
\chi([g])\chi([h])=\chi([g\times h]).
\end{equation}
By \eqref{ee1} and the proof of \cite[Theorem I.4.54]{sbook}, 
$\chi([g])\chi([h])$ is
represented by the composite
\begin{multline*}
S^1\wedge S^1 \wedge S^m\wedge S^n
\to
S^1\wedge S^m \wedge S^1\wedge S^n
\xrightarrow{\underline{g}^{[m+1]}\wedge \underline{h}^{[n+1]}}
(M_{\ip,X})_1
\wedge
(M_{\ip,Y})_1
\\
\to
(M_{\ip,X\times Y})_2
\end{multline*}
(cf.\ \cite[I.4.55]{sbook}).
% note that the first map in this composite has degree $(-1)^m$.  
By \cite[Section
18]{LM} this composite is equal to $-\underline{(g\times h)}^{[m+1,n+1]}$ (for
the sign, note that the
first map has degree $(-1)^m$ and the last map includes a sign of 
$(-1)^{m+1}$ by \cite[Remark 18.3]{LM})
so by \eqref{ee2} the proof of\eqref{ee3}  reduces 
to showing
that $\underline{(g\times h)}^{[m+1,n+1]}=\underline{(g\times h)}^{[0,m+n+2]}$,
and for this in turn it suffices to show for $0\leq l\leq m$ that
\begin{equation}
\label{e19}
\underline{(g\times h)}^{[m+1-l,n+1+l]}
=
\underline{(g\times h)}^{[m-l,n+2+l]}.
\end{equation}
To prove \eqref{e19},
let $F$ be the $(\Delta^{m+1-l}\times \Delta^{n+2+l})$-ad which takes
\begin{itemize}
\item
the top cell (with its canonical orientation) to the composite
\[
U\times V\times I\to U\times V\xrightarrow{g\times h} X\times Y
\]
(where the first map is the projection), 
\item
the cell
$\partial_0\Delta^{m+1-l}\times \Delta^{n+2+l}$ (with its canonical
orientation) to $(-1)^{m+n}$ times 
\[
U\times V\times \{1\}\to U\times V\xrightarrow{g\times h} X\times Y,
\]
\item
the cell $\Delta^{m+1-l}\times \partial_0\Delta^{n+2+l}$ (with its
canonical orientation) to $(-1)^{n-l}$ times
\[
U\times V\times \{0\}\to U\times V\xrightarrow{g\times h} X\times Y,
\]
\item
and 
all other cells to the map $\emptyset\to X\times Y$.
\end{itemize}
Then $F$ determines a map
\[
\Phi:\Delta^{m+1-l}\times \Delta^{n+1+l+1}
\to 
(M_{\ip,X\times Y})_2,
\]
and the restriction of $\Phi$ to the boundary of $\Delta^{m+1-l}\times
\Delta^{n+2+l}$ (which is nullhomotopic) is easily seen (using the signs in
Appendix F below) to be $(-1)^{m+n}\underline{(g\times h)}^{[m-l,n+2+l]}
+(-1)^{m+1+n}\underline{(g\times h)}^{[m+1-l,n+1+l]}$; this proves \eqref{e19}
and completes the proof of Lemma \ref{L2}
\qed

\section{Proof of Theorem \ref{t2}}
\label{pt2}

By \cite[Observation 1.3]{wwa} it suffices to show that $\Phi$ is strongly
excisive.  The fact that $(\Omega_\ip)_*$ is a strongly additive  homology 
theory implies that, for any collection of spaces $\{X_\alpha\}$, the 
canonical map 
\begin{multline*}
  \pi_*(\bigvee \Phi (X_\alpha ))
\cong \bigoplus \pi_*(\Phi (X_\alpha))
\cong \bigoplus {\Omega_\ip}_*(X_\alpha) \lra 
{\Omega_{\ip}}_*(\bigvee X_\alpha)
\\
\cong \pi_*(\Phi(\bigvee X_\alpha )) 
\end{multline*}
is an equivalence. Hence, 
$\Phi$ preserves arbitrary coproducts and it suffices to show that $\Phi$
preserves homotopy cocartesian squares.  
First we observe that $\Phi$ takes monomorphisms to cofibrations in the level
model structure given by \cite[Theorem 6.5]{mmss} (because a monomorphism $Z\to
Z'$ gives a map from the $k$-th space of $\bQ_{\ip,Z}$ to the $k$-th space of
$\bQ_{\ip,Z'}$ which is the inclusion of a sub-CW-complex).
For a based space $W$, let
\[
\bar{\Phi}(W)=\Phi(W)/\Phi(*).
\]
Then the natural map
\[
\Phi(W)\to \bar{\Phi}(W_+)
\]
(where $+$ denotes a disjoint basepoint)
is a weak equivalence because 
$(\Omega_\ip)_*$ is a homology theory and thus
$$ \pi_*\Phi(W)\cong {\Omega_\ip}_*(W)\cong
{\Omega_\ip}_*(W_+)/{\Omega_\ip}_*(*)\cong \pi_* \bar{\Phi}(W_+)
.$$  
It
therefore suffices to show that $\bar{\Phi}$ takes homotopy cocartesian squares
of based spaces to homotopy cocartesian squares of spectra.  

As a first step we give a relationship between $\Sigma\bar{\Phi}(W)$ and
$\bar{\Phi}(\Sigma W)$.  Let $CW$ be the cone $I\wedge W$, where $1$ is the
basepoint of $I$, and 
let $S(W)$
denote the pushout of the diagram
\[
\bar{\Phi}(CW) \leftarrow \bar{\Phi}(W) \rightarrow \bar{\Phi}(CW).
\]
Since $\bar{\Phi}(CZ)$ is contractible, and since $\bar{\Phi}$ takes
monomorphisms to cofibrations in the level model structure, $S(W)$ is weakly 
equivalent to 
$\Sigma\bar{\Phi}(W)$.
Since $\Sigma W$ is the pushout of 
\[
CW \leftarrow W \rightarrow CW,
\]
there is an evident map
\[
\mathfrak{S}:S(W)\to \bar{\Phi}(\Sigma W).
\]

\begin{lemma}
\label{l2}
$\mathfrak{S}$ is a weak equivalence for all $W$.
\end{lemma}

We defer the proof for a moment. Let 
\[
B \xleftarrow{i} A \xrightarrow{j} C
\]
be a diagram of based spaces. The homotopy pushout of this diagram is the 
pushout of the diagram
\[
Mi \hookleftarrow A \hookrightarrow Mj,
\]
where $Mi$ and $Mj$ are the mapping cylinders; we will denote this pushout by
$D$.
The homotopy pushout of 
\[
\bar{\Phi}(B) \xleftarrow{\bar{\Phi}(i)} \bar{\Phi}(A)
\xrightarrow{\bar{\Phi}(j)} \bar{\Phi}(C)
\]
is (up to weak equivalence) the pushout, which we will denote by $E$, of 
\[
\bar{\Phi}(Mi) \hookleftarrow \bar{\Phi}(A) \hookrightarrow \bar{\Phi}(Mj).
\]
It therefore suffices to show that the map $E\to \bar{\Phi}(D)$ is a weak
equivalence.
Consider the diagram
\[
\xymatrix@C=10pt{
\pi_i\bar{\Phi}(A)
\ar[d]_=
\ar[r]
&
\pi_i(\bar{\Phi}(B)
\vee
\bar{\Phi}(C))
\ar[d]
\ar[r]
&
\pi_i E
\ar[d]
\ar[r]
&
\pi_i S(A)
\ar[d]_{\mathfrak{S}}
\ar[r]
&
\pi_i(S(B)\vee S(C))
\ar[d]
\\
\pi_i\bar{\Phi}(A)
\ar[r]
&
\pi_i\bar{\Phi}(B\vee C)
\ar[r]
&
\pi_i\bar{\Phi}(D)
\ar[r]
&
\pi_i\bar{\Phi}(\Sigma A)
\ar[r]
&
\pi_i\bar{\Phi}(\Sigma B\vee \Sigma C),
}
\]
where the rightmost vertical arrow is induced by the maps 
\[
S(B)\xrightarrow{\mathfrak{S}} \bar{\Phi}(\Sigma B)
\to \bar{\Phi}(\Sigma B\vee \Sigma C)
\]
and 
\[
S(C)\xrightarrow{\mathfrak{S}} \bar{\Phi}(\Sigma C)
\to \bar{\Phi}(\Sigma B\vee \Sigma C).
\]
The top row of the diagram is exact because it is $\pi_i$ of a cofiber
sequence. The fact that $(\Omega_\ip)_*$ is a homology
theory implies that the second row of the diagram is exact (because it is
$\pi_i\bar\Phi$ of a cofiber sequence), and also that the
second vertical arrow is an isomorphism.
The fourth and fifth vertical arrows are isomorphisms by Lemma \ref{l2}, and
hence the middle vertical arrow is an isomorphism as required.

It remains to prove Lemma \ref{l2}.
We begin by describing a suspension map
\[
s:\pi_i \bar{\Phi}(W)\to \pi_{i+1} S(W).
\]
Let $a\in \pi_i \bar{\Phi}(W)$.
Let $\kappa_i(W)$ denote
the kernel of the map $\pi_i\Phi(W) \to \pi_i\Phi(*)$; then
$\kappa_i(W)\to\pi_i\bar\Phi(W)$ is an 
isomorphism (because the short exact sequence
$\pi_i\Phi(*)\to\pi_i\Phi(W)\to\pi_i\bar\Phi(W)$
is split by the map $W\to *$), so $a$ comes from an element $\tilde{a}\in 
\kappa_i(W)$.
Let $\Phi_0(W)$ denote the $0$-th space of the spectrum $\Phi(W)$
(\cite[Definitions 15.8 and 15.4]{LM}).  Since $\Phi(W)$ is an $\Omega$
spectrum (\cite[Proposition 15.9]{LM}) and 
$\Phi_0(W)$ is a Kan complex (\cite[Lemma 15.12]{LM}), $\tilde{a}$ is
represented by an $i$-simplex $\sigma$ of $\Phi_0(W)$ with all faces at the
basepoint.  Let $\sigma'$ be the image of $\sigma$ in $\Phi_0(CW)$; then
$\sigma'$ represents an element of $\kappa_i(CW)$. Since 
$\kappa_i(CW)=0$, there is an $(i+1)$-simplex $\tau$ of
$\Phi_0(CW)$ with $\partial_0(\tau)=\sigma$ and all other faces at the
basepoint.
Let $\tau_1$ and $\tau_2$ be the images of $\tau$ under the two inclusions of
$\bar\Phi_0(CW)$ into the $0$-th space of $S(W)$ and let 
\begin{equation}
\label{e5}
D=\Delta^i\cup_{d^0 \Delta^i} \Delta^i;
\end{equation}
then $\tau_1$ and $\tau_2$ give a map
\[
\mathbf a: D/\partial D\to S(W)
\]
which represents $s(a)$.

Combining the map $s$ with the
weak equivalence between $S(W)$ and $\Sigma\bar\Phi(W)$ gives the usual
suspension map $\pi_i\bar{\Phi}(W)\to \pi_{i+1}\Sigma\bar\Phi(W)$, as the 
reader can verify, so $s$ is an isomorphism.

It therefore suffices to show that the composite
\begin{multline}
\label{e3}
(\Omega_\ip)_i(W,*)
\cong
\mathrm{ker}((\Omega_\ip)_i(W)\to (\Omega_\ip)_i(*))
\cong
\kappa_i(W)
\cong
\pi_i\bar\Phi(W)
\\
\xrightarrow{s}
\pi_{i+1} S(W)
\xrightarrow{\mathfrak{S}}
\pi_{i+1} \bar\Phi(\Sigma W)
\cong
(\Omega_\ip)_{i+1}(\Sigma W,*)
\end{multline}
is an isomorphism.  We will show that this composite is equal to the
suspension isomorphism
\begin{equation}
\label{e4}
s':(\Omega_\ip)_i(W,*)\to (\Omega_\ip)_{i+1}(\Sigma W,*)
\end{equation}
of the homology theory $(\Omega_\ip)_*$.
First we give an explicit description of the composite \eqref{e3}.  
Let $b\in (\Omega_\ip)_i(W,*)$ and let $\tilde{b}$ be the corresponding
element of $\mathrm{ker}((\Omega_\ip)_i(W)\to (\Omega_\ip)_i(*))$.  Then 
$\tilde b$ is represented by a map $f:M\to W$ 
where $M$ is an $i$-dimensional IP-space which is a boundary; say $M=\partial 
N$.  Recall that $k$-simplices of $\Phi_0(W)$ are the same thing as elements 
of $\ad^0_{\ip,W}(\Delta^k)$.
Let $\sigma$ be the $i$-simplex of $\Phi_0(W)$ 
which takes $\Delta^i$ to $f$ and all faces
of $\Delta^i$ to $\emptyset \to W$, and let $\sigma'$ be the image of $\sigma$
in $\Phi_0(CW)$.
We can construct an $(i+1)$-simplex $\tau$ of $\Phi_0(CW)$ 
with $\partial_0(\tau)=\sigma'$ and all other faces at the
basepoint as follows.
Let $P$ be 
\[
(I\times M)\cup_{1\times M}\, N,
\]
let $g:P\to CW$ take $(t,x)\in I\times M$ to $[t,f(x)]\in CW$ and $N$ to 
$[1,*]$, and finally let $\tau$ 
take $\Delta^{i+1}$ to $g$, $d^0 \Delta^{i+1}$ to $f$, and the remaining
faces to $\emptyset\to W$.  
Let $\tau_1$ and $\tau_2$ be the images of $\tau$ under the
two maps $\Phi_0(CW)\to
\bar\Phi_0(\Sigma W)$; then (with the notation of Equation \eqref{e5}) 
$\tau_1$ and $\tau_2$ give a map
\[
{\mathbf b}: D/\partial D\to \bar\Phi_0{\Sigma W}
\]
which represents an element of 
$\pi_{i+1} \bar\Phi(\Sigma W)$, and the image of this element in
$(\Omega_\ip)_{i+1}(\Sigma W,*)$ is the image of $b$ under the composite
\eqref{e3}.

Next we show that this description of the image of $b$ in $\pi_{i+1} 
\bar\Phi(\Sigma W)$ can be simplified.
Let the two copies of 
$CW$ in $\Sigma W$ be $[-1,0]\wedge W$ and $[0,1]\wedge W$, where the 
basepoints of $[-1,0]$ and $[0,1]$ are $-1$ and $1$.
Let 
\[
Q=N\cup_{-1\times M} ([-1,1]\times M)\cup_{1\times M} N,
\]
and let $h:Q\to \Sigma W$ take $(t,x)$ to $[t,f(x)]$ and both copies of $N$ to
the basepoint.  Let $\tau_3$ be the $(i+1)$-simplex of $\Phi_0(\Sigma W)$
which takes $\Delta^{i+1}$ to $h$ and all faces of $\Delta^{i+1}$ to
$\emptyset\to \Sigma W$; then $\tau_3$ gives a map 
\[
{\mathbf b}':\Delta^{i+1}/\partial \Delta^{i+1}\to \bar\Phi(\Sigma W),
\]
and we claim that $\mathbf b$ and ${\mathbf b}'$ represent the same element of
$\pi_{i+1} 
\bar\Phi(\Sigma W)$.  To see this,  
let $\upsilon$ be the $(i+2)$-simplex of $\Phi_0(\Sigma W)$ 
such that
\begin{itemize}
\item
$\upsilon(\Delta^{i+2})$ is the composite
\[
I\times Q\xrightarrow{p} Q\xrightarrow{h} \Sigma W
\]
(where $p$ is the projection),
\item
$\upsilon(d^0 \Delta^{i+2})=h\circ p|_{0\times Q}$, 
\item
$\upsilon(d^1\Delta^{i+2})=h\circ p|_{1\times (N\cup_{-1\times M} 
([-1,0]\times M))}$,
\item
$\upsilon(d^2\Delta^{i+2})=h\circ p|_{1\times(([0,1]\times M)\cup_{1\times 
M} N)}$,
\item
$\upsilon$ takes all other faces to $\emptyset\to\Sigma W$.
\end{itemize}
Then $\upsilon$ gives a homotopy between $\mathbf b$ and ${\mathbf b}'$ which
verifies the claim.

The element of $(\Omega_\ip)_{i+1}(\Sigma W,*)$ corresponding to ${\mathbf b}'$
is represented by the map $h:Q\to \Sigma W$; this completes our calculation of
the image of $b$ under the composite \eqref{e3}. 

Next we claim that the image of $b$ under the map \eqref{e4} is represented by
$h|_{Q'}$, where 
\[
Q'=([-1,1]\times M)\cup_{1\times M} N.
\]
To see this, recall that the suspension map $s'$ is defined to be the inverse
of the composite
\[
(\Omega_\ip)_{i+1}(\Sigma W,*)
\xleftarrow[\cong]{q}
(\Omega_\ip)_{i+1}(C'W,W)
\xrightarrow{\partial}
(\Omega_\ip)_i(W,*),
\]
where $C'W=[-1,1]\wedge W$ (with the basepoint of $[-1,1]$ at 1), $q$ is the
quotient map, and 
$\partial$ is the boundary map of the homology theory $(\Omega_\ip)_*$.  
Now consider the map 
\[
k:Q'\to C'W
\]
which takes $(t,x)$ to $[t,f(x)]$ and $N$ to the basepoint. 
Recall from \cite[Section 5]{pardon} that the boundary map $\partial$ is
defined as in \cite[Section 4]{cf}; thus $\partial[k]=[f]=b$.
We also have $q\circ
k=h|_{Q'}$, so $s'(b)$ is represented by $h|_{Q'}$ as claimed. 

To complete the proof of Lemma \ref{l2} we observe that $h$ and $h|_{Q'}$
represent the same element of $(\Omega_\ip)_{i+1}(\Sigma W,*)$ because the 
composite
\[
I\times Q\xrightarrow{p} Q\xrightarrow{h} \Sigma W
\]
(where $p$ is the projection) is a bordism, in the sense of \cite[Section
4]{cf}, between $h$ and $h|_{Q'}$.

\appendix

\section{The intrinsic filtration of a finite-dimensional PL space}
\label{a1}

Let $X$ be a PL space.  Say that two points $x_1,x_2\in X$ are {\it equivalent}
if there are neighborhoods $U_1$ of $x_1$ and $U_2$ of $x_2$ with a PL
homeomorphism of pairs $(U_1,x_1)\approx (U_2,x_2)$.
Let $X$ be finite dimensional. Choose a triangulation of $X$, and
let $X(i)$ be the $i$-skeleton of this triangulation.  The {\it intrinsic
filtration} of $X$ is the filtration $X^i$ for which $x\in X^i$ if and only if
all points equivalent to $x$ are in $X(i)$.  This filtration is independent of
the chosen triangulation because it is the coarsest PL CS
stratification\footnote{The definition of CS stratification is a weaker version
of Definition \ref{d1}.} 
of $X$ (cf.\ \cite[Remark 2.10.7]{Fr}).
We will use the notation $X^i$ throughout this appendix to denote the intrinsic
filtration of $X$.

\begin{prop}
\label{p1}
Let $X$ be a finite-dimensional PL space.

{\rm (i)}  If $U$ is an open subset of $X$ then $U^i=X^i\cap U$.

{\rm (ii)} If $M$ is a PL manifold of dimension $m$ then 
\[
(X\times M)^i=
\begin{cases}
X^{i-m}\times M & m\leq i\leq \dim X + m, \\
\emptyset & \text{otherwise.}
\end{cases}
\]

{\rm (iii)} If $f:X\to Y$ is a PL homeomorphism then $f(X^i)=Y^i$.

{\rm (iv)} If $X$ is a PL pseudomanifold then the intrinsic filtration on $X$ is
a stratification in the sense of Definition \ref{d1}.
\end{prop}

\begin{proof}
Part (i) is \cite[Lemma 2.10.16]{Fr},
part (ii) is \cite[Lemma 2.10.17]{Fr},
part (iii) is immediate from the definition of the intrinsic filtration,
and part (iv) is \cite[Proposition 2.10.18]{Fr}.
\end{proof}

The following fact is \cite[Proposition 2.10.23]{Fr}.

\begin{prop}
\label{p2}
Let $X$ be a PL $\partial$-pseudomanifold, and define subsets $X[i]$ by letting
$X[i]\cap (X-\partial X)=(X-\partial X)^i$ and 
\[
X[i]\cap \partial X=
\begin{cases}
(\partial
X)^{i-1} & 1\leq i\leq \dim X,\\
\emptyset & \text{otherwise.}
\end{cases}
\]
Then the filtration $X[i]$ gives $X$ the structure of a stratified PL
$\partial$-pseudomanifold.
\qed
\end{prop}

\section{Modules over additive categories.}
\label{add}

See \cite[Section 9A]{luck} for an introduction to this topic.

An {\it additive category} is a small category in which the morphism sets are 
abelian
groups and the composition is bilinear.  A functor between additive categories
is {\it additive} if it is a homomorphism on each morphism set.

\begin{example}
A ring is the same thing as an additive category with a single object, and a
ring homomorphism is an additive functor.
\end{example}

By a left (resp., right) {\it module} over an additive category $\CC$ we mean
a covariant (resp., contravariant) additive functor $\M$ from $\CC$ to the 
category of abelian groups; we write $\M_c$ for the value of $\M$ at an object
$c$ of $\CC$.

If $\M$ and $\N$ are left modules over $\CC$  we define
$\Hom_\CC(\M,\N)$ to be the abelian group of natural transformations.

If $\M$ is a right module over $\CC$ and $\N$ is a left module, the {\it
tensor product} $\M\otimes_\CC \N$ is the abelian group 
\[
\bigoplus_{c\in \ \mathrm{Ob}(\CC)}
(\M_c\otimes \N_c)/Q;
\]
here $Q$ is generated by the elements
$\M(f)(m)\otimes n-m\otimes \N(f)(n)$, where $f$ runs through the morphisms in
$\CC$, $m\in \M_{\mathrm{target}(f)}$, and $n\in 
\N_{\mathrm{source}(f)}$.

The {\it tensor product} of two additive categories $\CC$ and $\CC'$ is the
additive category whose object set is $\mathrm{Ob}(\CC)\times \mathrm{Ob}(\CC')$,
and whose morphism set from $(c_1,c'_1)$ to $(c_2,c'_2)$ is
$\mathrm{Mor}_\CC(c_1,c_2)\otimes \mathrm{Mor}_\CC'(c_1',c_2')$. 

A $(\CC,\CC')$ {\it bimodule} is a left module over $\CC^\mathrm{op}\otimes
\CC'$.

Let $F:\CC\to \CC'$ be an additive functor.  The {\it canonical bimodule}
determined by $F$ is the $(\CC,\CC')$ bimodule 
$\P_F$ which takes $(c,c')$ to
$\mathrm{Mor}_{\CC'}(F(c),c')$, with the evident action on morphisms.
If $\M$ is a left $\CC$ module, the {\it Kan extension} of $\M$ along $F$,
denoted $\mathrm{Kan}_F\M$, is 
the left $\CC'$ module that takes $c'$ to $\P_F(-,c')\otimes_\CC \M$, with the
evident action on morphisms.  

\begin{example}
\label{aa29}
If $h:R\to S$ is a ring homomorphism and $M$ is an $R$-module then
$\mathrm{Kan}_h M$ is isomorphic to $S\otimes_R M$.
\end{example}

There is a natural map $\iota:\M\to \mathrm{Kan}_F\M$ which takes $m\in \M_c$ 
to
$\mathrm{id}_{F(c)}\otimes m$.
There is a natural bijection between
$\Hom_{\CC'}(\mathrm{Kan}_F\M,\N)$ (where $\N$ is a $\CC'$ module) and 
$\Hom_\CC(\M,\N)$ which takes $h$ to $h\circ \iota$.

% \begin{example}
% \label{aa32}
% If $\M$ is a $\Z[\pi_1 Z]$ module, $z$ is an element of $Z$, and 
% $i:\Z[\pi_1(Z,z)]\to \Z[\pi_1 Z]$ is the inclusion of categories, then
% there is a canonical isomorphism
% \[
% \mathrm{Kan}_i \M_z\to \M
% \]
% which takes $g\otimes a$ to $g_*(a)$, where $g$ is any morphism with source
% $z$ and $a\in \M_z$.
% \end{example}

If $G:\CC'\to\CC''$ is an additive functor there is an isomorphism
\[
\P_G\otimes_{\CC'}\P_F\cong
\P_{G\circ F}
\]
of $(\CC,\CC'')$ bimodules which on objects is induced by the composition map
$\mathrm{Mor}_{\CC''}(G(c'),c'')\otimes \mathrm{Mor}_{\CC'}(F(c),c')
\to 
\mathrm{Mor}_{\CC''}(G(F(c)),c'')$; thus there is an isomorphism of $\CC''$
modules 
\begin{equation}
\label{aa28}
\mathrm{Kan}_G(\mathrm{Kan}_F(\M))\cong \mathrm{Kan}_{G\circ F}\M.
\end{equation}

For use in Section \ref{gs6} we record

\begin{remark}
\label{aa33}
Recall the definition of the additive category $\Z[\pi_1 Z]$ from Section
\ref{ringoids} (and also recall that we denote composition of path homotopy
classes by letting
$\delta\gamma$ be ``first $\gamma$, then $\delta$'', analogously to
composition of functions).
Let $g:Z\to Z'$ be a map of connected spaces, and let $g_*:\Z[\pi_1 Z]\to
\Z[\pi_1 Z']$ be the induced functor.  Let $\M$ be a $\Z[\pi_1 Z]$ module 
and
let $z\in Z$.  We can give an explicit description of 
$\mathrm{Kan}_{g_*}(\M)_{g(z)}$,
as follows.  By definition, $\mathrm{Kan}_{g_*}(\M)_{g(z)}$ is
$\P_{g_*}(-,g(z))\otimes_{\Z[\pi_1 Z]} \M$, and $\P_{g_*}(w,g(z))$ is the 
free abelian group generated by the path homotopy classes $\delta$ from 
$g(w)$ to $g(z)$.  There is a map
\[
\Z[\pi_1(Z',g(z)]\otimes_{\Z[\pi_1(Z,z)]} \M_z\to 
\mathrm{Kan}_{g_*}(\M)_{g(z)}
\]
which takes $\gamma\otimes m$ to the class of $\gamma\otimes m$.  This map is
an isomorphism (it's straightforward to check that its inverse takes
$\delta\otimes n$, where $\delta$ is a path homotopy class from $g(w)$ to
$g(z)$ and $n\in \M_{w}$, 
to $\delta(g\circ \epsilon)^{-1}\otimes \epsilon n$,
where
$\epsilon$ is any path homotopy class from $w$ to $z$).
The diagram
\[
\xymatrix{
\M_z
\ar[rd]_\iota
\ar[r]
&
\Z[\pi_1(Z',g(z)]\otimes_{\Z[\pi_1(Z,z)]} \M_z
\ar[d]_\cong
\\
&
\mathrm{Kan}_{g_*}(\M)_{g(z)}
}
\]
commutes.

Similarly, for $z_1,z_2\in Z$ and modules $\M_1,\M_2$ over $\Z[\pi_1 Z]$, there
is an isomorphism
\begin{multline*}
(\Z[\pi_1(Z',g(z_1)]\otimes_{\Z[\pi_1(Z,z_1)]} (\M_1)_{z_1})
\otimes
(\Z[\pi_1(Z',g(z_2)]\otimes_{\Z[\pi_1(Z,z_2)]} (\M_2)_{z_2})
\\
\to 
\mathrm{Kan}_{g_*\otimes g_*}(\M_1\otimes \M_2)_{(g(z_1),g(z_2))}.
\end{multline*}
\end{remark}

\section{Subdivision of singular simplices}
\label{subdiv}

In this section we record some facts related to subdivision which will be
needed in the next section.

Recall that intersection chains can be defined for any filtered space
(\cite[Section 2]{fr}).  For a chain $\eta$, write $\supp(\eta)$ for the 
support of $\eta$, that is, the union of the images of the singular simplices 
that have nonzero coefficient in $\eta$.

\begin{prop}
\label{aa17}
Let $Z$ be a filtered space, let $U\subset Z$ be open, and 
let $\bar p$ be a perversity.
Then there is an operation which takes each singular
simplex $s$ to a chain $\bar{s}$, with the following properties.

{\rm{(i)}} If $\supp(s)\subset U$ then $\bar{s}=s$.

{\rm{(ii)}} $\supp(\bar{s})\subset U\cap \supp(s)$.

{\rm{(iii)}} If $s$ is allowable then so are all singular simplices that 
belong to
$\bar{s}$.

{\rm{(iv)}} If $\xi=\sum a_i  s_i$ is an element of $IS^{\bar p}_*(Z;\Z)$ with $a_i \in \Z$, 
write 
\[
\bar{\xi}=
\sum a_i\bar{s}_i.
\]
Then $\bar{\xi}\in IS^{\bar p}_*(U;\Z)$.
\end{prop}

\begin{proof}
In the construction of \cite[page 155 line $-7$ to page 156 line 
19]{friedmanmcclure}, ignore $\xi_i$, replace $\mathfrak B$ 
by the set of all singular simplices, replace $\mathfrak B^j$ by the set of
singular simplices of dimension $j$, and replace $U_1$ by $U$.  Define
$\bar{s}$ by the equation on line 15 of page 156 of
\cite{friedmanmcclure}.  Then (i) and (ii) are immediate, and (iii) follows
from \cite[Lemma 2.6]{FTrans}.  To see (iv),
first note that
Equation (9)
on page 156 of \cite{friedmanmcclure} remains valid with the same proof.  
Let $k$ be the dimension of $\xi$.
Then
\[
\partial \bar{\xi}\equiv
\sum_{ t\in \mathfrak B^{k-1}\mathrm{\ not allowable}} 
\left[
\sum_{s\in {\mathfrak B}^k} 
c_{\xi}(s)
c_{\partial s}( t)
\right]
\bar{ t}
\]
modulo allowable singular simplices.
The expression in brackets is equal to the coefficient of $ t$ in
$\partial\xi$, and this is zero because $\xi$ is an intersection chain.  Hence
$\bar{\xi}$ is allowable.
\end{proof}

\begin{lemma}
\label{fix}
Let $Z$ be a filtered space, let $\{W_1,\ldots,W_n\}$ be an
open cover of $Z$,
let $\bar p$ be a perversity 
% with $\bar{p}\leq \bar{t}$, 
and let $i\in \Z$.  Then 

{\rm{(i)}} the intersection of 
$IS^{\bar p}_i (Z;\Z)$ with $\sum_j S_i (W_j;\Z)$ (considered as
subgroups of $S_i (Z;\Z)$) is $\sum_j IS^{\bar p}_i
(W_j;\Z)$, and

{\rm{(ii)}} $IS^{\bar p}_i (Z;\Z)/\sum_j IS^{\bar p}_i
(W_j;\Z)$ is free over $\Z$.
\end{lemma}

% \begin{remark}
% \label{aaa}
% The requirement that  $\bar p\leq \bar t$ ensures that $IS^{\bar p}_i(Z;\Z)$
% is a subgroup of $S_i(Z;\Z)$ (\cite[Proposition 6.2.7]{Fr}).
% \end{remark}

\begin{proof}[Proof of Lemma \ref{fix}]
For part (i),
the proof is by induction on $n$.
Let 
\[
\xi\in IS^{\bar p}_i (Z;\Z) \cap \sum_{j=1}^n S_i (W_j;\Z).
\]
Write
\[
\xi=\sum  a_m s_m;
\]
then every $\supp( s_m)$ is contained in some $W_j$.
Apply Proposition \ref{aa17} with $U=W_1$ to get 
$\bar{\xi}\in IS^{\bar p}_* (W_1;\Z)$. Parts (i) and (ii) of Proposition 
\ref{aa17}
show that every singular simplex that belongs to $\xi-\bar{\xi}$ has support
in some $W_j$ with $j\geq 2$. Let $Z'=\cup_{j=2}^n W_j$. Then 
$\xi-\bar{\xi}$ is
an element of 
\[
IS^{\bar p}_i (Z';\Z) \cap \sum_{j=2}^n S_i (W_j;\Z),
\]
so by the inductive hypothesis 
\[
\xi-\bar{\xi}\in \sum_{j=2}^n  IS^{\bar p}_i (W_j;\Z),
\]
and therefore 
\[
\xi\in \sum_{j=1}^n IS^{\bar p}_i (W_j;\Z)
\]
as required.

For part (ii), note that by part (i) the map
\[
IS^{\bar p}_i (Z;\Z)/\sum_j IS^{\bar p}_i
(W_j;\Z)
\to
S_i(Z,\Z)/\sum_j S_i (W_j;\Z)
\]
is 1-1 and that $S_i(Z,\Z)/\sum_j S_i (W_j;\Z)$ is freely generated by the
singular simplices that do not land in any $W_j$.
\end{proof}

\begin{remark}
Part (i) of this result gives a simple proof of \cite[Lemma 
6.11]{friedmanmcclure}, as
follows.  With the notation of that Lemma, 
\[
\xi_1=-\xi_2-\cdots -\xi_m,
\]
so every singular simplex that belongs to $\xi_1$ has support in some $U_i$
with $i\geq 2$.  Thus
\[
\xi_1\in IS^{\bar p}_*(X;\Z) \cap \sum_{i=2}^m S_*(U_1\cap U_i;\Z),
\]
and by Lemma \ref{fix}(i) there exist $\eta_i\in IS^{\bar p}_*(U_i\cap U_1;\Z)$
for $2\leq i\leq m$ with
\[
\xi_1=-\eta_2-\cdots -\eta_m
\]
as required.
\end{remark}

\section{Extensions of some results from 
\cite{friedmanmcclure}}
\label{tech}

In this appendix we show that the results of
\cite[Subsections 6.2 and 6.3]{friedmanmcclure} remain valid with 
the field $F$ 
replaced by $\Z$.  We will use the notation of those subsections, 
except that we denote singular chains by $S_*$ instead of $C_*$ and we write 
$IS^{\bar p}_*$ instead of $I^{\bar p}C_*$ and $IS^{\bar p,\uu}_*$ instead of 
$I_\uu^{\bar p}C_*$.  

First we have an analog of \cite[Lemma 6.6]{friedmanmcclure} (the extra
generality in the statement is used in Subsections \ref{ss2} and \ref{aa3}).

We need some definitions and notation.
Let $Z$ be a filtered space and let $\bar p$ be a perversity. 
% with $\bar p\leq
% \bar t$ (see Remark \ref{aaa}).  
Let
$p:\tilde{Z}\to Z$ be a regular cover with automorphism group $\pi$, and for a subset 
$A\subset Z$ write $\tilde A$ for $p^{-1}(A)$.  Recall that $A$ is said to be
{\it evenly covered} if the cover $\tilde{A}\to A$ is trivial. Let
$\S$ be a finite collection of subsets of $Z$, and let
$\CC(\S)$ be the category whose objects 
are the elements of $\S$ and whose morphisms are inclusions.  
Let 
\[
\Phi:\CC(\S)\to \Z[\pi]{\mathrm{-mod}}
\]
be the functor that takes $A$ to $IS^{\bar p}_* (\tilde{A};\Z)$.
For each 
element $A$ of $S$ let $\CC(A)$ be the full subcategory of $\CC(\S)$ whose
objects are the elements of $\S$
which are proper subsets of $A$, and let ${\mathbf C}(A)$ be the cokernel of
the map
\[
\colim_{B\in \CC(A)} IS^{\bar p}_*(B;\Z)
\to
IS^{\bar p}_*(A;\Z).
\]
Let 
\[
\Psi:\CC(\S)\to \Z[\pi]{\mathrm{-mod}}
\]
be the functor which takes $A$ to 
\[
\bigoplus_{B\in \S \ \mathrm{and} B\subset A} \Z[\pi]\otimes{\mathbf C}(B)
\]
and has the obvious definition on morphisms.

\begin{lemma}
\label{aa15}
Suppose that, for every $A\in\S$,

{\rm{(a)}} $A$ is evenly covered,

{\rm{(b)}} the map
\[
\colim_{B\in \CC(A)} IS^{\bar p}_* (B;\Z)
\to
IS^{\bar p}_* (A;\Z)
\]
is a monomorphism, and

{\rm{(c)}}
$\mathbf{C}(A)$ is free over $\Z$.  

Then 

{\rm{(i)}} $\Phi$ is naturally isomorphic to $\Psi$, and

{\rm{(ii)}} $\colim_{A\in \CC(\S)} IS^{\bar p}_* (\tilde{A};\Z)
\cong
\bigoplus_{A\in \S} \Z[\pi]\otimes{\mathbf{C}(A)}$, and

{\rm{(iii)}} for each $\S'\subset \S$, the map
\[
\colim_{A\in \CC(\S')} IS^{\bar p}_* (\tilde{A};\Z)
\to
\colim_{A\in \CC(\S)} IS^{\bar p}_* (\tilde{A};\Z)
\]
has a left inverse over $\Z[\pi]$.

\end{lemma}

\begin{proof}[Proof of Lemma \ref{aa15}]
Part (ii) follows from part (i), since $\bigoplus_{A\in \S}
\Z[\pi]\otimes{\mathbf{C}(A)}$ satisfies the universal property to be 
$\colim_{A\in \CC(\S)} \Psi(A)$,
and (iii) is immediate from part (ii).

Proof of (i).
For each $A\in \S$, choose an isomorphism $\iota_A$ from the trivial cover
$\pi\times  A\to A$ to $p|_A$
and a splitting $s_A$ of the map $IS^{\bar p}_*(A;\Z)\to \mathbf{C}(A)$.
Let 
\[
\nu_A:\Psi(A)\to \Phi(A)
\]
be the map whose restriction to $\Z[\pi]\otimes{\mathbf C}(B)$
is the composite
\[
\Z[\pi]\otimes{\mathbf C}(B)
\xrightarrow{1\otimes s_B}
\Z[\pi]\otimes IS^{\bar p}_*(B;\Z)
\cong
IS^{\bar p}_*(\pi\times B;\Z)
\xrightarrow{\iota_B}
IS^{\bar p}_*(\tilde{B};\Z)
\to
IS^{\bar p}_*(\tilde{A};\Z).
\]
This is a natural transformation $\Psi\to \Phi$.  To see that $\nu_A$ is an
isomorphism, let us assume inductively that this has been shown for all
$B\subset A$.  Consider the diagram
\[
\xymatrix{
\colim_{B\in \CC(A)} \Psi(B)
\ar[d]_{\colim \nu_B}
\ar[r]
&
\Psi(A)
\ar[d]_{\nu_A}
\ar[r]
&
\Z[\pi]\otimes{\mathbf C}(A)
\ar[d]_=
\\
\colim_{B\in \CC(A)} \Phi(B)
\ar[r]
&
\Phi(A)
\ar[r]
&
\Z[\pi]\otimes{\mathbf C}(A),
}
\]
where the upper right arrow is the projection and the lower right arrow is the
composite
\[
\Phi(A)
\xrightarrow{\iota_A^{-1}}
IS^{\bar p}_*(\pi\times A;\Z)
\cong
\Z[\pi]\otimes IS^{\bar p}_*(A;\Z)
\to
\Z[\pi]\otimes {\mathbf{C}}(A).
\]
The diagram
commutes, and the left vertical arrow is an isomorphism, so it suffices to
show that 
the 
rows are short exact.  For the upper row this is clear, and for
the lower row this follows from hypothesis (b) and the definition of 
${\mathbf{C}}(A)$.
\end{proof}

\begin{remark}
\label{aa18}
% (i) All results in \cite[Section 6]{friedmanmcclure} should have
% included the
% hypothesis that $\bar p\leq \bar t$ (because this is needed for the proof of 
% \cite[Proposition 6.3]{friedmanmcclure}; 
% see Remark \ref{aaa}).
% 
% (ii)
If 
% $\bar p\leq \bar t$, 
all $A\in\S$ are open and $\S$ is closed under intersection then 
hypotheses (a) and (b) of Lemma \ref{aa15} are satisfied: (a) follows from 
\cite[Proposition 6.3]{friedmanmcclure} and (b) from \cite[Proposition 
6.3]{friedmanmcclure} and Lemma \ref{fix}(ii).
\end{remark}

Next we have the analog of 
\cite[Proposition 6.4]{friedmanmcclure}. 

\begin{prop}
\label{f1}
Let $Z$ be a filtered space and let $A$ be an open subset of $Z$.  Let $\bar 
p$ be a perversity. 
% with $\bar p\leq \bar t$.

{\rm (i)}
If $Z$ has a finite covering by evenly covered open sets (in particular, if
$Z$ is compact) then $IS^{\bar p}_*(\td{Z},\td{A};\Z)$ is chain
homotopy equivalent over $\Z[\pi]$ to a nonnegatively-graded chain complex of
free $\Z[\pi]$-modules.

{\rm (ii)}
For all $Z$, $IS^{\bar p}_*(\td{Z},\td{A};\Z)$ is chain homotopy
equivalent over $\Z[\pi]$ to a nonnegatively-graded chain complex of flat
$\Z[\pi]$-modules.
\end{prop}

This follows from the proof 
of \cite[Proposition 6.4]{friedmanmcclure}, using Lemma \ref{aa15} and Remark
\ref{aa18}(ii) instead of
\cite[Lemma 6.6]{friedmanmcclure}. 

Our next goal is to prove an analogue of \cite[Proposition
6.5]{friedmanmcclure}.  First recall that with $\Z$ coefficients the
K\"unneth theorem that was used in the proof of \cite[Proposition
6.5]{friedmanmcclure} requires a hypothesis about torsion; see \cite[Definition
6.4.5 and Theorem 6.4.7]{Fr} for the precise statement.  It will be useful to
have the following terminology from
\cite{gorsie}.

\begin{definition}
\label{x1}
Let $\bar{p}$ be any perversity.
A pseudomanifold is called \emph{locally $\bar{p}$-torsion free}
if $IH^{\bar{p}}_{c-2-\bar{p}(c)} (L')$ is torsion free for every link
$L'$, where $\dim L' =c-1$.
\end{definition}

\begin{remark}
\label{x2}
IP-spaces are locally $\bar{m}$-torsion free, since if $c-1$
is even, $c-2-\bar{m}(c)=\frac{1}{2}(c-1)$ and $IH_{(c-1)/2} (L')=0$,
while if $c-1$ is odd, 
$c-2-\bar{m}(c)=\frac{c}{2} -1$ and $IH_{c/2-1} (L')$ is torsion free.
\end{remark}

The proof of \cite[Proposition 6.5]{friedmanmcclure}
goes through without change 
% but actually need \cite[Section 7.3.10]{Fr} instead of usual K\"unneth
% because usual doesn't apply to $\partial$-pseudomanifolds since they're not
% CS spaces
to show

\begin{prop}
\label{P: tensor flat}
Let $X$ be a 
stratified PL $\partial$-pseudomanifold which is locally $\bar p$-torsion 
free or locally $\bar q$-torsion
free and let $A$ be an open subset of $X$.
Then the cross product
\[
IS^{\bar p}_*(\td X,\td A;\Z)\otimes_\Z IS^{\bar q}_*(\td X,\td A;\Z)
\to
IS^{Q_{\bar p,\bar q}}_*(\td X\times \td X,\td A\times \td X\cup \td X\times
\td A;\Z)
\]
induces a quasi-isomorphism
\begin{multline*}
\Z\otimes_{\Z[\pi]}
(
IS^{\bar p}_*(\td X,\td A;\Z)\otimes_\Z IS^{\bar q}_*(\td X,\td A;\Z)
)
\\
\to
\Z\otimes_{\Z[\pi]} IS^{Q_{\bar p,\bar q}}_*(\td X\times \td X,\td A\times \td
X\cup \td X\times
\td A;\Z).
\end{multline*}
\end{prop}

\begin{remark}
\label{aaa1}
Let us say that a perversity is {\it classical} if it satisfies the original
definition given by Goresky and MacPherson in \cite{gm1}.   By \cite[Theorem
5.5.1]{Fr}, if $X$ is a stratified PL pseudomanifold, $\mathfrak X$ is the
same PL pseudomanifold with the intrinsic stratification, and $\bar p$ is a
classical perversity, then the canonical map
\[
IS_*^{\bar p}(X;\Z)\to IS_*^{\bar p}(\mathfrak X;\Z)
\]
is a quasi-isomorphism.
\end{remark}

Next we have the analogue of \cite[Proposition 5.15]{friedmanmcclure}, except
that we require the perversity to be classical.

\begin{prop}
\label{finite}
Let $X$ be a compact stratified PL $\partial$-pseudomanifold. 
Let $\bar{p}$ be a classical perversity.  Then
$IS^{\bar p}_*(\td X;\Z)$ is
quasi-isomorphic over $\Z[\pi]$ to a finite $\Z[\pi]$ chain complex.
\end{prop}

\begin{proof}
It suffices to show that \cite[Lemma 6.7]{friedmanmcclure} remains valid with
$F$ replaced by $\Z$.  The only place in the proof of that result
where the hypothesis that $F$ is a field
is used is in the second paragraph of the proof of Lemma 6.9, as part of the 
verification that 
$IS^{\bar p}_*(\St(s);F)$ is homotopy finite over $F$, so we need to prove 
that 
$IS^{\bar p}_*(\St(s);\Z)$ is homotopy finite over $\Z$, where $s$ is a
simplex with no vertices in $\partial X$.

For the remainder of the proof, for a PL space $Y$, we will write $Y^*$ for $Y$
with its intrinsic filtration. 

It suffices to show that 
$IS^{\bar p}_*(\St(s)^*;\Z)$ is homotopy finite over $\Z$, because this is
quasi-isomorphic to $IS^{\bar p}_*(\St(s);\Z)$ by Remark \ref{aaa1}, and hence
it is chain homotopy equivalent because both chain complexes are free over $\Z$.
% Since $X$ satisfies conditions (a) and (b) of Definition \ref{bpseud}, it 
% follows that $\overline{\St}(s)$ also satisfies them.  Since
% $\overline{\St}(s)=s*\mathrm{Lk}(s)$, it follows that $\mathrm{Lk}(s)$ is a PL
% pseudomanifold, and hence that 
% $(\partial s)*\mathrm{Lk}(s)$ is a PL pseudomanifold.  

It is shown in the proof of \cite[Lemma 6.9]{friedmanmcclure} that $\St(s)$ 
is homeomorphic as a PL space to the open cone $c^\circ A$, where $A=(\partial
s)*\mathrm{Lk}(s)$, so it suffices to show that 
$IS^{\bar p}_*((c^\circ A)^*;\Z)$
is homotopy finite over $\Z$.  Now by Proposition
\ref{p1}  the restriction of the filtration of $(c^\circ A)^*$ to 
$A\times (0,1)$ is the same as the filtration of $A^*\times (0,1)$.
Hence if we let let $B$ be
\[
([0,\frac{1}{2}]\times A)/(0\times x\sim 0\times y),
\]
with the filtration inherited from $(c^\circ A)^*$,
then
the evident homotopy equivalence from $(c^\circ A)^*$ to
$B$ is a stratified homotopy
equivalence (see \cite[Appendix A]{FM} for the definition of stratified
homotopy equivalence).
Thus (by the last paragraph of \cite[Appendix A]{FM}) it suffices to show 
that the intersection chain complex of
$B$ with this filtration is homotopy finite over $\Z$.  

Next we need a lemma:

\begin{lemma}
\label{aab}
$B$, with the filtration inherited from $(c^\circ A)^*$, is a stratified
PL $\partial$-pseumanifold.
\end{lemma}

Before proving this, we note that it implies,
by
\cite[Corollary 5.4.6]{Fr}, that the intersection chain complex of $B$ is 
quasi-isomorphic (and hence chain homotopy equivalent) to the simplicial 
intersection chain complex $IC^{\bar p,T}_*(B;\Z)$ for any suitable 
triangulation $T$.  Since $B$ is compact, $IC^{\bar p,T}_*(B;\Z)$ is a 
finite chain complex, which completes the proof of Proposition \ref{finite}.
\end{proof}

\begin{proof}[Proof of Lemma \ref{aab}]
First we show that $\mathrm{Lk}(s)$ is a PL pseudomanifold.  The fact that it 
satisfies part (a) of Definition \ref{pseud} is immediate from the 
corresponding condition for $X$ and the definition of $\mathrm{Lk}(s)$.
The fact that $\mathrm{Lk}(s)$ satisfies part (b) of Definition \ref{pseud}
follows easily from the corresponding condition for $X$, the definition of
$\mathrm{Lk}(s)$, and the fact that $s$ is not contained in $\partial X$.

Now it follows that $A=(\partial s)*\mathrm{Lk}(s)$ is a PL pseudomanifold, and
hence $A^*$ is a stratified PL pseudomanifold by Proposition \ref{p1}(iv).

Next we observe that $B$ is the union of two open sets $U_1$ and $U_2$, each 
of which (with its inherited filtration) is a $\partial$-stratified  PL
pseudomanifold: let $U_1=([0,\frac{1}{2})\times A)/(0\times x\sim 0\times y)$,
which is a stratified PL pseudomanifold because it is homeomorphic as a filtered
space to $(c^\circ A)^*$ and hence to $\St(s)^*$, and let $U_2$ be 
$((0,\frac{1}{2}]\times A)$, which is a 
stratified PL $\partial$-pseudomanifold because 
(by Proposition \ref{p1})
it is homeomorphic as a 
filtered space to $((0,\frac{1}{2}]\times A^*)$.
\end{proof}

\section{Universal Poincar\'e and Lefschetz duality}
\label{univ}

In this appendix we give the analogues of \cite[Theorems 4.1 and
4.5]{friedmanmcclure} for $\Z$ coefficients.

First we construct a suitable cap product, following the method of
\cite[Section 3]{friedmanmcclure}. 

Let $X$ be a stratified PL $\partial$-pseudomanifold.

Let $p:\td X\to X$ be a regular cover with group $\pi$.  For any subset $A$ of
$X$ we write $\td A$ for $p^{-1}(A)$. We assume that $\td X$ is stratified
by the preimages of the strata of $X$. Note that $IS^{\bar p}_*(\td X;\Z)$
possesses a left $\Z[\pi]$-module structure induced by the geometric action of
$\pi$ on $\td X$.

\begin{notation}
\label{x3}
(i) 
Given a perversity $\bar p$ on $X$, the perversity on $\td X$ which takes a
stratum $S$ to ${\bar p}(p(S))$ will also be denoted by $\bar p$.

(ii)
We will write
$I\bar{S}_{\bar p}^*(\td X;\Z)$ for $\Hom_{\Z[\pi]}( IS^{\bar p}_*(\td X;\Z), 
\Z[\pi])$ and
$I\bar{H}_{\bar p}^*(\td X;\Z)$ for the cohomology groups of this complex.
\end{notation}

Now recall Definition \ref{x1} and suppose that $\bar p$, $\bar q$ and $\bar 
r$ are perversities with $D\bar r\geq D\bar p+D\bar q$ and that $X$ is either 
locally $\bar p$-torsion free or locally $\bar q$-torsion free. 

Let  $$\td d:IH^{\bar r}_*(X;\Z)\to  H_*( IS^{\bar p}_*(\td 
X;\Z)^t\otimes_{\Z[\pi]} IS^{\bar q}_*(\td X;\Z))$$
be the composition
\begin{align*}
IH^{\bar r}_*(X;\Z) &\xleftarrow{\cong} H_*(\Z\otimes_{\Z[\pi]} 
IS^{\bar r}_*(\td X;\Z))\\
&\xrightarrow{1\otimes d} H_*(\Z\otimes_{\Z[\pi]} IS^{Q_{\bar p,\bar 
q}}_*(\td X\times \td X;\Z))\\
&\xleftarrow{\cong}  H_*(\Z\otimes_{\Z[\pi]} (IS^{\bar p}_*(\td X;\Z)\otimes_\Z 
IS^{\bar q}_*(\td X;\Z)))\\
&\cong  H_*( IS^{\bar p}_*(\td X;\Z)^t\otimes_{\Z[\pi]} IS^{\bar q}_*(\td 
X;\Z)).
\end{align*}
Here $d$ is the diagonal map given by \cite[Proposition 4.2.1]{FM},
the first  isomorphism is given by
Proposition 
\cite[Proposition 6.1.3]{friedmanmcclure}, and 
the second isomorphism is given by 
Proposition \ref{P: tensor flat}  (this is why the torsion free assumption
is needed).  The third isomorphism is elementary.

Now we can define the cap product
\begin{equation}
\label{e21}
I\bar{H}^i_{\bar q}(\td X;\Z)\otimes IH^{\bar r}_j( X;\Z)\to
IH^{\bar p}_{j-i}(\td X,;\Z)
\end{equation}
by
\[
\alpha \smallfrown x=(1\otimes \alpha)\td d(x)
\]
(using the fact that $H_*( IS^{\bar p}_*(\td X;\Z)^t)$ is the same
$\Z$-module as $IH^{\bar p}_*(\td X;\Z)$).

Similarly, we get a
cap product 
\[
I\bar{H}^i_{\bar q}(\td X,\td A ;\Z)\otimes IH^{\bar r}_j( X,A\cup B;\Z)\to
IH^{\bar p}_{j-i}(\td X, \td B;\Z)
\]
when $A$ and $B$ are open subsets of $X$.

\begin{remark}
\label{x4}
If the covering $p:\td X\to X$ is trivial then $I\bar{S}_{\bar p}^*(\td
X;\Z)$ is canonically isomorphic to $\Hom(IS^{\bar p}_*(X;\Z),\Z[\pi])$
and hence $I\bar{H}_{\bar p}^*(\td
X;\Z)$ is canonically isomorphic to $IH^{\bar p}_*(X;\Z[\pi])$.
Under this isomorphism the cap product \eqref{e21} corresponds to the cap 
product
\[
I\bar{H}^i_{\bar q}( X;\Z[\pi])\otimes IH^{\bar r}_j( X;\Z)\to
IH^{\bar p}_{j-i}(X,;\Z[\pi]).
\]
\end{remark}

Let us write 
\[
(I\bar{H}_{\bar p}^i)_c(\td X;\Z)
\]
for
\[
\colim_{K} I\bar{H}_{\bar p}^i(\td X, \td X-\td K;\Z)
\]
where $K$ runs through the compact subsets of $X$
(the subscript $c$ stands for ``compact supports''; cf.\ \cite[Section
7.4]{Fr}).

As in \cite[Section 4]{friedmanmcclure}, we obtain a map
\[
{\mathcal{D}}: 
(I\bar{H}_{\bar p}^i)_c(\td X;\Z)
\to
IH^{D\bar p}_{n-i}(\td X;\Z).
\]

Now we can state the analogue of \cite[Theorem 4.1]{friedmanmcclure} for $\Z$ 
coefficients.

\begin{thm}[Universal Poincar\'e duality]\label{T: universal duality}
Let $X$ be a $\Z$-oriented PL pseudomanifold and let $p:\td X\to X$ be a 
regular covering of $X$.
Suppose that $X$ is locally
$\bar p$-torsion free (in particular this is the case if $X$ is an IP space and
$\bar p=\bar m$, by Remark \ref{x2}).
Then $\mathcal{D}$ is an isomorphism.
\end{thm}

For the proof of Theorem \ref{T: universal duality} we use \cite[Theorem 
5.1.4]{Fr}, with $F_*(U)=(I\bar{H}_{\bar p}^*)_c(\td U;\Z)$, 
$G_*(U)=IH^{D\bar p}_{n-i}(\td U;\Z)$, and $\Phi=\mathcal D$.
We need to verify the four conditions in \cite[Theorem 5.1.4]{Fr}. 

For condition 1, we first observe that the Mayer-Vietoris sequence for $F_*$ 
exists because if $U\subset V$ are open subsets of $X$ then the inclusion 
$IS^{\bar p}_*(\td U;F)\hookrightarrow IS^{\bar p}_*(\td V;F)$ is split as a 
map of $\Z[\pi]$-modules  (this follows from the proof of \cite[Proposition 
2.9]{FTrans}: use the construction in that proof with $X$ taken to be $V$ and 
the ordered open cover taken to be $(U,V)$). The rest of the verification of
condition 1 is the same as the proof of \cite[Lemma 7.4.8]{Fr}, except that 
we use the following instead of \cite[Proposition 7.3.59]{Fr}:

\begin{lemma}
\label{x5}
There exist chains
\[
\beta_{U-L}\in IS^{\bar p}_*(U-L)\otimes_{\Z[\pi]} IS^{\bar q}_*(U-L, U-K\cup L),
\]
\[
\beta_{U\cap V}\in IS^{\bar p}_*(U\cap V)\otimes_{\Z[\pi]} IS^{\bar q}_*(U\cap V, U\cap
V-K\cup L)
\]
and
\[
\beta_{V-K}\in IS^{\bar p}_*(V-K)\otimes_{\Z[\pi]} IS^{\bar q}_*(V-K,V-K\cup L)
\]
such that $\beta_{U-L}+\beta_{U\cap V}+\beta_{V-K}$ represents
${\td d}(\Gamma_{K\cup L})\in H_*(IS^{\bar p}_*(X)\otimes_{\Z[\pi]} IS^{\bar q}_*(X,X-K\cup
L)) $.
\end{lemma}

The proof of Lemma \ref{x5} is entirely parallel to the proof of \cite[Lemma
6.9]{FM}, except that we use \cite[Proposition 6.1.2]{friedmanmcclure} to 
show that the map
\begin{multline*}
\lambda:
H_*(\Z\otimes_{\Z[\pi]} \colim_{W\in \mathcal C}\, IS^{Q_{\bar p,\bar q}}_*(\td W\times 
\td W,\td W\times (\td W-\td K\cup \td L)))
\\
\to
H_*(\Z\otimes_{\Z[\pi]}IS^{Q_{\bar p,\bar q}}_*(\td Y,\td Y-(\td X\times (\td 
K\cup \td L))))
\end{multline*}
is an isomorphism.

Returning to the proof of Theorem \ref{T: universal duality},
the verification for condition 2 is the same as the corresponding part of the
proof of \cite[Theorem 8.2.4]{Fr}; the excision needed for part of the argument
is given by \cite[Proposition 6.1.2]{friedmanmcclure}.

As background for conditions 3 and 4, we note that \cite[Theorem 8.2.4]{Fr}
remains valid with $\Z[\pi]$ coefficients instead of $\Z$ coefficients, with
the same proof.

Now for condition 3, we observe that the covering $p$ pulls back trivially to
$\R^i \times cL$, so condition 3 follows from Remark \ref{x4} and the 
$\Z[\pi]$ version of \cite[Theorem 8.2.4]{Fr} (the hypothesis of condition 3
isn't needed for this purpose).

The verification of condtion 4 is the same as for condition 3.  This concludes
the proof of Theorem \ref{T: universal duality}.

Finally, we have the analog of 
\cite[Theorem 4.5]{friedmanmcclure}.

\begin{thm}[Universal Lefschetz Duality]\label{T: univ lef}
Let $X$ be an $n$-dimensional compact PL $\partial$-pseudomanifold
with a
$\Z$-orientation of
$X-\partial X$ and let $p:\td X\to X$ be a regular
covering of $X$. Suppose that $X$ is locally $\bar p$-torsion free (in
particular this is the case if $X$ is an IP space and $\bar p=\bar m$, by
Remark \ref{x2}).
Then
the cap product with $\Gamma_X$ gives isomorphisms
\[
I\bar{H}_{\bar p}^i(\td X,p^{-1}(\partial X);F)\to
IH^{D\bar p}_{n-i}(\td X;F)
\]
and
\[
I\bar{H}_{\bar p}^i(\td X;F)\to
IH^{D\bar p}_{n-i}(\td X,p^{-1}(\partial X);F).
\]
\end{thm}
The proof is the same as for \cite[Theorem 4.5]{friedmanmcclure}.

%\section{The bar construction for modules over $\mathbb{Z}[\pi_1 X]$}
%\label{bar}
%
%In this appendix we give some facts needed for Section #.  We will write $\RR$
%for $\mathbb{Z}[\pi_1 X]$.
%
%Let $M$ and $N$ be left modules over $\RR$.  Define $M\boxtimes N$ to be the
%left $\RR$ module that takes $x$ to $M_x\otimes N_x$, with the evident 
%action by morphisms of $\RR$.  Now define the bar construction $B(R,M)$

\section{Multiplicativity of the assembly map}
\label{am}

This appendix gives the proof of

\begin{prop}
\label{p7}
Let $\mathbf F$ be a homotopy invariant functor from spaces to spectra, and
suppose that there is a natural transformation 
\[
\mu: \bF(X)\wedge \bF(Y)\to \bF(X\times Y).
\]
Then the diagram
\[
\xymatrix{
(X_+\wedge\bF(*))\wedge
(Y_+\wedge\bF(*))
\ar[d]_{\alpha\wedge \alpha}
\ar[r]^-\mu
&
(X\times Y)_+\wedge \bF(*)
\ar[d]_\alpha
\\
\bF(X)\wedge\bF(Y)
\ar[r]^-\mu
&
\bF(X\times Y)
}
\]
commutes up to homotopy, where $\alpha$ denotes the assembly map.
\end{prop}

Let us recall the definition of the assembly map from \cite[page 334]{wwa}.
The construction uses homotopy colimits (see \cite[Section 1]{wwa} for a brief
description of the homotopy colimit construction).
For a space $X$, let $\CC_X$ (which is denoted simp$(X)$ in \cite{wwa}) be the
category whose objects are maps $\Delta^n\to X$, and whose morphisms are
commutative triangles 
\[
\xymatrix{
\Delta^m
\ar[rr]^-{f_*}
\ar[rd]
&&
\Delta^n
\ar[ld]
\\
&X&
}
\]
where $f_*$ is the map induced by a monotone injection\footnote{It is not clear
why \cite{wwa} does not use all monotone maps in this definition.  The reader 
can check that
using all monotone maps would give the same assembly map, and would allow us
to work with simplicial rather than semisimplicial sets in what follows.}
from $\{0,\ldots,m\}$ to $\{0,\ldots,n\}$.  
There is a natural equivalence in the homotopy category
\[
\lambda: X\to \hocolim_{\CC_X}\, *
\]
(see below)
where $*$ denotes the functor which takes all objects to a point.
Let $D$ be the functor from $\CC_X$
to spaces which takes $\Delta^n\to X$ to $\Delta^n$.
The assembly map is the
following composite in the homotopy category of spectra (where $\wedge$ is 
the derived smash product)
\begin{multline*}
\label{e20}
X_+\wedge \bF(*)
\xrightarrow{\lambda\wedge 1}
(\hocolim_{\CC_X}\, *)_+\wedge \bF(*)
\cong
\hocolim_{\CC_X}\, \bF(*)
\\
\xleftarrow{\simeq}
\hocolim_{\CC_X}\, \bF\circ D
\to
F(X).
\end{multline*}

Our first task is to give an explicit description of $\lambda$
(this was left as an exercise for the reader in \cite{wwa}).  
We need a lemma.

\begin{lemma}
The map 
\[
\hocolim_{\CC_X}\, D
\to
\colim_{\CC_X}\, D
\]
is a weak equivalence.
\end{lemma}

\begin{proof}
The category $\CC_X$ is a Reedy category (\cite[Def 15.1.2]{H}; the subcategory
$\stackrel{\longrightarrow}{\CC_X}$ is equal to $\CC_X$ and the subcategory
$\stackrel{\longleftarrow}{\CC_X}$ has only identity morphisms) which has
fibrant constants (\cite[Definition 15.10.1]{H}) by the proof of
\cite[Proposition 15.10.4(1)]{H}.  The functor $D$ is Reedy cofibrant
(\cite[Definition 15.3.3(2)]{H}) and the result follows by \cite[Theorem
19.9.1(1)]{H}.
\end{proof}

Now let $SX$ be the semisimplicial set whose $n$-simplices are the maps 
$\Delta^n\to X$.  Then $\lambda$ can be chosen\footnote{To know that this is a
correct choice, the paragraph after Observation 1.3 in \cite{wwa} says that 
one simply has to show that 
it gives an assembly map with the
properties in \cite[Theorem 1.1]{wwa}, and this follows from the fact that this
choice of $\lambda$ is natural
with respect to $X$.} 
to be the composite
\[
X
\xleftarrow{\simeq}
|SX|
=
\colim_{\CC_X}\, D
\xleftarrow{\simeq}
\hocolim_{\CC_X}\, D
\xrightarrow{\simeq}
\hocolim_{\CC_X}\, *.
\]

Our next lemma
gives a multiplicative property of $\lambda$.
Let 
\[
d:\CC_{X\times Y}\to
\CC_X\times\CC_Y
\]
be the functor which takes $f:\Delta^n\to X\times Y$ to the pair $(p_1\circ
f,p_2\circ f)$, where $p_1$ and $p_2$ are the projections.  

\begin{lemma}
\label{l7}
{\rm (i)}
The diagram
\[
\xymatrix{
X\times Y
\ar[r]^-{\lambda\times\lambda}
\ar[d]_\lambda
&
({\displaystyle\hocolim_{\CC_X}}\, *)
\times
{\displaystyle(\hocolim_{\CC_Y}}\, *)
\\
{\displaystyle\hocolim_{\CC_{X\times Y}}}\, *
\ar[r]^-\delta
&
{\displaystyle\hocolim_{\CC_X\times\CC_Y}}\, *
\ar[u]^\cong
}
\]
commutes, where $\delta$ is induced by $d$ and the verical arrow is induced by
the projections $\CC_{X\times Y}\to\CC_X$ and $\CC_{X\times Y}\to \CC_Y$.

{\rm (ii)}
The map $\delta$ in part (i) is a weak equivalence.
\end{lemma}

The proof of part (i) is left to the reader.  Part (ii) follows from (i) and
the fact that $\lambda$ is a weak equivalence.

Next we need some notation.  For a space $Z$ let 
$\Lambda$ denote the composite
\[
Z_+\wedge \bF(*)
\xrightarrow{\lambda\wedge 1}
(\hocolim_{\CC_Z}\, *)_+ \wedge \bF(*)
\cong
\hocolim_{\CC_Z}\, \bF(*).
\]
Let 
\[
\delta':\hocolim_{\CC_{X\times Y}}\, \bF(*)
\to
\hocolim_{\CC_X\times\CC_Y}\,\bF(*)
\]
be the map induced by $d$; this is a weak equivalence by Lemma \ref{l7}(ii).
Let $E$ be the functor from
$\CC_X\times\CC_Y$ to spaces that takes $(\Delta^m\to X,\Delta^n\to Y)$ to 
$\Delta^m\times \Delta^n$, and let 
\[
\delta'': \hocolim_{\CC_{X\times Y}}\, \bF\circ D
\to
\hocolim_{\CC_X\times\CC_Y}\,\bF\circ E
\]
be the map induced by $d$ and the diagonal maps $\Delta^n\to\Delta^n\times
\Delta^n$.

To complete the proof of Proposition \ref{p7}, we need only observe that the
following diagram is commutative:
\[
\xymatrix{
X_+\wedge \bF(*)\wedge Y_+\wedge \bF(*)
\ar[rr]
\ar[d]_{\Lambda\wedge\Lambda} 
&&
(X\times Y)_+\wedge \bF(*)
\ar[d]_{\Lambda}
\\
{\displaystyle\hocolim_{\CC_X}}\, \bF(*)
\wedge
{\displaystyle\hocolim_{\CC_Y}}\, \bF(*)
\ar[r]
&
{\displaystyle\hocolim_{\CC_X\times\CC_Y}}\, \bF(*)
&
{\displaystyle\hocolim_{\CC_{X\times Y}}}\, \bF(*)
\ar[l]_{\delta'}
\\
{\displaystyle\hocolim_{\CC_X}}\, \bF\circ D
\wedge
{\displaystyle\hocolim_{\CC_Y}}\, \bF\circ D
\ar[r]
\ar[u]^\simeq
\ar[d]
&
{\displaystyle\hocolim_{\CC_X\times\CC_Y}}\, \bF\circ E
\ar[u]^\simeq
\ar[dr]
&
{\displaystyle\hocolim_{\CC_{X\times Y}}}\, \bF\circ D
\ar[l]_{\delta''}
\ar[u]^\simeq
\ar[d]
\\
\bF(X)\wedge\bF(Y)
\ar[rr]
&&
\bF(X\times Y)
}
\]

\section{Corrections to some signs in \cite{LM}}
\label{aa1}

This appendix is only relevant for the proof of Lemma  \ref{L2}.
All references are to \cite{LM}.

In Definition 15.4(i), the $i$-th face map 
$\ad^k(\Delta^n) \to \ad^k(\Delta^{n-1})$ should be $(-1)^i$ times the composite
with the map induced by the $i$-th coface map $\Delta^{n-1}\to\Delta^n$.
There should also be an analogous sign in Definition 17.2. 

In order for Proposition 17.8 to be true with this sign, the paragraph that 
comes after Lemma 15.7 needs to be replaced by the following.

``Next observe that for each $n$ there is an isomorphism of $\Z$-graded
categories
\[
\nu:
\Cell(\Delta^{n+1},\partial_0\Delta^{n+1}\cup \{0\})
\to
\Cell(\partial_0\Delta^{n+1})
\]
which lowers degrees by 1, 
defined as follows: a simplex $\sigma$ of $\Delta^{n+1}$ which is not in
$\partial_0\Delta^{n+1}\cup \{0\}$ contains the vertex $0$.  
Let
$\nu$ take $\sigma$ (with its canonical orientation) to the simplex of
$\partial_0\Delta^n$ spanned by the vertices of $\sigma$ other than $0$ (with
its canonical orientation).  
Let 
\[
\theta:\Cell(\Delta^{n+1},\partial_0\Delta^{n+1}\cup \{0\})
\to
\Cell(\Delta^n)
\]
be the composite of $\eta$ with the isomorphism induced by the face map
$\Delta^n\to \partial_0\Delta^{n+1}$.
$\theta$ is incidence-compatible,  so
by part (e) of Definition 3.10 it induces a bijection
\[
\theta^*:
\ad^k(\Delta^n)
\to
\ad^{k+1}(\Delta^{n+1},\partial_0\Delta^{n+1}\cup \{0\}).''
\]

This change leads to corresponding changes in Section 16 and Lemma 17.11, 
which we leave to the interested reader.

\providecommand{\bysame}{\leavevmode\hbox to3em{\hrulefill}\thinspace}
\providecommand{\MR}{\relax\ifhmode\unskip\space\fi MR }
\providecommand{\MRhref}[2]{%
  \href{http://www.ams.org/mathscinet-getitem?mr=#1}{#2}
}
\providecommand{\href}[2]{#2}

\end{document}